\documentclass[reqno]{amsart}

\usepackage{amsmath,amsfonts,amssymb,amsthm}
\usepackage[dvipsnames]{xcolor}
\usepackage{mhsetup,bm}
\usepackage{rotating}
\usepackage{mathtools}
\usepackage[pdfpagelabels]{hyperref}
\usepackage{enumerate}
\usepackage{float}
\usepackage{tikz}
\usepackage{textcomp}
\usetikzlibrary{shapes}
\usepackage{enumitem}   
\usetikzlibrary{positioning}
\usepackage{textcomp}

\usepackage{marginnote}
\usepackage{mathrsfs}
\usepackage{mathabx}

\newtheorem{theorem}{Theorem}
\newtheorem{definition}[theorem]{Definition}
\newtheorem{proposition}[theorem]{Proposition}

\newtheorem{lemma}[theorem]{Lemma}
\newtheorem{hyp}[theorem]{Hypothesis}
\newtheorem{cor}[theorem]{Corollary}

\DeclareMathOperator{\Tr}{Tr}
\newtheorem{remark}[theorem]{Remark}

\def\my_c{c_\infty}

\newcommand{\mynewtheorem}[2]{
        \newaliascnt{#1}{dummy}
        \newtheorem{#1}[#1]{#2}
        \aliascntresetthe{#1}
        \expandafter\def\csname #1autorefname\endcsname{#2}
}

\newcommand{\be}{\begin{equation}}
        \newcommand{\ee}{\end{equation}}
\newcommand{\bde}{\begin{displaymath}}
        \newcommand{\ede}{\end{displaymath}}
\newcommand{\beq}{\begin{eqnarray*}}
        \newcommand{\eeq}{\end{eqnarray*}}
\newcommand{\beqa}{\begin{eqnarray}}
        \newcommand{\eeqa}{\end{eqnarray}}
\newcommand{\bel }{\left\{\begin{array}{ll}}
        \newcommand{\eel}{\cr \end{array} \right.}

\newcommand{\dcb}{\begin{array}{lll}}
        \newcommand{\dce}{\end{array}}
\newcommand{\ebe}{\begin{enumerate}\setlength{\baselineskip}{13pt}\setlength{\parskip}{0pt}}
        \newcommand{\dbe}{\end{enumerate}}

\newcommand{\leftrightharpoonup}{%
        \mathrel{\mathpalette\lrhup\relax}%
}
\newcommand{\lrhup}[2]{%
        \ooalign{$#1\leftharpoonup$\cr$#1\rightharpoonup$\cr}%
}
\newcommand\rharp[1]{\mathstrut\mkern2.5mu#1\mkern-11mu\raise1.2ex%
        \hbox{$\scriptscriptstyle\rightharpoonup$}}
\newcommand\lharp[1]{\mathstrut\mkern2.5mu#1\mkern-11mu\raise1.2ex%
        \hbox{$\scriptscriptstyle\leftharpoonup$}}
\newcommand\lrharp[1]{\mathstrut\mkern2.5mu#1\mkern-11mu\raise1.2ex%
        \hbox{$\scriptscriptstyle\leftrightharpoonup$}}

\def\Tr{{{\rm Tr}}}

\def\0{{\mathbf{0}}}

\begin{document}

        \title[Derivative of killed diffusion: Half space]{Probabilistic representation of the gradient of a killed diffusion semigroup: The half-space case}
                \author{Dan Crisan and Arturo Kohatsu-Higa }
        \address{\noindent Dan Crisan:
                Department of Mathematics, Imperial College London, 180 Queen's Gate, London
                SW7 2AZ, UK\newline Arturo Kohatsu-Higa:
                Department of Mathematical Sciences, Ritsumeikan University 1-1-1
                Nojihigashi, Kusatsu, Shiga, 525-8577, Japan.}
        \begin{abstract} 
                We introduce a probabilistic representation of the derivative
                of the semigroup associated to a multidimensional killed diffusion process defined on the half-space. The semigroup derivative is expressed as a functional of a process that is normally reflected when it hits the hyperplane. The representation of the derivative also involves a matrix-valued process which replaces the Jacobian of the underlying process that appears in the traditional pathwise derivative of a classical diffusion. The components of this matrix-valued process jump to zero except for those on the first row every time the reflected process touches the boundary. The results in this paper extend those in \cite{CK}, where the one-dimensional case was studied but the jump behavior did not appear.\\

        \end{abstract}
        
\maketitle

\section{Introduction}

%
%
%
%
In this paper we study a uniformly elliptic multi-dimensional diffusion process of the type \begin{align}
        \label{eq:2a}
        X_t=&x+\int_0^tb(X_s)ds+\int_0^t\sigma^{k} (X_s)dW^{k} _s,\ \ t<\tau ,\ \ \ x\in H^d_L =(L,\infty)\times \mathbb{R}^{d-1},
\end{align}
where $ W=(W^1,...,W^d)^{\top}  $ is a $ d $-dimensional Wiener process and $\tau\equiv \tau(X):=\inf\{t>0,X^1_t(x)=L\}$. We will assume that the coefficients $ b,\sigma^{k} :\mathbb{R}^d\to\mathbb{R}^d $, $ k=1,...,d $ are twice differentiable with bounded derivatives.\footnote{We assume Einstein summation over repeated indexes.} In this article we seek a probabilistic representation for  $ {\mathbf{D}} _x\mathbb{E}\left[f(X_{T\wedge\tau} )\right ]$, where $ f(y)=0 $ for $ y\in\partial H^d_L $.

The case $ d=1 $
has been studied in \cite{CK}. As in that situation the probabilistic representation for $ {\mathbf{D}} _x\mathbb{E}\left[f(X_{T\wedge\tau} )\right ]$ involves the associated solution of a reflected stochastic differential equation (sde) with coefficients $ b $ and $ \sigma $. To build the reader's   understanding as to why we separated the cases $ d=1 $ and $ d>1 $, let us start with the following simple 2D example.
Consider the domain\footnote{See Section \ref{sec:not} for details of the notation on vector calculus.}  $ H_0^2=(0,\infty)\times \mathbb{R} $ and for $ a,b,c,\varrho\in\mathbb{R} $ (in other words $ L=0,d=2 $)
$$ X_t:=(X^1_t,X^2_t)=(x_1+aW^1_t,x_2+a^{-1}\varrho \int_0^tX_s^1dX^1_s+cW^2_t). 
$$
Then, by using a slight extension of the results in one dimension 
(see \cite{CK}), one can obtain that 
\begin{align*}
        {\mathbf{D}}_x\mathbb{E}
        \left[f(X_{T\wedge\tau})\right]=&
        \mathbb{E}\left[{\mathbf{D}} f(Y_{T})\mathcal{E}_T\right]\\
        \mathcal{E}_t:=&\left(\begin{matrix}
                        1&0\\
                        a^{-1}\varrho (Y^1_t-x_11_{\tau(Y)>  t})&1_{\tau(Y)> t}
                \end{matrix}
                \right),\ t\in [0,T].
\end{align*}
Here $ Y =(Y^1,Y^2)$ is defined as 
\begin{align*}
        (Y^1_t,Y^2_t):=&(x_1+aW^1_t+B_t, x_2+a^{-1}\varrho \int_0^tY^1_sdY^1_s+cW^2_t)\\
        B_t=&\int_0^t1_{(Y^1_s=L)}d|B|_s.
\end{align*}
That is, the first component $Y^1 $ is indeed a one dimensional reflected Brownian motion and $ B $ represents the corresponding local time process\footnote{Here and elsewhere, $ |B|_t $ denotes the total variation of the process $ B $ in the interval $ [0,t] $.} and $ \tau(Y):=\inf\{t>0;Y_t\in\partial H^d_L\} = \inf\{t>0;Y^1_t=0\} $. As advertised in the abstract, the process $\mathcal{E}$ replaces the Jacobian $\mathbf{D}X_T$ in the classical representation for the diffusion derivative:
\begin{align*}
        {\mathbf{D}}_x\mathbb{E}
        \left[f(X_{T})\right]=&
        \mathbb{E}\left[{\mathbf{D}} f(X_T)\mathbf{D}X_T\right]\\
      \mathbf{D}X_T:=&\left(\begin{matrix}
                        1&0\\
                        \varrho W_T^1&1
                \end{matrix}
                \right).
\end{align*}
Note that both derivatives $ \mathbf{D}X_T $ and $ \mathcal{E}_T $ coincide if $ \tau(Y)>T $. 
Inspired by this classical result, we call $ \mathcal{E} $, the derivative process. If $ \tau(Y)\leq T $, that is if $Y$ hits the boundary in the interval $ [0,T] $, 
the second row of  $\mathcal{E}$   jumps to 0 at that time. 

A slightly different example is the case when $ \sigma  $ is a non-constant diagonal matrix (i.e. $ \varrho=0 $ in the previous example). That is, we assume that  $ \sigma^k:\mathbb{R}^2\to\mathbb{R}^2$, $ k=1,2 $ is a sufficiently smooth function satisfying $ \sigma^k\cdot \mathbf{e}^j=0 $ for $ j\neq k $. Here, $ \{\mathbf{e}^j,j=1,2\} $ denotes the standard basis of $ \mathbb{R}^2 $ (vectors are considered as column vectors). We consider the equation:
\begin{align*}
         X_t=&x+\int_0^t\sigma^{k} (X_s)dW^{k} _s.
\end{align*}
In this case, our results imply that for $ f\in C^1_b(\mathbb{R}^2,\mathbb{R} )$ with $ f(0,x_2)=0 $ for all $ x_2\in\mathbb{R} $,
\begin{align}
        \label{eq:1}
        {\mathbf{D}}_x\mathbb{E}\left[f(X_{T\wedge \tau})\right] =\mathbb{E}\left[{\mathbf{D}} f(Y_T)\mathcal{E}_T\right.],
\end{align}
where the process $(Y,\mathcal{E})$ is the unique solution of the sde 
    \begin{align*}
        Y_t=&x+\int_0^t\sigma^{k} (Y_s)dW^{k} _s+\mathbf{B}_t\in H^d_L,\\
        \mathbf{B}_t=&\int_0^t\mathbf{e}^1d|\mathbf{B}|_s\notag\\
        (\mathbf{e}^1)^{\top}\mathcal{E}_t=&(\mathbf{e}^1)^{\top}+\int_0^t(\mathbf{e}^1)^{\top} \mathbf{D}\sigma ^k(Y_s)\mathcal{E}_sdW^{k} _s.\\
        (\mathbf{e}^2)^{\top}\mathcal{E}_t=&
        \begin{cases}
                (\mathbf{e}^2)^{\top}+\int_0^t (\mathbf{e}^2)^{\top} \mathbf{D} \sigma^k(Y_s)\mathcal{E}_sdW^{k} _s, & t< \tau(Y),\\
               \int_{\rho_t}^t (\mathbf{e}^2)^{\top} \mathbf{D} \sigma^k(Y_s)\mathcal{E}_sdW^{k} _s ,&  t\geq \tau(Y).
        \end{cases}
\end{align*}
Here $ \mathbf{D}\sigma^k $ stands for the Jacobian of $ \sigma^k $ and $ (\mathbf{e}^j)^{\top} $, $j=1,2$, denotes the transpose of the vector $ \mathbf{e}^j $, $j=1,2$, respectively. Finally $ \rho_t:=\sup\{s<t,Y_s^1=L\}\vee 0 $ and the  above stochastic integral is interpreted as 
\begin{align*}
        \int_{\rho_t}^t\mathbf{D} \sigma^k (Y_s)\mathcal{E}_sdW^{k} _s:= \int_{u}^t\mathbf{D} \sigma^k (Y_s)\mathcal{E}_sdW^{k} _s\Big|_{u=\rho_t}.
\end{align*} 

As in the first example, the process $ Y $ is normally reflected when it hits $\partial H^d_L$ and   $ \mathbf{B} $ is parallel to $ \mathbf{e}^1 $ and it represents the reflecting local time for $ Y^1 $.   In fact, $ \mathbf{B}_t=B_t\mathbf{e}^1 $ where $ B $ is a real continuous increasing adapted process that only increases when $ Y $ is at the boundary (see formula (\ref{eq:B}) below).
Just as in the previous example, $ (\mathbf{e}^2)^{\top}\mathcal{E}_t $ becomes zero every time the process $ Y $ touches the boundary. 


The general case when the drift is non-zero and the diffusion coefficient is non-diagonal exhibits additional  characteristics at the boundary which do not appear in the one dimensional case. We will assume that the process $X$ satisfies \eqref{eq:2a}
and the following holds:

\begin{hyp}
        \label{hyp:12}
        The functions $ b,\sigma^k \in C^2_b(\mathbb{R}^d)$ are bounded\footnote{In other words, the functions are well defined as extensions over $ H^d_L $.}, $ a=(a^{ij})=\sigma\sigma^{\top}  $ is uniformly elliptic and it satisfies $   \sigma^{1\ell }(y)=\sigma^{\ell 1}(y)=0$ for $ y\in \partial H^d_L $ and $ \ell\neq 1 $. This hypothesis implies the following condition for $ \ell\neq 1 $, $ y\in\mathbb{R}^d $
        \begin{align*}
                { a^{\ell 1}(y) } =\hat{a}^\ell (y)(y^1-L),
        \end{align*}
        where $  \hat{a}^\ell\in C_b^1(\mathbb{R})$ and is bounded. With this notation, we define $ \pi(y):=\sum_{\ell=2}^d \hat{a}^\ell (y)\mathbf{e}^\ell(\mathbf{e}^1)^{\top}$.
        
        Without loss of generality we also assume that $ \sigma(x) $ is invertible for all $ x\in\mathbb{R}^d $. 
\end{hyp}

 The main result of this article is the following probabilistic representation for $ x\in H^d_L $
\begin{align}
        \label{eq:main}
        {\mathbf{D}}_x\mathbb{E}[f(X_T)1_{(\tau>T)}]=&
        \mathbb{E}[{\mathbf{D}} f(Y_T){\Psi}_T].
\end{align}
Here, $ (Y,\Psi)  $ solves\footnote{We remark that in the one dimensional case studied in \cite{CK}, the notation $ B $ corresponds here to $ \mathbf{B} $.}
\begin{align}
          Y_t=&x+\int_0^tb(Y_s)ds+\int_0^t\sigma^k(Y_s)dW^{k}_s
        +\mathbf{B}_t\in H^d_L, \label{eq:Y}\\
        \mathbf{B}_t=&\mathbf{e}^1\int_0^t1_{(Y^1_s=L)}d|\mathbf{B}|_s,\label{eq:B}\\
        \Psi_t=&I1_{(\rho_t=-\infty)}+ 1_{(\rho_t\in [0,t])}\mathbf{e}^1(\mathbf{e}^1)^{\top} 
        \Psi_{\rho_t-}+\int_{\rho_t\vee 0}^t\left(\mathbf{D} b(Y_s)ds+\mathbf{D}\sigma ^k(Y_s)dW^{k} _s\right)
        \Psi_s,\label{eq:Psi}  
        \\
        (\mathbf{e}^1)^\top\Psi_t=&(\mathbf{e}^1)^\top+\int_0^t(\mathbf{e}^1)^\top d\alpha_s
        \Psi_s,\label{eq:sc}
        \\
         d\alpha_s:=&\mathbf{D} b(Y_s)ds+\mathbf{D}\sigma ^k(Y_s)dW^{k} _s+2{b^1}(Y_s)IdB_s,\label{eq:alpha}\\
        dB_s:=&a^{11}(Y_s)^{-1}\mathbf{e}^1\cdot d\mathbf{B}_s, \label{bb}\\
                {\rho}_t:=&\sup\{s\leq t;Y_s\in\partial H_L\}.\label{rho_t}
\end{align}

In order to prove this result, we will first show the following intermediate result:
\begin{align}
               {\mathbf{D}}_x\mathbb{E}[f(X_T)1_{(\tau>T)}]=&
        \mathbb{E}[{\mathbf{D}} f(Y_T){\xi}_T]+
        \mathbb{E}\left[f(Y_T)
      {  \left(\frac{b^1}{{a^{11}}}\right)(Y_{\rho_T\vee 0})(\mathbf{e}^1)^{\top}  } \xi_{\rho_T\vee 0}     1_{(\tau\leq T)}
        \right].
        \label{eq:diff}
\end{align}
 Here,
\begin{align}
                \xi_t=&I1_{(\rho_t=-\infty)}+1_{(\rho_t\in [0,t])}\mathbf{e}^1(\mathbf{e}^1)^{\top} 
        \xi_{\rho_t-}+\int_{\rho_t\vee 0}^t\left(d\alpha_s-{b^1}({Y}_s)dB_s\right)
        \xi_s.
        \label{eq:E}   \\
      (\mathbf{e}^1)^\top\xi_t=&(\mathbf{e}^1)^\top+\int_0^t(\mathbf{e}^1)^\top \left(d\alpha_s-{b^1}({Y}_s)dB_s\right)
      \Psi_s,\nonumber
        \end{align}

Let us explain the various terms appearing in the above formulas:  

\begin{itemize}
\item Following equations \eqref{eq:Y} and \eqref{eq:B}, $ Y $ is a process with normal reflection at the boundary of the domain  $H^d_L$. This is a consequence of Hypothesis \ref{hyp:12}.
\item  The  first  term in \eqref{eq:Psi} ensures that $\Psi_t$ is initialized from the identity matrix $I$ at time $t=0$. Let us denote by $\Psi_t^\ell=(\mathbf{e}^\ell)^{\top}\Psi_t$, $\ell=1,...,d$ the row vectors of the matrix $\Psi_t$. That is \[
\Psi _{t}=\left( 
\begin{array}{l}
\Psi _{t}^{1} \\ 
\Psi _{t}^{2} \\ 
. \\ 
. \\ 
\Psi _{t}^{d}%
\end{array}%
\right) .
\]
The second term in \eqref{eq:Psi} tells us that all but the first row vectors $\Psi_t^\ell$, $\ell\neq 1 $ jump to $0$ every time $Y$ hits the boundary of the domain.  Afterwards,  $\Psi_t$ evolves according to equation (\ref{eq:Psi}) until the next time $Y$ hits the boundary when, again, all but the first  but the first row vectors $\Psi_t^\ell$, $\ell=1,...,d$ become $0$ and so on.  The fact that the row vectors $\Psi_t^l$, $l\neq 1 $ jump to $0$ every time $Y$ hits the boundary of the domain can be interpreted as a loss of memory of the past in those respective dimensions whenever the process $Y$ hits the boundary. The same behavior holds true for $ \mathcal{E}_t $ as defined in (\ref{eq:E}).  Again this is a consequence of Hypothesis \ref{hyp:12}.
\item The dynamics of the ``derivative'' processes $ \Psi_t $ in the particular case when  $ b=0 $ and $a$ is a constant  matrix are as follows: $ \Psi_t $ is the identity matrix $ I $ as long as $ \tau(Y)\leq t $ after that the derivative becomes $ \mathbf{e}^{\, 1}(\mathbf{e}^{\, 1})^{\top}  $.  
\item Note that due to Hypothesis \ref{hyp:12}, we have that $ (a^{11}b^\top a^{-1}\mathbf{e}^1+{b^1}) (Y_s)dB_s =2b^1(Y_s)dB_s$ which appears in \eqref{eq:alpha}. This term appears due to the interaction between the drift and the boundary. Heuristically speaking, the first term is a by-product of the application of Girsanov's change of measure and the second is due to the differentiation of the killing measure in one dimension. 
\item The condition \eqref{eq:sc} is a property satisfied by $ \Psi $ which expresses how local time is accumulated by this process only when multiplied by the vector $ (\mathbf{e}^1)^\top $.


\end{itemize}

In related research, some of the above effects have been identified in the case of reflection with constant diffusion coefficient albeit with different definitions, reasonings and proofs (see \cite{Andres}, \cite{LK}, \cite{DZ}, \cite{P1},\cite{P2}). It is noteworthy that, for reflected processes,
the above rich behaviour of the derivative process is not yet known. 

The proof of our result follows from the iteration of the push forward derivative formula in Lemma \ref{lem:2.1H} for an approximation of the killed process and Theorem \ref{th:5} which states that limits of the iterated formula can be obtained.

Our arguments are based on the derivative of approximant semigroups because the laws of the killed diffusion and the reflected diffusion are not  absolutely continuous with respect to each other. Instead, we will use the theory of weak convergence for stochastic integrals as in \cite{KP}, and \cite{EthierKurtz}. Two new characteristics arise in comparison with the one dimensional case result.

First, we note that the characterization of the limit process requires arguments that are new in the multi-dimensional case as not only that we need to study the behavior of the derivative process $ (\mathbf{e}^1)^{\top}\Psi $ (which is somewhat equivalent to the one-dimensional case, except that in general the directions of reflection are covariance generated), but also the derivative  in other directions $ (\mathbf{e}^\ell)^{\top}\Psi $, $ \ell\neq 1 $ (which do not appear in the one-dimensional case) which jump to zero every time the underlying reflected process touches the boundary.

Second, this study is also a stepping stone towards the more general case of a smooth domain which contains further properties which do not appear in the present case. In particular note that the hypothesis $ a^{\ell 1}(y)=0 $ for  $ y\in\partial H^d_L $ and for $\ell >1 $ is crucial in the present representation.  In the case this is not assumed we believe that in general $ Y $ is the solution of an obliquely reflected equation and the directions  of the projection of $ \mathcal{E} $  when $ Y $ touches the boundary changes due to the general covariance structure at the boundary. This will be discussed in future research.

{\it Technical comparison with the one dimensional case:}
In several instances, the proofs in the multi-dimensional case are build along the same lines as in the one-dimensional case. If this is the case, the corresponding  proofs will be omitted, however we will note the main changes required to adapt the arguments to the multi-dimensional case.
Let us stress the three main differences between the results in the one-dimensional case obtained in \cite{CK} and the ones presented here: 

\begin{itemize}

\item In general, the effect of a non-diagonal diffusion matrix on the behavior of the derivative process is to introduce a component covariance which influences the reflection. In fact,  Hypothesis \ref{hyp:12} implies that $ Y $ in \eqref{eq:Y} is normally reflected  at the boundary. 

\item Hypothesis \ref{hyp:12} implies that all rows except the first row of $ \Psi $ becomes zero when $ Y $ touches the boundary. We expect that these two first properties change without Hypothesis \ref{hyp:12}. Therefore this hypothesis influences both  $ Y $ and $ \Psi $.


\item There is lack of commutativity in matrix operations that have to be used in order to carry out the analysis in the paper. Obviously this technical problem is 
absent in the proofs for the one-dimensional case. This is a difficulty that will appear throughout the article, (but particularly in the arguments incorporated in the appendices) and is one of the reasons why we decided to separate the one dimensional case in \cite{CK}. In fact, section \ref{sec:7.6} requires a different argument in comparison with the one dimensional case. Other technical issues are discussed at the end of next section.
\end{itemize}


The representation \eqref{eq:main} is a stepping stone  to provide a natural physical interpretation of the boundary behaviour of the derivative of the semigroup.
We exploit this in a sequel to the paper (see \cite{CK3} for details), where we study the semigroup associated with a diffusion killed when it reaches the boundary of a domain which is perturbed, depending on a parameter $\varepsilon$ that measures the size of the perturbation.
In \cite{CK3}, we cover several cases:

\begin{itemize}
	
	\item Firstly, the initial domain is a half-space with boundary perturbed by a given smooth function. In this case we obtain an expression for the derivative of the semigroup with respect to the perturbation parameter. We also analyse (numerically as well as theoretically) its dependence with respect to various parameters: the drift term, the size of the time interval, the sign of the perturbation function, etc.  
	
	\item In a second case, 
%
%
	we study the case when the boundary is an	$\alpha$-H\"older continuous path (for example a Brownian path). Also in this case we obtain an expression for the derivative of the semigroup with respect to the perturbation parameter by using Malliavin integration by parts formula under the assumption that the diffusion coefficient is the identity matrix. We analyze its behaviour as a function of the distance to the boundary. \\[2mm]
	
\end{itemize}

Heuristically, the results in \cite{CK3} can be interpreted as a change in the concentration of  
particulates in a domain and, as a result, their  accumulation at the boundary of the domain. 
We think of the solution of  
(\ref{eq:2a}) as a model for the evolution of particulates
that drift through a given domain until they hit  the boundary. When they hit the boundary, they stick to it. The boundary then gets deformed, which in turns influences the evolution of the particulates. The question we want to answer is how does the boundary perturbation influences the rate of the accumulation of the sediments. The boundary can be deformed in different ways as explained above. \\[1mm]                   


For further information on these and other applications, heuristic explanations on the arguments used in this research we refer the reader to our YouTube video presentation \cite{YT} on the one-dimensional case. Many comments and references to related subjects that were done in the introduction of \cite{CK} will not be repeated here and we refer the reader to the discussions in \cite{CK}.

The geometric considerations presented here seem to have some dual connections with previous results about the differentiation of reflected processes (see \cite{Andres}, \cite{Burdzy}, \cite{DZ} and \cite{LK}). Mathematical results on this matter are not yet available.

\section{Notation}
        \label{sec:not}
        Much of the notation to be used is the natural extension of the one used in the one dimensional case. As such, we refer the reader to  \cite{CK} for further references on the choice of notation used in this article. 
 Let us start describing the notation used throughout the present article.
        
         The Dirac delta distribution at the point $L $ is denoted by $%
         \delta_L(x) $. Without raising any confusion, because of the context, we let $ \delta_{jk} $ denote the Kroenecker delta for $ j,k\in\{1,...,d\} $.

        We will also use the time indexes which were used in the one-dimensional case and we will also have upper indexes which relate to components of matrices.
        We assume the Einstein summation convention over doubly appearing indexes in a product which are related to the components of matrices and/or vectors such as in \eqref{eq:2a}. That is, if the same (matrix/vector) index symbol appears twice in the same formula  in an expression we automatically mean that a summation symbol has been suppressed for simplicity in the notation. Note that this rule does not apply to time indexes which may appear many times in the same formula and for which we systematically use the index symbol $ i $.

        Vectors are always understood as column vectors. The transpose of a matrix $ A $ is denoted by $ A^{\top}   $ and components of vectors are denoted by upper indexes. Recall that lower indices are used for time variables. When denoting powers of a one dimensional stochastic process $ X_t $, we will use $ (X_t)^3 $ instead of $ X_t^3 $ where the latter denotes the third component of the multi-dimensional stochastic process $ X_t $. One exception to this rule is the case of basis vectors for which we will use $ \{\mathbf{e}^1,...,\mathbf{e}^d\} $ to denote the standard basis of $ \mathbb{R}^d $. We also denote  $ x^j= x\cdot \mathbf{e}^j$ . 
        

        In the particular case of the diffusion coefficient matrix which is denoted by $ \sigma $, we will use the notation $ \sigma^{k}  $ to denote the $ k $-th column vector of the matrix $ \sigma $ which will be associated to the $ k $-th noise process.  In the case that we want to specifically use a $ k $-th row of the matrix $ A $ as a vector, we denote it by $ A^{k \cdot } $ { which is then considered as a column vector  given the above conventions. } We will also consider products of $ d\times d $ matrices using the notation $ \prod_{j=1}^nA_j=A_n\cdot ...\cdot A_1$ for  matrices $ A_1,...,A_n $.
        
        Many processes are defined by natural product definition. We use
        the notation 
        \begin{align*}
                A^n_{k:j}=\prod_{\ell=k+1}^ja_\ell.
        \end{align*}
        
                 We use total derivative notation $ \mathbf{D} $ which naturally include gradients.  In particular, in the case that $f:\mathbb{R}^d\to\mathbb{R}  $, $ {\mathbf{D}} f $ stands for the gradient row vector
unless stated otherwise.  The Hessian of $ f $ is denoted by $ \mathbf{D}^2f=(\partial^2_{jk}f)_{jk} $. In the cases considered in this article, this is always a symmetric matrix.

        When differentiating a function $ f:\mathbb{R}^d\to\mathbb{R}^d $, we use the following notation for the Jacobian
$ \mathbf{D}f=(\partial_kf_j)_{jk} $.  We may use alternatively the notation $ \partial_{x_i}f(x) $ or $ \mathbf{D}_xf $ in the case that we want to make explicit the variable with respect to which we are differentiating. Recall that under the present notation the chain rule reads 
$ {\mathbf{D}} (f\circ g)(x)
        =({\mathbf{D}} f)(g(x))\mathbf{D} g(x) $ for $ f\in C^1(\mathbb{R}^d,\mathbb{R}) $ and $ g\in C^1(\mathbb{R}^d,\mathbb{R}^d) $.
        

        We will use the capital $ O $ and small $ o $ Landau notation. In particular, for any (random) vector or matrix $ A $, we denote as $ A=O(\Delta^p) $ for $ p>0 $ the fact that $ |A|\leq C\Delta^p $ where $ C $ is non-random and is independent of $ \Delta<1  $. When $ A=0(1) $, we say that $ A $ is of constant order.
        
        
        We consider a filtered probability space $ (\Omega, (\mathcal{F}_t)_{t\geq 0},\mathbb{P}) $ which supports a $ d $-dimensional Wiener process $ W=(W^1,...,W^d) ^{\top}  $.  We will use approximations of solutions of sde's which require a time partition of the interval $ [0,T] $ whose norm is  $  \Delta $ which is chosen to be small enough with respect to certain fixed parameters to be chosen at a later stage.  
        
        We will denote by $C^1_b(A,B) $ the  space of bounded continuous functions defined on the open set $ A $  taking values in the open set $ B $. The notation $ f\in C^1({H^d_L},\mathbb{R}) $ means that the function is continuously differentiable in $ H^d_L $. In few instances, we will use the notation $ f\in C^1({\bar{H}^d_L},\mathbb{R}) $ which additionally means that the function is continuously differentiable for $ x\in \partial H^d_L $ when taking sequences in $ H^d_L $ converging to $ x $.
        
        When several functions are evaluated at the same value $y$, we may use the notation $ (fgh)(y)\equiv f(y)g(y)h(y) $.
        
        As already mentioned, we will be using the one dimensional results presented in \cite{CK}. Although some of the  arguments used throughout the paper are close to the corresponding ones introduced in \cite{CK}, the analysis here is more involved. We enumerate a few of the new difficulties: 
        \begin{enumerate}
                \item In the one dimensional case, the derivative of the killed process leads to a measure which corresponds to the reflected process in the case where there is no drift. With this in mind, one uses the Girsanov theorem. In the present case, the situation is similar but after using the Girsanov change of measure on the approximation process, the resulting process does not necessarily corresponds to the normally reflected equation unless one assumes Hypothesis \ref{hyp:12}. For more on this, see Lemma \ref{th:9} and \eqref{eq:sec}. 
                \item The characterization of the derivative process in one dimension was naturally obtained from the derivative of the killed semigroup. In the present case due to correlations this requires Hypothesis \ref{hyp:12} in order to simplify many arguments.
                \item The derivative process $ \Psi$ has the same local time terms as in the one dimensional case. In particular, in the one dimensional case, one has that in Section 4 of \cite{CK}, we proved that $ \Gamma^n$ converges to zero. In this case, this is not strictly true. In fact, if we impose Hypothesis \ref{hyp:12} the direction of the jump when the reflected process touches the boundary cancels the local time term which appears in $ \Gamma^n $. This is due to the multi-dimensional character of the problem and this does not occur in the one dimensional case (for definitions, see \cite{CK} or equivalently the term $ \gamma $ defined in \eqref{eq:td1}). This analysis generates the jump in the final result.
                \item The proof of convergence is more involved because the processes $ (\mathbf{e}^\ell)^{\top}\mathcal{E} $, $ \ell\neq 1 $ jump to zero when $ Y $ touches the boundary.
                \item From the technical point of view the proof of boundedness for second derivatives is more involved that in the one dimensional case due to the  differentiation of jumps of the process $ (\mathbf{e}^\ell)^{\top}\mathcal{E} $, $ \ell\neq 1 $ and the lack of commutativity because we have to deal with matrices in the present general case. 
        \end{enumerate}

\section{Set-up and preliminaries}

        \label{sec:set-up}
        In analogy, with the one dimensional case, we will work on the time interval $[0,T]$ with the following uniform partition $%
        t_{i}\equiv t_{i}^{n}:=i\Delta$, $i=0,...,n$ with $\Delta=\frac{T}{n}$. We assume without loss of generality, that $%
        \Delta\leq 1 $ and $ T $ is fixed. We consider the $d$-dimensional stochastic differential equation
        \begin{align*}
                X_t\equiv X_t(x)=x+\int_0^tb(X_s)ds+
                \int_0^t \sigma^{k} (X_s)dW^{k} _s.
        \end{align*}
                where $b,\sigma^{k}:\mathbb{R}^d\rightarrow \mathbb{R}^d$, $k=1,\ldots, d$ satisfy the conditions stated in Hypothesis \ref{hyp:12}. In the following, we will use the Euler approximation process of $X$ with
        initial point $X_0 $, defined as follows: For $i=0,...,n-1$ let
        \begin{equation*}  
        {X}_{t_{i+1}}^{n}=X_{t_i}^{n}+b(X_{t_i}^{n})(t_{i+1}-t_{i})+\sigma
        (X_{t_i}^{n})(W_{t_{i+1}}-W_{t_{i}}).
        \end{equation*}%
        We also use the following notation to simplify equations: $
        X_i^{n,x}=X_{t_i}^{n,x} $  where $X_{0}^{n,x}=x$ and in general $ X_i=X^n_{t_i} $. For coefficients, we use  $b_i=b_{i}^{n}=b(X_{i}^{n,x}),$ $\sigma_i=\sigma
        _{i}^{n}=\sigma (X_{i}^{n,x})$ and\footnote{%
                We use the difference operator notation $\Delta_{i+1}$ for other processes too.} $\Delta_{i+1}W=W_{t_{i+1}}-W_{t_{i}}$. Therefore, we have 
        \begin{equation}
        X_{i+1}^{n,x}\equiv X_i^{n,x}+b(X_i^{n,x})(t_{i+1}-t_{i})+\sigma
        (X_i^{n,x})\Delta_{i+1}W\equiv X_i^{n,x}+b_{i}^{n}\Delta +\sigma
        _{i+1}^{n}\Delta_iW,  \label{eq:defbX}
        \end{equation}%

        Subsequently, whenever possible, we will omit the dependence of $%
        X_{i+1}^{n,x}$ on $n$ and/or $x$ in their notation as well as that of $%
        b_{i}^{n}$ and $\sigma _{i}^{n}$ on $n$, hence  \begin{equation}
                X_{i+1}\equiv X_{i}+b(X_i)\Delta +\sigma^{k}  (X_{i})(W^{k} _{t_{i+1}}-W^{k} _{t_{i}})={X}%
                _{{i}}+b_i\Delta +\sigma^{k}  _{i}\Delta_{i+1}W^{k} ,  \label{eq:defbXm}
        \end{equation}%

Therefore using this notation, one has  
        \begin{equation}
                \mathbf{D}_{x}X_{n}=(I+\mathbf{D} b_{n-1}\Delta+\mathbf{D}\sigma _{n-1}^k\Delta_{i}W^{k} ){\mathbf{D}} _{x}X_{n-1}=\prod_{{ i  }=1}^{n}(I+\mathbf{D} b_{i-1}\Delta+\mathbf{D}\sigma _{i-1}^{k }\Delta_{i}W^{k} ).
                \label{eq:fdm}
        \end{equation}%


         Recall that we will need the following notation for the $ j$-th component of the column vector $ \sigma^{k} $ as $\sigma^{k,j} $ and $ \sigma^{\cdot j} $ stands for the $ j$-th row of the matrix $ \sigma $ considered as a column vector. Note that this notation does not create any confusion with $ \sigma_i $ introduced above where $ i $ is a time parameter\footnote{The index $ i $ is reserved to be used as an index of time throughout the rest of the article.}.
        
        \begin{definition}
                \label{def:2a}
                As in the one dimensional case we define for any r.v. $ G_i=G(X_1,...,X_i) $, the derivatives ${\mathbf{D}} _{i}G\equiv
                \partial _{X_{i}}G_i$ as the gradient (therefore it is a row vector) with respect to the i-th vector variable in $ G_i$. This definition generalizes naturally to random variables that may also depend on other independent random variables  and is also easily generalized to multi-dimensional r.v.'s. In the latter, we will not change the notation although the derivative becomes a matrix. For example, we have $ {\mathbf{D}}_{i-1}X_i= I+\mathbf{D} b_{i-1}\Delta+\mathbf{D} \sigma^k_{i-1}\Delta_iW^k$ as $ {\mathbf{D}}_{i-1}b_{i-1}=\mathbf{D}b_{i-1} $ and $ {\mathbf{D}}_{i-1}\sigma^k_{i-1}=\mathbf{D}\sigma_{i-1}^k. $
                In a similar fashion, $ D^k_i $ denotes the derivative with respect to the algebraic variable $ \Delta_iW^k $. Therefore $ D_i^kf(X_{i-1},\Delta_iW) $ denotes the derivative with respect to the $ d+k $-th variable in the function $ f $ (the first $d$ components correspond to $X_{i-1}$).                 
        \end{definition}
    
 We also make use of the continuous version $X^{c,n}$ of the Euler approximation defined as        
        \begin{equation}  \label{xcn}
        {X}_{t}^{c,n}=X_{t_i}^{c,n}+b(X_{t_i}^{c,n})(t-t_{i})+\sigma^k
        (X_{t_i}^{c,n})(W^k_{t}-W^k_{t_{i}}).
        \end{equation}%
    for $i=0,...,n-1$ and $t\in \left[ t_{i-1},t_{i}\right] $.
Note that $X_{t_i}^{c,n}=X_{t_i}^{n}=X_i$ for $i=0,...,n$.

Conditional expectations with respect to $\mathcal{F}_{i}=\sigma
(X_{j};j=1,...,i)$ are denoted by $\mathbb{E}_{i}$. 
In particular, $\mathbb{E}_{i-1}\left[ \sigma _{i-1}\Delta_{i}W\right] =\sigma
_{i-1}\mathbb{E}_{i-1}\left[ \Delta_{i}W\right] =0$.
We also let  $ \mathbb{E}_{i,x}\left[\cdot\right]:=\mathbb{E}_{i,x}\left[\cdot|X_i=x\right]$.

\begin{remark}Hypothesis \ref{hyp:12} allows for some substantial simplifications
in the calculations whilst still preserving the interesting properties of the general model. See Remark \ref{rem:12}(2) for details.  In the case where this Hypothesis is not assumed, the study of this problem seems possible but with different probabilistic representations leading to oblique reflection for $ Y $. 

\end{remark}
   \noindent    \textbf{\noindent Notation.}  
        \textit{In the following Lemma and later on we will use the notation }
       
\begin{equation}\label{maandmbar}
\frame{$
\begin{array}{|ll|ll|}
\hline\\[-4mm]
  \mathtt{X}_i^1&:=X_i^1-L&
    p_i&:=
  \exp\left(-2\frac{\mathtt{X}_i^1\mathtt{X}_{i-1}^1}{a^{11}_{i-1}\Delta}\right)
  =\exp\left(-2\frac{(X_i^1-L)(X_{i-1}^1-L)}{a^{11}_{i-1}\Delta}\right)
  \\[2mm]
  \hline\\[-4mm]
   \mathsf{h}_{i} & :=1_{(U_{i}\leq p_{i})} &  \pi_{\partial H^d_L}(X_i)&\equiv \pi(X_i):=(L,X^{2}_i,...,X^{d}_i)\\[2mm] 
 \hline\\[-4mm]
m_{i} & :=1_{(X_{i}^{1}>L)}1_{(U_{i}>p_{i})} & \bar{m}_{i} & :=1_{(X_{i}^{1}>L)}\left( 1+1_{(U_{i}\leq p_{i})}\right)  \\[2mm] 
\hline\\[-4mm]
M^n_{j} & :=\prod_{i=1}^{j}m_{i} & \bar{M}^n_{j} & :=\prod_{i=1}^{j}\bar{m}_{i}\\[1mm]
\hline
\end{array}%
$}
\end{equation}
\textit{We used a similar notation in the one dimensional (1D) case.  As in the 1D case,  we define the measure $\mathbb{P}^{M}$ such that $\left. \frac{d\mathbb{P}^{M}%
                }{d\mathbb{P}}\right\vert _{\mathcal{F}_{{n}}}=M^n_{n}$. Note that $M^n_{n}$ is not a
                martingale, but a \emph{supermartingale}, see Lemma \ref{lem:4} for details.
                Hence the measure $\mathbb{P}^{M}\,\ $is, in fact, a sub-probability measure as it
                does not integrate to 1. Under it, the continuously interpolated Euler approximation $X^{c,n}_{n}$ defined by (\ref{xcn}) has the same law as the process $X^{c,n}_{n}$ killed when it exits
                the domain, see Section \ref{apkilres} for details.
  Finally the push-forward derivative formula requires the following additional notation }
 \begin{align}
        \label{eq:piH}
        \pi_{i-1}:=& \sum_{\ell=2}^d\frac{\hat{a}^\ell_{i-1}}{a^{11}_{i-1}}
        \mathbf{e}^\ell(\mathbf{e}^{1})^{\top}  =
        \begin{pmatrix}
0&\frac{\hat{a}^2_{i-1}}{a^{11}_{i-1}}&\ldots &\frac{\hat{a}^d_{i-1}}{a^{11}_{i-1}}\\0&0&\ldots &0\\\vdots&\ddots &\ddots &\vdots \\0&\ldots&0&0\end{pmatrix}.
 \end{align}

 \begin{align}
         \nonumber
        e_{i}:=&I(1-{\mathsf{h}}_i)+\mathbf{e}^1(\mathbf{e}^1)^{\top}\mathsf{h}_i+\bar{e}_i\\
        \bar{e}_i:=&
        r_i+\mathbf{D}b_{i-1}\Delta\left(1-{\mathsf{h}}_i
        \right) +{\pi}_{i-1}\mathtt{X}^1_{i-1} \mathsf{h}_i
       \label{eq:td1bH}\\
   r_i:=&   \left(1-{\mathsf{h}}_i
   \right)\mathbf{D}\sigma_{i-1} ^{k }\Delta_{i}W^{k} 
   {+}
   \mathtt{X}^1_{i-1}\sigma^{k}_{i-1} {\mathbf{D}}_{i-1}\left(\frac{     \sigma^{1k} _{i-1}
   }{{a^{11}_{i-1}}}\right)
   \mathsf{h}_i\\
   E^{n}_j:=&\prod_{i=1}^je_i, \ \   E^{n}_0:=I \label{e}\\
        \varrho_i:=&\frac{{b}^1_{i-1}
        }{{a^{11}_{i-1}}}  
        \mathbf{e}^1\mathsf{h}_i+\tilde{\varrho}_i\label{eq:rho}\\
        \tilde{\varrho}_i:=&
        \mathsf{h}_i
        \mathtt{X}^1_{i-1}{\mathbf{D}}_{i-1}\left(\frac{{b}^1_{i-1}
        }{{a^{11}_{i-1}}}  
        \right)
        .\label{eq:rhot}
 \end{align}
\textit{The random variable  
 $ {\pi}_{i-1}\mathtt{X}^1_{i-1} \mathsf{h}_i $ can be rewritten as
\begin{align*}
{\pi}_{i-1}\mathtt{X}^1_{i-1} \mathsf{h}_i =
\sum_{\ell=2}^d\left(\frac{a_{i-1}^{1\ell}}{{a^{11}_{i-1}}}  -\frac{a^{1\ell}(\pi({X}_{i-1}))}{a^{11}_{i-1}} \right)
\mathbf{e}^\ell(\mathbf{e}^{1})^{\top}  \mathsf{h}_i.
\end{align*}
We will show that the process with increments given by $ {\pi}_{i-1}\mathtt{X}^1_{i-1} \mathsf{h}_i $  converges to a process that is proportional to  the process $B$ introduced in (\ref{bb}) . }

 Using the notation introduced above, the push-forward formula is as follows
        \begin{lemma}[\bf{The push-forward formula}]
        \label{lem:2.1H}
        Let $ f\in C^1_b(\mathbb{R}^d,\mathbb{R}) $ satisfying  $ f \big |_{\partial H^d_L}=0$ then  $f_{i}\in { C^1_b(\bar{H}^d_L,\mathbb{R})  } $ satisfies the boundary condition 
        $f_{i}\big |_{\partial H^d_L}=0$. Furthermore, for  $i=1,...,n$,
        \begin{align}
                \label{eq:tdH}
                & {\mathbf{D}} _{i-1}\mathbb{E}_{i-1}\left[ {f}_{i}{m}_{i}\right] =\mathbb{E}_{i-1}%
                \left[ {\mathbf{D}} _{i}{f}_{i}(X_i){e}_{i}
                \bar{m}_i
                \right] +\mathbb{E}_{i-1}%
                \left[{f}_{i}{{\varrho}}^{\top}  _i\bar{m}_i
                \right].
        \end{align}
        This gives the following iterated formulas for the derivative 
        \begin{align}
                \nonumber
                {\mathbf{D}} _{x}\mathbb{E}\left[ f\left( X_{T\wedge \tau }^{n,x}\right) \right]
                =&{\mathbf{D}} _{x}\mathbb{E}\left[ f\left( X_{n}^{n,x}\right) {M}_{n}\right]\\
                =&\mathbb{E}\left[{\mathbf{D}} {f}\left( X_{n}^{n,x}\right) E^{n}_n\bar{M}_{n}              \right] +\mathbb{E}\left[\sum_{i=1}^nf_i \varrho_i^{\top} E^{n}_{i-1} \bar{M}_{i}\right ].\label{fprimemH}
        \end{align}
 \end{lemma}
The proof of this result is given in Section \ref{diffrules} and is based on a set of differentiation formulas stated in the same section.

        \begin{remark}
                \label{rem:3} 
                Hypothesis \ref{hyp:12} implies that the term corresponding to the constant order in the definition of $ e_i $ has the following simplified heuristic structure:
                        \begin{align*}
                                (I+(\pi_{i-1}-I)\mathsf{h}_i)\bar{m}_i\approx
                                {\begin{pmatrix}1+\mathsf{h}_i&O^E_{i}(1)&\ldots &O^E_{i}(1)&O^E_{i}(1)\\0&1-\mathsf{h}_i&0&\ldots &0\\\vdots&\ddots&\ddots &\ddots &\vdots \\\vdots&&\ddots&\ddots &0\\0&\ldots&\ldots&0&1-\mathsf{h}_i\end{pmatrix}}.
                        \end{align*}
                        That is, the first component in the diagonal will behave as reflection while all other terms in the first line will be multiples of $ \mathtt{X}^1_{i-1}\mathsf{h}_i $. The other diagonal terms correspond to the characteristics of killing as $ m_i $ in the conditional probability of not crossing the boundary in the time interval $ [t_{i-1},t_i] $.                           
                
        \end{remark}

\section{The Girsanov change of measure}
\label{sec:4}
        
Similar to the one dimensional case, we will apply the Girsanov's change of measure in the original probability space $(\Omega ,%
       ( \mathcal{F}_t)_{t\in[0,T]},\mathbb{P})$ which removes the drift in the Euler approximation\ (%
        \ref{eq:defbX}). As in the one dimensional case, this change of measure is necessary in order to be able to use that $X^1 _i$ is a process without drift with normal reflection when it hits the boundary of the domain $H^d_L$ (see Lemma \ref{th:9}) under $\bar{M}^n $\footnote{This characterization is not valid for Brownian motion with non-zero drift.}. For this, we
        introduce the (discrete) exponential martingale $\mathcal{K}^n_{n}$ corresponding to the
        increments   
\begin{align*}\kappa_j:=&(\sigma^{-1}_{j-1})^{k\cdot}\cdot b_{j-1} Z^k_j-\frac{1}{2}\sum_{k=1}^d\left\|(\sigma^{-1}_{j-1})^{k\cdot}\cdot b_{j-1}\right \|^2\Delta\\
        {Z}^k_j:=&\Delta_jW^{k} + (\sigma^{-1}_{j-1})^{k\cdot}
        \cdot b_{j-1}\Delta,
\end{align*}
where 
\begin{equation}\label{kn}
\mathcal{K}^n_{i}:=\exp\left(\sum_{j=1}^i\kappa_j\right).
\end{equation}
 We define $\widetilde{\mathbb{P}}\equiv \widetilde{\mathbb{P}}^n$ to be given by 
        \begin{equation*}
        \left. \frac{d\widetilde{\mathbb{P}}^n}{d\mathbb{P}}\right\vert _{\mathcal{F}_{{i}}}=(\mathcal{K}%
        _{n}^{n})^{-1}.
        \end{equation*}%
        Then, under $\widetilde{\mathbb{P}}^n$, the Euler approximation $X^{n,x}$ satisfies 
        \begin{equation}
        X_{i}^{n,x}=X_{i-1}^{n,x}+\sigma ^k_{i-1}{Z}^k_{i},  \label{eq:defbXunderPtilde}
        \end{equation}%
        and, under $\widetilde{\mathbb{P}}^n$, the random variables ${Z}_{i}$, $i=1,...,n$ are
        i.i.d. random variables with common distribution $N\left( 0,\Delta \right).$ We will denote by  $ \mathbb{\widetilde{E}} $ the expectation under $\widetilde{\mathbb{P}}^n$ and from now on we work under this measure.

All random variables in Lemma \ref{lem:2.1H} are equivalently expressed  so that all terms appearing in \eqref{eq:tdH} depend now on $ Z_i^k $ and not on $ \Delta_iW^k $. In particular, we will decompose the terms in $ e_i $ (see \eqref{eq:td1bH}) as follows:
\begin{align}
        \label{eq:td1}
        e_i=&I+(\mathbf{e}^1(\mathbf{e}^1)^{\top}-I)
        \mathsf{h}_i
        +\bar{e}_i,\\
        \bar{e}_i=&\mathbf{D}\sigma_{i-1} ^{k }
        \bar {r}^k_{i}+\bar{b}_{i-1}\Delta(1-\mathsf{h}_i)+
        \gamma_{i}
        +\pi_{i-1}\mathtt{X}^1_{i-1}\mathsf{h}_i,\notag\\
        \bar{r}^k_i:=&{Z}^k_i-\tilde{\mathbb{E}}_{i-1}\left[{Z}^k_i\bar{m}_i\right],\\
        \gamma_i:=&\mathbf{D}\sigma_{i-1} ^{k }\tilde{\mathbb{E}}_{i-1}\left[{Z}^k_i\bar{m}_i\right]
        -\mathsf{h}_i\left(\mathbf{D}\sigma_{i-1} ^{k }Z_i^k-\mathtt{X}^1_{i-1}             \sigma^{k}_{i-1} {\mathbf{D}}_{i-1}\left(\frac{        \sigma^{1k} _{i-1} 
        }{{a^{11}_{i-1}}}\right)
\right) ,\nonumber
        \\
        \bar{b}_{i-1}:=&\mathbf{D}b_{i-1}-\mathbf{D}\sigma_{i-1}^k\left((\sigma^{-1}_{i-1})^{k\cdot}\cdot b_{i-1}\right).\nonumber
\end{align}
Similar to the 1D case $ \bar{r} $ will be the noise driving term of $ E^n $, its drift term will be $ \bar{b} $. The martingale part of the process $ \gamma_{i} $ is asymptotically  negligible as we will see later in Lemma \ref{lem:14}. 

Next we define the measure $ {\mathbb{Q}}^n$ such that $\left. \frac{d{\mathbb{Q}}^n}{d\widetilde{\mathbb{P}}^n}\right\vert _{\mathcal{F}_{{n}}}=\bar{M}^n_{n}$. Note that
        the process $\bar{M}^n_{n}$ is a martingale with mean $ 1. $

Therefore the goal is to determine for $ {f}_{i} =\mathbb{E}_{i,x}\left[ f\left( {X}_{n}\right) \mathcal{K}_{i:n}{M}_{i:n}\right]  $, the limit of the expression 
\begin{align*}
        &\tilde{\mathbb{E}}_{0,x}\left[  {\mathbf{D}} {f}\left( X_{n}\right)E^{n}\mathcal{K}_n\bar{M}_{n}%
        \right] +\tilde{\mathbb{E}}_{0,x}\left[\sum_{i=1}^n f_i\varrho_i^{\top} E^{n}_{i-1}\mathcal{K}_i\bar{M}_{i}\right ].
\end{align*}

To be sure that the above is well defined, we start with some uniform moments results.
In the same way as in the one dimensional case we have 

\begin{lemma}
        \label{moments}For arbitrary $p\in\mathbb{N}$, there exists a constant $C=C\left(
        p\right) $ independent of $i$ and $n$ such that for any compact set $ K\subseteq {H^d_L} $ 
            \begin{align}
                    \sup_{x\in K}\tilde{\mathbb{E}}_{0,x}\left[ \max_{i=1,...,n}e^{-C(t_n-t_i)}\left( {{\mathcal{K}}}_{i:n}^{n}\right) ^{p}
                \bar{M}_n^n\right]\leq& C,\nonumber\\ \sup_{x\in K}\tilde{\mathbb{E}}_{0,x}\left[ \max_{i=1,...,n}\left( \bar{M}_{i:n}^{n}\right)
                ^{p}\right]  \leq &2^{pn}. \label{Lnboundsm}
            \end{align}
\end{lemma}

As in the one dimensional case, the following result is an application of Lemma \ref{lem:60}.
\begin{lemma}
        \label{lemma:7} The following bound is satisfied for any $ p\in\mathbb{N} $ and for a positive constant $ C $ independent of $ n $, $ i $ and $ x $  \footnote{Here $ \|\cdot\|_F $ denotes the Frobenius norm. That is,  for any square matrix $ A $, $ \|A\|_F^2:=\sum_{j,k=1}^dA_{jk}^2 $. We often drop the subscript $ F $ as the norm is well understood if the argument is a matrix.}
        \begin{align*}
                \sup_{x\in K}\tilde{\mathbb{E}}_{0,x}\left[\max_{j\geq i}e^{-C(t_j-t_i)}\|E^{n}_{i:j}\|_F^{2p}\bar{M}_n\right]\leq C.
        \end{align*}
       
\end{lemma}

\begin{proof}[Proof of Lemmas \ref{moments} and \ref{lemma:7}]
        The proof follows along the same lines as in the 1D case using the moment results in Lemma \ref{lem:60} of Section \ref{sec:app3}. See the proof of Lemma 9 in \cite{CK}.
        The only point that is new is to verify the hypotheses on $ K_i:=(I-\mathbf{e}^1(\mathbf{e}^1)^*)\mathsf{h}_i $ stated in Lemma \ref{lem:60}. This is straightforward.
\end{proof}

Recall that 
\[
   p_i:=p_i(X_{i-1},X_i)=
        \exp\left(-2\frac{\mathtt{X}_{i-1}^1\mathtt{X}_i^1}{a^{11}_{i-1}\Delta}\right)
        =\exp\left(-2\frac{(X_{i-1}^1-L)(X_i^1-L)}{a^{11}_{i-1}\Delta}\right).
\]

We have the following elementary lemma which is parallel to a similar result in the 1D case. We will use the notation (as in the 1D case)

\begin{align*}
         \mathtt{X}_{i-1}^{1,\sigma}:=\frac{\mathtt{X}_{i-1}^{1}}{\sqrt{a^{11}_{i-1}\Delta}}=\frac{{X}_{i-1}^{1}-L}{\sqrt{a^{11}_{i-1}\Delta}}
\end{align*}

        \begin{lemma}
                \label{lem:4}
                The discrete time process  $M^n$ is  a supermartingale and
                $\bar{M}^n$ is a martingale. In fact, we have
                \begin{align*}
                        \tilde{\mathbb{E}}_{i-1}[m_i]=&\bar{\Phi}\left(-
                        {\mathtt{X}_{i-1}^{1,\sigma}}\right)-\bar{\Phi}\left(
                        {\mathtt{X}_{i-1}^{1,\sigma}}\right)\leq 1,\\
                        \tilde{\mathbb{E}}_{i-1}[\bar{m}_i]=&\bar{\Phi}\left(-
                        {\mathtt{X}_{i-1}^{1,\sigma}}\right)+\bar{\Phi}\left(
                        {\mathtt{X}_{i-1}^{1,\sigma}}\right)=1.
                \end{align*}

        \end{lemma}
    
        \begin{proof}
Recall that $m_i=1_{({X}^1_i>L)}(1-\mathsf{h}_i) $. We only compute one of the terms  as the proof of the other ones are similar.  We have that    

                               \begin{align*}
                        \tilde{\mathbb{E}}_{i-1}\left[  1_{\left( {X}%
                                _{i}^1>L\right )}p_i\right]=&
                                 \int_{x_i^1>L}
                        \frac{p_i(X_{i-1},x_i)}{\sqrt{(2\pi \Delta)^d\det(a_{i-1})}}\exp\left(-\frac{(x_i-X_{i-1})\cdot[a_{i-1}^{-1}(x_i-X_{i-1})]}{2\Delta}\right)dx_i\\
                        =&
                        {\frac{1}{\sqrt{2\pi {a^{11}_{i-1}}\Delta}}} \int_{x_i^1>L}
                        \exp\left(-\frac{(x^1_i+X^1_{i-1}-2L)^2 }{2{a^{11}_{i-1}}\Delta}\right)dx^1_i,
                \end{align*}
where $ x^1_i $ denotes the first coordinate of the vector $ x_i $.  
        Finally, using the definition of $ \bar{\Phi} $ gives the result. 
        \end{proof}

        When dealing with remainders it will be useful to introduce the concept of asymptotically negligible in expectation.  We give next its definition which already appeared in the 1D case. 
       \begin{definition}
        \label{def:2}

        Let $ \Upsilon_i\equiv \Upsilon^n_i(x)\in \mathcal{F}_{t_i} $ be a sequence of $ L^{2q} (\Omega)$ -integrable random variables,
        we say that
        the family of random variables $Y_i\in \mathcal{F}_i $, $i=1,...,n 
        $ is of order $ \Delta^p $ in expectation (under the measure $ \bar{M}^n $) if for any $ q\in\mathbb{N} $ and any compact set $ K\subseteq \mathbb{R}^d $, we have
        \begin{align}
                \label{eq:neg}
                \sup_{x\in K}\tilde{\mathbb{E}}_{0,x}\left[\left|\sum_{i=1}^n\Upsilon_i\right |^q\bar{M}^n_n\right]\leq C\Delta^{pq}.
        \end{align}
        As notation, we will use $ \Upsilon_i=O^E_i(\Delta^p) $ or $ O^{E,\bar{M}}_i(\Delta^p) $ in the case the Radon-Nikodym random variable $ \bar{M}^n $ is stressed. 
        The usual big O notation will also be used as in $\Upsilon_i=O_i(\Delta^p) $. This means that $ \Upsilon_i\in\mathcal{F}_i $ and $ |\Upsilon_i|\leq C\Delta^p $. We may also abuse slightly this notation using $ \Upsilon_i=O_i(Z_i) $  to mean that $ |\Upsilon_i|\leq C|Z_i| $.
       \end{definition}
       
   The extension to higher dimensions is also clear by changing the absolute value in \eqref{eq:neg} by the appropriate norm. 
   
   When dealing with remainders this concept is useful and we say that the sequence is  asymptotically negligible in expectation when $ p>0 $. 
   
   As in the one dimensional case,  we will deduce the limit of the  noise driving process
   \begin{align}\label{rn}
        R^n_t:=&\sum_{i=0}^{n-1}1_{[0,t_{i+1})}(t)\bar{r}_i,\\
   \label{gamman}
        \Gamma^n_t:=&\sum_{i=0}^{n-1}1_{[0,t_{i+1})}(t)\gamma_i,
   \end{align}
   as well as the limit of the following processes 
   \begin{align}
        { \mathbf{B} _{t}^{n}  }  :=&\sum_{i=0}^{n-1}1_{[0,t_{i+1})}(t)\sigma_{i-1}\tilde{\mathbb{E}}_{i-1}[Z_{i}\bar{m}_{i}],
        \label{fn}\\
        E_{t}^{n} :=&
        \sum_{i=0}^{n-1}1_{[0,t_{i+1})}(t)\prod_{j=1}^{i}e_{j},\nonumber\\
        \mathcal{K}_{t}^{n} :=&
        \sum_{i=0}^{n-1}1_{[0,t_{i+1})}(t)\mathcal{K}^n_{i}.\label{Kn}
         \end{align} 
  In the proof of the following result, we will use the following notation: Let $g_{i}$ denote
the conditional density of the increment $\sigma _{i-1}Z_{i}$ with respect
to $\mathcal{F}_{i-1}$ (recall that $ a_{i-1}=\sigma_{i-1}\sigma^{\top}_{i-1}   $),
given by\begin{equation*}
        g_{i}(z)=\frac{1}{\sqrt{\det(a_{i-1})\left(2\pi \Delta\right)^{d}}}\exp \left( -\frac{%
                z\cdot [a_{i-1}^{-1}z]}{2\Delta }\right) .
\end{equation*}%

The marginal density for the first component (corresponding to the normal direction to the boundary) is given for $ z\in\mathbb{R} $ as 
\begin{align*}
        g_{i}^1(z)&=\sqrt{\frac{1}{{2\pi {a^{11}_{i-1}}\Delta  }}}
        \exp\left(-\frac{z^2 }{2{a^{11}_{i-1}}\Delta}\right).
\end{align*}

       Furthermore, from \eqref{eq:6.1m}, note that $ \Delta_i\mathbf{B}^n= a_{i-1}\mathbf{e}^1\Delta_iB^n$, where 
       \begin{align*}
        \Delta_iB^{n}   =&   2\Delta g_{i-1}^1(\mathtt{X}^1_{i-1})
        -      
        2\bar{\Phi}\left(
        \mathtt{X}_{i-1}^{1,\sigma}
        \right)
        \frac{\mathtt{X}_{i-1}^1}{ {{a^{11}_{i-1}} }}.
       \end{align*}
       With these definitions, we let
       \begin{align*}
       B _{t}^{n} :=&\sum_{i=0}^{n-1}1_{[t_i,t_{i+1})}(t)
        \sum_{j\leq i}\Delta_jB ^{n}.
       \end{align*}
        \begin{lemma}
                \label{lemma:4} Under $ \mathbb{Q}^n $ and for $ j,k\in\{1,...,d\} $ and $ i\in\{1,...,n\} $, we have  for any $ p>1 $
                \begin{align}
                        \tilde{\mathbb{E}}_{i-1}\left[Z_i\bar{m}_{i}\right]=&     
        (\sigma_{i-1})^{\top}  \mathbf{e}^1\Delta_iB^n.
                        \label{eq:6.1m}
                \end{align}
        
        For covariances, we have
\begin{align*}
   \tilde{\mathbb{E}}_{i-1}\left[Z^j_iZ^k_i\bar{m}_{i}\right]=&   \delta_{jk}\Delta+
   O_{i-1}^E(\sqrt{\Delta}).
\end{align*}

                Finally, for any $ p\geq 1 $, there exists a constant $ C_p>0 $ such that
                \begin{align}
                        \label{eq:6.4m} \tilde{\mathbb{E}}_{i-1}\left[|Z_i|^p\bar{m}_{i}\right]\leq C_p\Delta^{p/2}.
                \end{align}
        \end{lemma}

        \begin{proof}
                The proof uses many arguments that already appeared in the 1D case and in extended form in the proof of Lemma \ref{lem:4}.       
                We will explain in detail the steps for the estimation of $ \tilde{\mathbb{E}}_{i-1}\left[Z_i\bar{m}_{i}\right] $ without going into detailed calculations which are similar to the 1D case and uses the integration method which appears in the proof of Lemma \ref{lem:4}. The higher moments are done in a similar way. 
                
                Recall that we assume throughout the proof that $ X^1_{i-1}>L$. Using the definition of $ \bar{m}_i $, we have for $ i\in\{1,...n\} $ 
                \begin{align*}
                        \tilde{\mathbb{E}}_{i-1}\left[Z_i\bar{m}_{i}\right]=&   \tilde{\mathbb{E}}_{i-1}\left[Z_i       1_{\left( {X}%
                                _{i}^1> L\right )}\left(
                        1+\mathsf{h}_i
                        \right )\right].
                \end{align*}            
             We continue with the calculation in two parts as follows:
                \begin{align}
                        \label{eq:z1}
                        &      \tilde{\mathbb{E}}_{i-1}\left[Z_i   1_{\left(X^1_i
                                > L \right )}\right],\\
                        &\tilde{\mathbb{E}}_{i-1}\left[Z_i   1_{\left( X^1_i> L
                                \right )}\mathsf{h}_i
                \right].
                        \label{eq:z2}
                \end{align}

                {\it Part 1: Estimate for \eqref{eq:z1}} (recall that we are using the summation over double indexes notation)

                \begin{align*}
                        &\tilde{\mathbb{E}}_{i-1}\left[Z_i 1_{\left( {X}%
                                _{i}^1> L\right )}\right]\\=&
                        \int_{x^1_i> L}
                        \frac{[\sigma^{-1}_{i-1}(x_i-X_{i-1})]}{\sqrt{(2\pi \Delta)^d\det(a_{i-1})}}\exp\left(-\frac{(x_i-X_{i-1})\cdot[a_{i-1}^{-1}(x_i-X_{i-1})]}{2\Delta}\right)dx_i\\
                        =&
                        \int_{x^1_i> L}
                        \frac{[\sigma^{-1}_{i-1}\mathbf{e}^k](x_i-X_{i-1})^k}{\sqrt{(2\pi \Delta)^d\det(a_{i-1})}}\exp\left(-\frac{(x_i-X_{i-1})^{\ell} (x_i-X_{i-1})^l \mathbf{e}^{\ell} \cdot[a_{i-1}^{-1}\mathbf{e}^l]}{2\Delta}\right)dx_i.
                \end{align*}
%
                

                Therefore, using the explicit conditional distributions of multivariate Gaussian laws and integrating with respect to $ x^2,...,x^{d} $, we obtain for any $ p>0 $
                \begin{align*}
                        &\tilde{\mathbb{E}}_{i-1}\left[Z_i 1_{\left( {X}%
                                _{i}^1> L\right )}\right]\\=&
                       {\frac{1}{\sqrt {2\pi {a^{11}_{i-1}} \Delta}}}[\sigma^{-1}_{i-1}\mathbf{e}^k] \frac{ a^{k1}_{i-1}}{{a}^{11}_{i-1}}\int_{x_i^1>L}(x_i-X_{i-1})^1
                        \exp\left(-\frac{(x^1_i-X^1_{i-1})^2 }{2{a^{11}_{i-1}}\Delta}\right)dx^1_i\\
                        =&\sqrt{\frac{\Delta }{{2\pi {a^{11}_{i-1}} }}}{[\sigma^{-1}_{i-1}\mathbf{e}^k]}\frac{ a_{i-1}^{k1}}{{a}^{11}_{i-1}}
                        {a^{11}_{i-1}}
                        \exp\left(-\frac{(X^1_{i-1}-L)^2 }{2{a^{11}_{i-1}}\Delta}\right)\\
                        =&
                        [\sigma^{-1}_{i-1}\mathbf{e}^k]
                        { a_{i-1}^{k1}}\Delta g_{i-1}^1(X^1_{i-1}).
                \end{align*}
        The final step is to realize that $[\sigma^{-1}_{i-1}\mathbf{e}^k]^\ell
         a_{i-1}^{k1}= \sigma_{i-1}
        ^{1\ell} $.        
                This finishes our arguments for the estimation of \eqref{eq:z1}.    
                This type of argument will be used repeatedly in the rest of this proof.

                {\it Part 2: Estimate for \eqref{eq:z2}:}    
            Define $ \bar{g}_{i-1}(X_{i-1},x_i):=   p_i(X_{i-1},x_i)g_{i-1}(x_i-X_{i-1}) $, then \eqref{eq:z2}, becomes
                \begin{align}
                        \label{eq:9.1}
                        &\tilde{\mathbb{E}}_{i-1}\left[Z_i 
                        p_i\right]
                        = 
                        \int_{x^1_i> L}[\sigma^{-1}_{i-1}(x_i-X_{i-1})]
                        \bar{g}_{i-1}(X_{i-1},x_i)dx_i.
                \end{align}
        
                As in the previous calculation, algebraic simplifications and integrating with respect $x^\ell $, $ \ell\neq 1 $ first and then with respect to $ x^1 $, one obtains
                \begin{align}
                        \label{eq:7.1}
                       \tilde{\mathbb{E}}_{i-1}\left[Z_i 
                       p_i\right]
                       =&
                                             [\sigma^{-1}_{i-1}\mathbf{e}^k]{ a^{k1}_{i-1}}
                        \left (
                        \Delta g_{i-1}^1(\mathtt{X}^1_{i-1})
                        -2      \sqrt{\Delta}
                                \bar{\Phi}\left(
                                \mathtt{X}_{i-1}^{1,\sigma}
                                        \right)
                                        \frac{\mathtt{X}_{i-1}^{1,\sigma}}{\sqrt{a^{11}_{i-1}}}     
                        \right).
                \end{align}
                This finishes Part 2.
                
                The estimates on covariances are obtained 
                using results for conditional covariances of the components of the vector $ X_i-X_{i-1}$. For simplicity, we only discuss the case $ d=2 $. Using the conditional law of $  (X_i-X_{i-1})^2$ conditioned to $ (X_i-X_{i-1})^1 $, we obtain after explicit calculations:
                \begin{align*}
                        \tilde{\mathbb{E}}_{i-1}\left[(X_i-X_{i-1})^j(X_i-X_{i-1})^k\bar{m}_i\right]=a_{i-1}^{jk}\Delta+O_{i-1}^E(\sqrt{\Delta}).
                \end{align*}
            In fact, denote by $  \rho^2_{i-1}$, the correlation between  $ (X_i-X_{i-1})^1$ and $(X_i-X_{i-1})^2 $ then the term $ O_{i-1}^E(\sqrt{\Delta}) $ in the above expression can be written as
            \begin{align*}
     &-4\Delta    \mathtt{X}^{1,\sigma}_{i-1}
                \vartheta _{i-1}\times 
                \begin{cases}
                        a^{11}_{i-1}&\text{ if }j=k=1\\
                        \rho^2_{i-1}a^{22}_{i-1}&\text{ if }j=k=2\\
                        \rho_{i-1}\sqrt{a^{11}_{i-1}a^{22}_{i-1}}& \text{ if }j\neq k
                \end{cases}\\
            \vartheta _{i-1}    := &\sigma _{i-1}\sqrt{\Delta }
                g_{i}(\mathtt{X}^1_{i-1})-\mathtt{X}_{i-1}^{1,\sigma }\bar{\Phi}\left( \mathtt{X}_{i-1}^{1,\sigma }
                \right).
            \end{align*}
        Then in order to obtain the result, we compute
        \begin{align*}
                \tilde{\mathbb{E}}_{i-1}\left[Z_{i-1}^jZ_{i-1}^k\bar{m}_i\right]=&\tilde{\mathbb{E}}_{i-1}\left[\left(\sigma_{i-1}^{-1}(X_{i}-X_{i-1})(X_{i}-X_{i-1})^{\top}(\sigma_{i-1}^{-1})^{\top} \right)^{jk}\bar{m}_i\right]\\=&
                \sigma^{-1}_{i-1}a_{i-1}(\sigma_{i-1}^{-1})^{\top}\Delta+O_{i-1}^E(\sqrt{\Delta}).
        \end{align*}

                The estimates about general $ p $ moments in \eqref{eq:6.4m} follow using Cauchy-Schwartz inequality and the fact that 
               $ \tilde{\mathbb{E}}_{i-1}\left[\bar{m}^2_{i}\right] \leq 4$.
        \end{proof}
    

        \begin{lemma}
                \label{lem:5m}
                
                There exists a
                constant $c=c\left( T\right) $ independent of $n$ such that for any $ p\in\mathbb{N} $,
                \begin{equation}
                \widetilde{\mathbb{E}}\left[ \left 
                | \sum_{i=1}^{n}\widetilde{\mathbb{E}}_{i-1}\left[ 
                {Z}_{i}\bar{m}_{i}\right] \right| ^p\bar{M}^n_{n}\right] \leq c,\quad~~~~\widetilde{%
                        \mathbb{E}}\left[ \max_{j\leq n-k-1}\left| \sum_{i=j}^{j+k}
                        \widetilde{\mathbb{E}}%
                _{i-1}\left[ {Z}_{i}\bar{m}_{i}\right] \right|^p \bar{M}^n_{n}\right] \leq c%
                (\left( k+1\right) \Delta )^{p\over 2}  \label{control1}
                \end{equation}
                In addition, the following bounds hold $\mathbb{P}$ (or $\widetilde{\mathbb{P}}$%
                )-almost surely\footnote{%
                        Due to the absolute continuity of ${\mathbb{Q}}^n$ with respect to $\mathbb{P}$%
                        , bounds (\ref{control2}) and (\ref{control3}) also hold ${\mathbb{Q}}^n$%
                        -almost surely.} 
                \begin{align}
                \sum_{i=1}^{n}\widetilde{\mathbb{E}}_{i-1}\left[ |{Z}_{i}|^{2}\bar{m}_{i}\right]
                \leq &c,\quad\quad \max_{j\leq n-k-1}\sum_{i=j}^{j+k}\widetilde{\mathbb{E}}_{i-1}\left[
                |{Z}_{i}|^{2}\bar{m}_{i}\right] \leq c\left( k+1\right) \Delta .
                \label{control2} \\
                \sum_{i=1}^{n}\widetilde{\mathbb{E}}_{i-1}\left[ |{Z}_{i}|^{4}\bar{m}_{i}\right]
                \leq &c\Delta, \quad\quad\max_{j\leq n-k-1}\sum_{i=j}^{j+k}\widetilde{\mathbb{E}}%
                _{i-1}\left[ |{Z}_{i}|^{4}\bar{m}_{i}\right] \leq c\left( k+1\right) \Delta
                ^{2}.  \label{control3}
                \end{align}
        
        \end{lemma}
        \begin{proof}
                The estimates for \eqref{control2} and \eqref{control3} follow directly from \eqref{eq:6.4m}. In the proof of \eqref{control1}, one applies Lemma \ref{lem:essb}.
                                \end{proof}

                     Using the moment results in Lemma \ref{lemma:4}, one obtains:
        \begin{cor}
                \label{cor:2}
                Under $ \mathbb{Q}^n $ and for any $ p\geq 2 $, we have 
                \begin{align*}
                        \tilde{\mathbb{E}}_{i-1}\left[\left(        I(1-\mathsf{h}_i)+\mathbf{e}^1(\mathbf{e}^1)^{\top}
                        \mathsf{h}_i
                        \right)\bar{m}_i\right]=&I\left(1-2\bar{\Phi}\left(
                        {{\mathtt{X}_{i-1}^{1,\sigma}}}\right)\right )+       2\bar{\Phi}\left(
                        {{\mathtt{X}_{i-1}^{1,\sigma}}}\right)\mathbf{e}^1(\mathbf{e}^1)^{\top},
                        \\
                        \tilde{\mathbb{E}}_{i-1}\left[\left(        \mathbf{D}\sigma_{i-1} ^{k }
                        \bar {r}^k_{i}\right)\bar{m}_i\right]=&0.
                \end{align*}
        \end{cor}
    \begin{lemma}
        \label{lem:14}
       The process $\bar\Gamma^n:=\sum_{i=0}^{n-1}1_{[t_i,t_{i+1})}(t) \left(\gamma_i- \mathbb{E}_{i-1}\left[ \gamma_{i}\bar{m}_{i}\right]\right)$ is a square integrable martingale under $\mathbb{Q}_{n}$ and its quadratic variation converges to zero.   
        \end{lemma}
        \begin{proof}
               Applying Lemmas \ref{lemma:4} and \ref{lem:4}, one gets after algebraic simplification: 
                \begin{align}
                        \tilde{\mathbb{E}}_{i-1}\left[ \gamma_{i}\bar{m}_i\right]=2\bar{\Phi}\left(
                        {{\mathtt{X}_{i-1}^{1,\sigma}}}\right)\left(\mathbf{D}\sigma_{i-1} ^{k }
                        \frac{  \sigma^{k1} _{i-1} 
                        }{{a^{11}_{i-1}}}\mathtt{X}^1_{i-1}
                        +\mathtt{X}^1_{i-1}              \sigma^{k}_{i-1} {\mathbf{D}}_{i-1}\left(\frac{       \sigma^{1k} _{i-1} 
                        }{{a^{11}_{i-1}}}\right)\right).
                \label{eq:gnz}
                \end{align}
            Using the product rule for derivatives, one obtains:                
                 \begin{align*}
                        \tilde{\mathbb{E}}_{i-1}\left[ \gamma_{i}\bar{m}_i\right]=2\bar{\Phi}\left(
                        {{\mathtt{X}_{i-1}^{1,\sigma}}}\right)\mathtt{X}_{i-1}^{1}\pi_{i-1}\mathbf{e}^1+O_{i-1}^E(\sqrt{\Delta}).
                \end{align*}
            Therefore the first result follows.
            
                The result about the quadratic variation of $ \Gamma $  being negligible follows from explicit calculations together with Lemma \ref{lem:4} and Lemma \ref{lemma:4}.
        \end{proof}
    
In order to study terms related to local times, one uses arguments from the 1D case. In fact, one proves the following result: 
        \begin{lemma}
                \label{lem:essb} Let $F:\mathbb{R}\mapsto (0,\infty )$ be a positive valued
                function with Gaussian decay at infinity. In other words, there exists $c>0$
                such that 
                \begin{equation}
                        F(x)\leq ce^{-\frac{\left\vert x\right\vert ^{2}}{2c}}.  \label{eq:Fbu'}
                \end{equation}%
                Assuming that $\Delta \leq 1$ there exists a constant $C$ independent of $n$
                and $T$ such that for any $q\in\mathbb{N} $ and any compact set $ K\subseteq H^d_L $, we have
                \begin{equation}
                        \Delta ^{\frac{1}{2}}\sup_{x\in K}\tilde{\mathbb{E}}_{0,x}\left[ \left|\sum_{i=1}^{n}
                        F(\mathtt{X}_{i-1}^{1,\sigma })\right|^q \bar{M}_{n}\right] ^{1/q}<C\sqrt{T}.  \label{eq:essb}
                \end{equation}
                We will frequently use the above result with $g(x)=|x|^ke^{-\frac{|x|^2}{2}} $
                or $g(x)=|x|^k\bar{\Phi}(x) $, $k\geq 0 $.
        \end{lemma}
        \begin{proof}
                The proof follows the same sequence of equalities as in the 1D case if one recalls how integrals were handled in the proof of Lemma \ref{lem:4} and using the results in Section \ref{sec:mlt}. In fact, one has to replace the term $ X^1 $ of \eqref{eq:6.1m}
                by the distance of the reflected equation to the boundary based on $ \mathcal{X} $ which is introduced in Section \ref{sec:mlt}.
        \end{proof}
        
        \begin{remark}
                \label{rem:15}
                From Lemmas \ref{lem:4} and \ref{lem:essb}, we have  that the following sequences are asymptotic negligible for $ k,m\in \mathbb{N} $,
                \begin{align*}
                     \tilde{\mathbb{E}}_{i-1}\left[(\mathtt{X}^1_{i-1})^k1_{(U_i\leq p_i)}\right]   =&O^E_{i-1}(\Delta^{(k-1)/2}),\\
                    \tilde{\mathbb{E}}_{i-1}\left[    |Z_i|^{k}1_{(U_i\leq p_i)}\right] =&O^E_{i-1}(\Delta^{(k-1)/2}),\\
                        \tilde{\mathbb{E}}_{i-1}\left[|Z_i|^k (\mathtt{X}^1_{i-1})^m1_{(U_i\leq p_i)}\right]=& O^E_{i-1}(\Delta^{(k+m-1)/2}),\\
                       \tilde{\mathbb{E}}_{i-1}\left[ |Z_i|^k (\mathtt{X}^1_{i})^m1_{(U_i\leq p_i)}\right]=& O^E_{i-1}(\Delta^{(k+m-1)/2}) .
                \end{align*}
        \end{remark}
        From this remark, one deduces 
        \begin{cor}
                \label{cor16}
                The sequence $ \{-\tilde{\mathbb{E}}_{i-1}\left[\bar b_{i-1}\Delta\mathsf{h}_i\right];i=1,...,n\} $ is asymptotically negligible.
        \end{cor}

        The following lemma is also an extension of the analogous result in 1D.
        
        \begin{lemma}
                \label{lemmaBmm}There exists a $ \mathbb{R}^d$-valued random vector $\theta $ that satisfies the
                control%
                \begin{equation}
                        \sup_{t\in \left[ 0,T\right] }\tilde{\mathbb{E}}\left[ \left\vert \theta
                        _{t}\right\vert \bar{M}_{t}\right] \leq c\sqrt{\Delta }  \label{Bmcontrolm}
                \end{equation}%
                where the constant $c=c\left( T\right) $ is independent of $n$ such that for any $ j,k\in\{1,...,d\} $
                \begin{equation*}
                        \sum_{\left\{ i,t_{i}\leq t\right\} }\left( Z_{i}^j-\tilde{\mathbb{E}}_{i-1}\left[
                        Z_{i}^j\bar{m}_{i}\right] \right)\left( Z_{i}^k-\tilde{\mathbb{E}}_{i-1}\left[
                        Z_{i}^k\bar{m}_{i}\right] \right) =\delta_{jk}t+\theta^{jk} _{t}.
                \end{equation*}
        \end{lemma}
        \begin{proof}
                The proof goes along the proof for the 1D case just by modifying the following definition 
                \begin{align}
                        \theta^{jk} _{i} =&\left( Z^j_{i}-\tilde{\mathbb{E}}_{i-1}\left[ Z^j_{i}\bar{m}_{i}\right]
                        \right)\left( Z^k_{i}-\tilde{\mathbb{E}}_{i-1}\left[ Z^k_{i}\bar{m}_{i}\right]
                        \right) -\delta_{jk}\Delta.  \label{thetaim}
                \end{align}
        \end{proof}

We now analyze the properties of $E^{n}$ and $R^{n}$ under the measure that appears due to the differentiation of the killed process. That is, define
\begin{align*}
        \mathbb{Q}^n(A):=&\tilde{\mathbb{E}}\left[1_A\bar{M}_n\right].
\end{align*} 
        
        As in the one dimensional case, we obtain due to Lemma \ref{lemmaBmm},
        \begin{theorem}\label{th:Rn}
                The sequence $R^{n}$ is a $  \mathbb{Q}^n$ square martingale which converges in distribution to a Brownian motion $W$.
        \end{theorem}

We complete this section with a couple of results concerning the uniform boundedness of derivatives and moments of the processes which appear in \eqref{fprimemH}.

\begin{lemma}
        \label{lem:24}
        There exists a positive constant $ C $ independent of $ n $ and $ i $ such that $ \|{\mathbf{D}}_if_i \|_\infty\leq C $.
\end{lemma}
Now, we prove that the terms gathered in $ \tilde{\varrho}_i $  in the expression \eqref{fprimemH} are negligible.  The proof is similar to the proof of Lemma 24 in \cite{CK}.

\begin{lemma}
        \label{lem:uc}
        For any compact set $ K $
        \begin{align*}
                \lim_{n\to\infty}\sup_{x\in K}   \tilde{\mathbb{E}}_{0,x}\left[\sum_{i=1}^n f_i(\tilde{\varrho}_i)^{\top}  E^{n}_{i-1}\mathcal{K}_i\bar{M}_{i}\right ]=0
        \end{align*}
\end{lemma}
\begin{proof}
        As in the 1D case $ \mathbb{E}_{i-1}\left[(f_i-f_i(L,{X}^2_i,...,X^d_i))\tilde{\rho}_ie^{\kappa_i} \bar{m}_i\right]=O_{i-1}^E(\sqrt{\Delta})$. From here, the result follows from the application of Lemma \ref{lem:35} (c), the uniform boundedness of the first derivative of $ f_i $ and the moment estimates for $ E^n $.
\end{proof}

\section{Proof of the main result}
\label{sec:5}
Let us recall the definition of the various processes involved in the
        approximate representation of the derivative of the killed semigroup. In the list below, we review the discrete time versions of these processes, i.e, we specify the processes only at the times $t_i$, $i=0,1,...,n$. However, in the subsequent analysis and, whenever needed, we will work with their standard imbedding  of the discrete time versions into the piece-wise constant continuous time processes.

        \begin{itemize}
                \item $X^n$ satisfies the recurrence formula 
                \begin{align}
                        X_{i}^{n,x} =&X_{i-1}^{n,x}+\sigma _{i-1}{Z}_{i}  \notag \\
                        =& X_{i-1}^{n,x}+\sigma _{i-1}({Z}_{i}-\widetilde{\mathbb{E}}_{i-1}\left[ {Z}_{i}\bar{%
                                m}_{i}\right]) +\sigma _{i-1}\widetilde{\mathbb{E}}_{i-1}\left[ {Z}_{i}\bar{m}%
                        _{i}\right] \notag \\
                        =&X_{i-1}^{n,x}+\sigma
                        _{i-1}(R_{t_i}^{n}-R_{t_{i-1}}^{n})+(\mathbf{B}_{t_i}^{n}-\mathbf{B} _{t_{i-1}}^{n}),  \label{eq:defbXunderPtilde''}
                \end{align}
                We will show that $X^n$ converges in distribution to a reflected diffusion $%
                Y $ as defined in (\ref{eq:Y}), see also (\ref{reflected}) below.
                
                \item $R^n$ as defined in (\ref{rn}) is the martingale part the process $X^n$. We will show that $R^n$ converges in distribution
                to a Brownian motion.

                \item $\mathbf{B}^n$ as defined in (\ref{fn}) is the predictable part in the
                (discrete) Doob-Meyer decomposition of $X^n$ under the measure ${\mathbb{Q}}^n$. From (\ref{fn}) and \eqref{eq:6.1m} we have the alternative representation
   \begin{align}
        { \mathbf{B} _{t}^{n}  }  :=&\sum_{i=0}^{n-1}1_{[t_i,t_{i+1})}(t)a_{i-1}\mathbf{e}^1\Delta_iB^n
        \label{fn2},
   \end{align}
       where 
       \begin{align}
        \Delta_iB^{n}   =&   2\Delta g_{i-1}^1(\mathtt{X}^1_{i-1})
        -      
        2\bar{\Phi}\left(
        \mathtt{X}_{i-1}^{1,\sigma}
        \right)
        \frac{\mathtt{X}_{i-1}^1}{ {{a^{11}_{i-1}} }},\label{fn3}\\
       B _{t}^{n} =&\sum_{i=0}^{n-1}1_{[t_i,t_{i+1})}(t)\Delta_iB ^{n}. \label{fn4}
       \end{align}
We will prove that $\mathbf{B}^n$ and, respectively, $B^n$ converge in distribution to a pair $(\mathbf{B},B)$ of local time-type processes with finite variation paths and which only increase when $ Y $ is at the boundary so that 
\[
 \mathbf{B}_t=\int_{0}^{t}a(Y_{s})\mathbf{e}^1dB_s.
\]
                
                \item $E^n$ as defined in (\ref{e}) is the derivative of
                the ``stochastic flow" of $X^n$. Under the measure 
                ${\mathbb{Q}}^n$, it has the form 
\[
E^{n}_j:=\prod_{i=1}^je_i, \ \   E^{n}_0:=I 
\]         
where
\[
e_i=I+(\mathbf{e}^1(\mathbf{e}^1)^{\top}-I)
        \mathsf{h}_i+\bar{b}_{i-1}\Delta(1-\mathsf{h}_i)
        +\mathbf{D}\sigma_{i-1} ^{k }
        (R_{t_i}^{n}-R_{t_{i-1}}^{n})^k+
        (\Gamma_{t_i}^{n}-\Gamma_{t_{i-1}}^{n})
        +\pi_{i-1}\mathtt{X}^1_{i-1}\mathsf{h}_i.
\]       
\item $\Gamma^{n}$ as defined in (\ref{gamman}) is one of the
                process driving $E^n$. It satisfies the recurrence formula 
                \begin{align}
                        \Gamma_{t_i}^{n}-\Gamma_{t_{i-1}}^{n} =&({\gamma}_{i}-\widetilde{\mathbb{E}}_{i-1}\left[ {\gamma}_{i}\bar{m}_{i}\right]) +\widetilde{\mathbb{E}}_{i-1}\left[ {\gamma}_{i}\bar{m}%
                        _{i}\right] \notag \\
                        =&\bar\Gamma_{t_i}^{n}-\bar\Gamma_{t_{i-1}}^{n} 
                        +a^{11}_{i-1}\pi_{i-1}({B}_{t_i}^{n}-{B} _{t_{i-1}}^{n})+O_{i-1}^E(\sqrt{\Delta}). \label{eq:defG}
                \end{align}
We have shown that $\bar{\Gamma}^n$ is a martingale that vanishes
                in the limit (see Lemma \ref{lem:14}) and that $\Gamma^{n}$ converges in distribution to a process $\mathbf{\Gamma}$ that given by \[
 \mathbf{\Gamma}_t=\int_{0}^{t}a^{11}(Y_{s})\pi(Y_s)\mathbf{e}^1dB_s.
\]
                Later, we will see that this term does not contribute to the equation satisfied by the limit of $ E^n $.  

\item We will prove that $E^n$ converges in
                distribution to the process $\mathcal{E}=(\mathcal{E}^\ell)_{\ell=1}^d$, $\mathcal{E}^\ell:=(\mathbf{e}^\ell)^{\top}\mathcal{E}$ whose components satisfy equations see also (\ref{el1}) and (\ref{el2}) below. This process will satisfy equation (\ref{eq:E}) after adding the drift back in the equation for the reflected diffusion $Y$ (again, via Girsanov's theorem)   and is denoted by  $\xi$ in equation (\ref{eq:E}).   
                
                \item $\mathcal{K}^n$ as defined in (\ref{kn}) is the
                Radon-Nikodym derivative of the measure $\widetilde{\mathbb{P}} $ (this is the measure
                under which $X^n$ has no drift) with respect to the original measure $\mathbb{P}.$ That is,
\begin{equation}\label{kprimen}
\mathcal{\mathcal{K}}_{t_i}^{n} =\exp \left( \sum_{j=1}^{i}\left( 
(\sigma^{-1}_{j-1})^{k\cdot}\cdot b_{j-1}\left((R_{t_{j}}^{n}-R_{t_{j-1}}^{n})^k+\sigma^{1k}_{j-1}
(B_{t_{j}}^{n}-B_{t_{j-1}}^{n})\right)-\frac{1}{2}\left\|\left((\sigma^{-1}_{j-1})^{k\cdot}\cdot b_{j-1}\right)\mathbf{e}^k\right \|^2\Delta\right) \right) .
\end{equation}                       
We will prove that $\mathcal{K}^n$ converges in distribution to the process $               \mathcal{K}$ defined in (\ref{K}).
        \end{itemize}
We will show that the piece-wise constant version of the vector process
$$(X^{n},E^{n},R^{n},\Gamma^n, \bar\Gamma^n,\mathcal{K}^n,B^n,\bold{B}^n)$$ converges in distribution. We start with a relative compactness result and a convergence result of the processes that drive the various terms appearing in the push-forward formula.       
       
       \begin{theorem}
        \label{relcomptheoremm}The family of processes $(X^{n},E^{n},R^{n},\Gamma^n, \bar\Gamma^n,\mathcal{K}^n,B^n,\bold{B}^n)$ 
        is relatively compact under $\mathbb{Q}_{n}$
       \end{theorem}
   \begin{proof}
   
        The proof of the corresponding statement in the one dimensional case is done considering the control on each component as follows:
          \begin{equation*}
                \sup_{n}\widetilde{\mathbb{E}}\left[ \left\vert q_{t}^{n}\right\vert \bar{M}^n_{n}%
                \right] <\infty ,
        \end{equation*}%
        where $q^{n}$ is replaced by each of the processes in the sequence 
        $(X^{n},E^{n},R^{n},\Gamma^n, \bar\Gamma^n,\mathcal{K}^n,B^n,\bold{B}^n)$ in turn. The control for the processes in
        the triplet $({X}^{n},E^{n},\mathcal{K}^{n})$ follows from Lemmas \ref{lemma:7} and \ref{moments}. The bounds on the moments of $(B^{n})$ follows from Lemma \ref{lem:essb}.
This, in  turn implies the control on $\bold{B}^n$ from the representation
(\ref{fn2}) and the boundedness of $a$. The process $\bar\Gamma^n$ vanishes in the limit (see Lemma \ref{lem:14}) so it is also relative compact. The control on $\Gamma^n$ follows from (\ref{eq:defG}) and the respective control of $\bar\Gamma^n$ and $B^n$. Finally, we have the result for $R^n$ from Theorem \ref{th:Rn},    
         
        As in the one-dimensional case, to complete the relative compactness argument, we apply Theorem 8.6 and Remark 8.7 in Chapter 3 from \cite{EthierKurtz}. In particular, we show that, for all the processes but $E^n$ we find a family $(\gamma _{\delta }^{n})_{0<\delta <1, n\in 
                \mathbb{N}}$ of nonnegative random variables such that  
        \begin{equation}\label{filt}
                \widetilde{\mathbb{E}}\left[ \left\vert q_{t+u}^{n}-q_{t}^{n}\right\vert \bar{m}%
                _{t+u}|\mathcal{F}_{t}^{n}\right] \leq \widetilde{\mathbb{E}}\left[ \gamma
                _{\delta }^{n}\bar{M}^n_{n}|\mathcal{F}_{t}^{n}\right]
        \end{equation}%
        and $\lim_{\delta \rightarrow 0}\limsup_{n\rightarrow \infty }\widetilde{\mathbb{%
                        E}}\left[ \gamma _{\delta }^{n}\bar{M}^n_{n}\right] =0$ for $t\in \lbrack
        0,T],u\in \lbrack 0,\delta ]$, where $q^{n}$ is replaced by $X^{n},R^{n},\Gamma^n, \bar\Gamma^n,\mathcal{K}^n$, $B^n$, and $\bold{B}^n$. The arguments are similar to those in the one-dimensional case and we omit them.   
        
        The control of the oscillation of the processes $ E^n $ is much more delicate and we have to use a different argument. We show that, for arbitrary $0\leq t\leq
T,~0\leq \delta \leq t$, there exists a constant independent of $n,t$ and $%
\delta $ such that 
\begin{equation}
\Xi ^{n,t,\delta }:=\widetilde{\mathbb{E}}\left[ \min (\Vert E_{t+\delta
}^{n}-E_{t}^{n}\Vert ,1)^{4}\min (\Vert E_{t}^{n}-E_{t-\delta }^{n}\Vert
,1)^{4}\bar{m}_{t+\delta }|\mathcal{F}_{t}\right] \leq C\delta ^{2}
\label{maincontrol}
\end{equation}%
This is an application of Theorem 8.8 page 139 in \cite{EthierKurtz}. Note
that we require the control of the fourth moments of the increments as we
need the exponent of $\delta $ on the right hand side of (\ref{maincontrol})
to be strictly larger than 1.

Furthermore, we may assume without loss of generality that $\Delta
\leq 2\delta $. In fact, if $\delta <\frac{\Delta }{2}$ at least one of the two
quantities $\Vert E_{t+\delta }^{n}-E_{t}^{n}\Vert $ or $\Vert
E_{t}^{n}-E_{t-\delta }^{n}\Vert $ is null.

 For $0\leq u<v\leq T$, we define the set 
\begin{equation*}
\Phi ^{n,u,v}=\left\{ \omega \in \Omega |\mathsf{h}_{i}\left( \omega \right)
=0~\mathrm{for~any}~i~\text{\textrm{such~that}}~t_{i}\in \left[ u,v\right]
\right\} 
\end{equation*}
and note that 
\begin{align}
\Xi ^{n,t,\delta } \leq &\widetilde{\mathbb{E}}\left[ \Vert E_{t+\delta
}^{n}-E_{t}^{n}\Vert ^{4}1_{\Phi ^{n,t,t+\delta }}\bar{m}_{t+\delta }\right]
+\widetilde{\mathbb{E}}\left[ \Vert E_{t}^{n}-E_{t-\delta }^{n}\Vert
^{4}1_{\Phi ^{n,t-\delta ,t}}\bar{m}_{t}\right]   \notag \\
&+\widetilde{\mathbb{E}}\left[ \Vert E_{t+\delta }^{n}-E_{t}^{n}\Vert
^{4}1_{\Omega \backslash \left( \Phi ^{n,t-\delta ,t}\cup \Phi
^{n,t,t+\delta }\right) }\bar{m}_{t+\delta }\right]. \label{bib}
\end{align}%
We control each of the three terms on the right hand side of (\ref{bib}):

For the first term, observe that on the set $\Phi ^{n,t,t+\delta }$, \ the
matrices $e_{i}$ have the simpler form (due to the fact that all the $h_{i}$%
's are null in the interval $\left[ t,t+\delta \right] $ by the definition
of $\Phi ^{n,t,t+\delta }$). 
\begin{align}
e_{i}& =I+\bar{b}_{i-1}\Delta +\mathbf{D}\sigma _{i-1}^{k}{Z}_{i}^{k}
\label{eq:R1} \\
\bar{b}_{i-1}:=& \mathbf{D}b_{i-1}-\mathbf{D}\sigma _{i-1}^{k}\left( (\sigma
_{i-1})_{k\cdot }^{-1}\cdot b_{i-1}\right) .  \notag
\end{align}%
Using the fact that the matrices $\bar{b}_{i-1}$ are uniformly bounded and
that $\Delta \leq 2\delta $, one deduces that there exists a constant $C$
independent of $i$ and $n$ such that%
\begin{eqnarray}
\Vert E_{t+\delta }^{n}-E_{t}^{n}\Vert  &\leq &\Delta \left( \left[ \frac{%
\delta }{\Delta }\right] +1\right) C\max_{\left\{ i,t_{i}\in \left[
t,t+\delta \right] \right\} }\left\| E_{i-1}^{n}\right\|+\left\vert
\left\vert \sum_{\left\{ i,t_{i}\in \left[ t,t+\delta \right] \right\}
}\left( \mathbf{D}\sigma _{i-1}^{k}{Z}_{i}^{k}\right) E_{i-1}^{n}\right\vert
\right\vert .  \notag \\
&\leq &\left( \delta +\Delta \right) C\max_{\left\{ i,t_{i}\in \left[
t,t+\delta \right] \right\} }\left\| E_{i-1}^{n}\right\|+\left\| \mathcal{A}^{n,t,t+\delta }+\mathcal{B}^{n,t,t+\delta
}\right\|   \notag \\
&\leq &3C\delta \max_{i=0,...,n}\left\|E_{i}^{n}\right\|
+\left\| \mathcal{A}^{n,t,t+\delta }\right\|
+\left\| \mathcal{B}^{n,t,t+\delta }\right\|
\label{f3}
\end{eqnarray}%
where\footnote{Here and elsewhere $\left[ a\right] $ is the integer part of $a$ and $\left\lceil
a\right\rceil $ is the smallest integer larger than or equal to $a$.} 
\begin{eqnarray*}
\mathcal{A}^{n,t,t+\delta } &=&\sum_{\left\{ i,t_{i}\in \left[ t,t+\delta %
\right] \right\} }\left( \mathbf{D}\sigma _{i-1}^{k}\left( {Z}_{i}^{k}-%
\tilde{\mathbb{E}}_{i-1}\left[ Z_{i}^{k}\bar{m}_{i}\right] \right) \right)
E_{i-1}^{n}=\sum_{i=\left\lceil \frac{t}{\Delta }\right\rceil }^{\left[ 
\frac{t+\delta }{\Delta }\right] }\left( \mathbf{D}\sigma _{i-1}^{k}\left( {Z%
}_{i}^{k}-\tilde{\mathbb{E}}_{i-1}\left[ Z_{i}^{k}\bar{m}_{i}\right] \right)
\right) E_{i-1}^{n} \\
\mathcal{B}^{n,t,t+\delta } &=&\sum_{\left\{ i,t_{i}\in \left[ t,t+\delta %
\right] \right\} }\left( \mathbf{D}\sigma _{i-1}^{k}\tilde{\mathbb{E}}_{i-1}%
\left[ Z_{i}^{k}\bar{m}_{i}\right] \right) E_{i-1}^{n}.
\end{eqnarray*}%
 Since the process 
\begin{equation*}
j\rightarrow \sum_{i=\left\lceil \frac{t}{\Delta }\right\rceil }^{i+j}\left( 
\mathbf{D}\sigma _{i-1}^{k}\left( {Z}_{i}^{k}-\tilde{\mathbb{E}}_{i-1}\left[
Z_{i}^{k}\bar{m}_{i}\right] \right) \right) E_{i-1}^{n}
\end{equation*}%
is a (discrete) martingale, it follows that (see e.g. \cite{Bur}) there
exists a universal constant $M$ such that 
\begin{eqnarray}
\widetilde{\mathbb{E}}\left[ \left\| \mathcal{A}^{n,t,t+\delta
}\right\| ^{4}\bar{m}_{t+\delta }\right]  &\leq &M\widetilde{%
\mathbb{E}}\left[ \left( \sum_{i=\left\lceil \frac{t}{\Delta }\right\rceil
}^{\left[ \frac{t+\delta }{\Delta }\right] }\left\| \left( 
\mathbf{D}\sigma _{i-1}^{k}\left( {Z}_{i}^{k}-\tilde{\mathbb{E}}_{i-1}\left[
Z_{i}^{k}\bar{m}_{i}\right] \right) \right) E_{i-1}^{n}
\right\| ^{2}\right) ^{2}\bar{m}_{t+\delta }\right]   \notag \\
&\leq &c\left( \left[ \frac{t+\delta }{\Delta }\right] -\left\lceil \frac{t}{%
\Delta }\right\rceil \right) \sum_{i=\left\lceil \frac{t}{\Delta }%
\right\rceil }^{\left[ \frac{t+\delta }{\Delta }\right] }\widetilde{\mathbb{E%
}}\left[ \left\|\left( {Z}_{i}-\tilde{\mathbb{E}}_{i-1}\left[
Z_{i}\bar{m}_{i}\right] \right) \right\| ^{4}\left\| E_{i-1}^{n}\right\| ^{4}\bar{m}_{t+\delta }\right] 
\notag \\
&\leq &c\left( \left[ \frac{\delta }{\Delta }\right] +1\right)
\sum_{i=\left\lceil \frac{t}{\Delta }\right\rceil }^{\left[ \frac{t+\delta }{%
\Delta }\right] }\widetilde{\mathbb{E}}\left[ \widetilde{\mathbb{E}}_{i}%
\left[ \left\|\left( {Z}_{i}-\tilde{\mathbb{E}}_{i-1}\left[
Z_{i}\bar{m}_{i}\right] \right) \right\| ^{4}\bar{m}_{i}%
\right] \left\| E_{i-1}^{n}\right\| ^{4}\bar{m}%
_{t+\delta }\right]   \notag \\
&\leq &16c\left( \left[ \frac{\delta }{\Delta }\right] +1\right) \widetilde{%
\mathbb{E}}\left[ \max_{i=0,...,n}\left\|E_{i}^{n}\right\|
^{4}\left( \sum_{i=\left\lceil \frac{t}{\Delta }\right\rceil }^{\left[ \frac{%
t+\delta }{\Delta }\right] }\widetilde{\mathbb{E}}_{i}\left[ \left\|{Z}_{i}\right\| ^{4}\bar{m}_{i}\right] \right) 
\bar{m}_{t+\delta }\right]   \notag \\
&\leq &16c\left( \left[ \frac{\delta }{\Delta }\right] +1\right) ^{2}\Delta
^{2}\widetilde{\mathbb{E}}\left[ \max_{i=0,...,n}\left\|E_{i}^{n}\right\|^{4}\bar{m}_{t+\delta }\right]   \notag \\
&\leq &C\delta ^{2}  \label{ff1}
\end{eqnarray}%
where we used the fact that $\mathbf{D}\sigma _{i-1}^{k}$ is uniformly
bounded, the control on the fourth moment of $\max_{i=0,...,n}\left\vert
|E_{i}^{n}|\right\vert $ (Lemma \ref{lemma:7}) and Lemma \ref{lem:5m}. Next, we have that%
\begin{equation*}
\left\| \mathcal{B}^{n,t,t+\delta }\right\|
\leq C\max_{i=0,...,n}\left\| E_{i}^{n}\right\|
\max_{j\leq n-k-1}\left( \sum_{i=\left\lceil \frac{t}{\Delta }\right\rceil
}^{\left[ \frac{t+\delta }{\Delta }\right] }\left\| \tilde{%
\mathbb{E}}_{i-1}\left[ {Z}_{i}\bar{m}_{i}\right] \right\|\right)
\end{equation*}%
from which we deduce that 
\begin{equation}
\widetilde{\mathbb{E}}\left[ \left\| \mathcal{B}^{n,t,t+\delta
}\right\| ^{4}\bar{m}_{t+\delta }\right] \leq C\delta ^{2}
\label{fff2}
\end{equation}%
again by using the control on the fourth moment of $\max_{i=0,...,n}\left%
\|E_{i}^{n}\right\| $ (Lemma \ref{lemma:7}) and Lemma \ref{lem:5m}. The control on the first term
on the right hand side of (\ref{bib}) follows from (\ref{f3}), (\ref{ff1})
and (\ref{fff2}).

The control on second term on the right hand side of (\ref{bib}) is done
identically as with the first term of  (\ref{bib}) and we omit it here.

We treat now the third term of (\ref{bib}). On $\Omega \backslash \left(
\Phi ^{n,t-\delta ,t}\cup \Phi ^{n,t,t+\delta }\right) $, there exists at
least one $t_{i}\in \left[ t-\delta ,t\right] $ such that $\mathsf{h}%
_{i} =1_{(U_{i}\leq p_{i})}=1$.\ Let us denote by  $\bar{\rho}_t^{n}:=\sup\{t_i<t;U_i\le p_i\} $  and by $i_{last}^{t}:= n\bar{\rho}_{t}^n  $ the largest $%
i$ such that $t_{i}\in \left[ t-\delta ,t\right] $ and $\mathsf{h}_{i}\left(
\omega \right) =1$ and we similarly define $i_{last}^{t+\delta }$. Therefore  
\begin{eqnarray*}
E_{t_{i_{last}^{t}}}^{n}-\mathbf{e}^{1}(\mathbf{e}^{1})^{\top }%
E_{t_{i_{last}^{t}-1}}^{n} &=&\bar{e}%
_{i_{last}^{t}}E_{t_{i_{last}^{t}-1}}^{n} \\
E_{t_{i_{last}^{t+\delta }}}^{n}-\mathbf{e}^{1}(\mathbf{e}^{1})^{{\top }%
}E_{t_{i_{last}^{t+\delta }-1}}^{n} &=&\bar{e}_{i_{last}^{t+\delta
}}E_{t_{i_{last}^{t+\delta }-1}}^{n}.
\end{eqnarray*}%
Using these relations, we can estimate
\begin{eqnarray}
\left\| E_{t+\delta }^{n}-E_{t}^{n}\right\|
&\leq &\left\| E_{t+\delta }^{n}-E_{t_{i_{last}^{t+\delta
}}}^{n}\right\| +\left\|
E_{t}^{n}-E_{t_{i_{last}^{t}}}^{n}\right\| +\left\| E_{t_{i_{last}^{t+\delta }}}^{n}-\mathbf{e}^{1}(\mathbf{e}^{1})^{{\top }%
}E_{t_{i_{last}^{t+\delta }-1}}^{n}\right\| \notag \\
&&+\left\| E_{t_{i_{last}^{t}}}^{n}-\mathbf{e}^{1}(\mathbf{e}^{1})^{{%
\ast }}E_{t_{i_{last}^{t}-1}}^{n}\right\|+\left\| \mathbf{e}^{1}(\mathbf{e}^{1})^{{\top }}E_{t_{i_{last}^{t+\delta
}-1}}^{n}-\mathbf{e}^{1}(\mathbf{e}^{1})^{{\top }}E_{t_{i_{last}^{t}-1}}^{n}\right%
\|  \notag \\
&=&\left\| E_{t+\delta }^{n}-E_{t_{i_{last}^{t+\delta
}}}^{n}\right\| +\left\|
E_{t}^{n}-E_{t_{i_{last}^{t}}}^{n}\right\| +\left\| \bar{e}_{i_{last}^{t}}E_{t_{i_{last}^{t}-1}}^{n}\right\| +\left\| \bar{e}_{i_{last}^{t+\delta
}}E_{t_{i_{last}^{t+\delta }-1}}^{n}\right\|  \notag \\
&&+\left\| \mathbf{e}^{1}(\mathbf{e}^{1})^{{\top }%
}E_{t_{i_{last}^{t+\delta }-1}}^{n}-\mathbf{e}^{1}(\mathbf{e}^{1})^{{\top }%
}E_{t_{i_{last}^{t}-1}}^{n}\right\|. \label{coco}
\end{eqnarray}%
The first two terms in (\ref{coco}) are controlled in the same way as the
first two terms on the right hand side of (\ref{bib}): Recall that there are
no $t_{i}$'s in the intervals $\left[ t_{i_{last}^{t}},t\right] $ and $\left[
t_{i_{last}^{t+\delta }},t+\delta \right] $ such that $\mathsf{h}_{i}\left(
\omega \right) =1$.

The term $\bar{e}_{i_{last}^{t}}$( $\bar{e}_{i_{last}^{t+\delta }}$ is
treated in a similar way) takes the form (recall that $\mathsf{h}%
_{i_{last}^{t}}\left( \omega \right) =1$ on\\ $\Omega \backslash \left( \Phi
^{n,t-\delta ,t}\cup \Phi ^{n,t,t+\delta }\right) $)   
\begin{equation*}
\bar{e}_{i_{last}^{t}}=\mathtt{X}_{i_{last}^{t}-1}^{1}\sigma
_{i_{last}^{t}-1}^{k}{\mathbf{D}}_{i_{last}^{t}-1}\left( \frac{\sigma
_{i_{last}^{t}-1}^{k1}}{{a_{i_{last}^{t}-1}^{11}}}\right) \mathsf{h}%
_{i_{last}^{t}}+{\pi }_{i_{last}^{t}-1}\mathtt{X}_{i_{last}^{t}-1}^{1}%
\mathsf{h}_{i_{last}^{t}}.
\end{equation*}%
Therefore 
\begin{eqnarray*}
\widetilde{\mathbb{E}}\left[ \left\| \bar{e}%
_{i_{last}^{t}}E_{t_{i_{last}^{t}-1}}^{n}\right\| ^{4}\bar{m}%
_{t+\delta }\right]  &\leq &C\widetilde{\mathbb{E}}\left[ \max_{j=1,..,n}%
\left\| \mathtt{X}_{j-1}^{1}\right\|^{4}%
\widetilde{\mathbb{E}}_{j-1}\left[ \mathsf{h}_{j}\bar{m}_{j}\right]
\left\| E_{j-1}^{n}\right\| ^{4}\bar{M}_{n}^{n}%
\right]  \\
&\leq &C\sqrt{\widetilde{\mathbb{E}}\left[ \max_{j=1,..,n}\left\| \mathtt{X}_{j-1}^{1}\right\|^{8}\widetilde{%
\mathbb{E}}_{j-1}\left[ \mathsf{h}_{j}\bar{m}_{j}\right] ^{2}\bar{M}_{n}^{n}%
\right] \widetilde{\mathbb{E}}\left[ \max_{j=1,..,n}\left\|
E_{j-1}^{n}\right\| ^{8}\bar{M}_{n}^{n}\right] }
\end{eqnarray*}%
which gives the result following the control on the eighth moment of $%
\max_{j=1,..,n}\left\| E_{j-1}^{n}\right\| $
(Lemma \ref{lemma:7} and the fact that 
\begin{equation*}
\widetilde{\mathbb{E}}\left[ \max_{j=1,..,n}\left\| \mathtt{X}%
_{j-1}^{1}\right\| ^{8}\widetilde{\mathbb{E}}_{j-1}\left[ 
\mathsf{h}_{j}\bar{m}_{j}\right] ^{2}\bar{M}_{n}^{n}\right] \leq c\Delta ^{4}.
\end{equation*}%
The above inequality follows because $ \widetilde{\mathbb{E}}_{j-1}\left[ 
\mathsf{h}_{j}\bar{m}_{j}\right]=2\bar{\Phi}\left(
{\mathtt{X}_{j-1}^{1,\sigma}}\right)\leq Ce^{-\frac{(\mathtt{X}_{j-1}^{1,\sigma})^2}2} $ (see the proof of Lemma \ref{lem:4}) and the fact that there exists $ C>0 $ such that the inequality $ x^8e^{-\frac{x^2}2}\leq C $ is valid for $ x\geq 0 $.
The term $\bar{e}_{i_{last}^{t+\delta }}E_{t_{i_{last}^{t+\delta }-1}}^{n}$
is treated in a similar manner. 

Finally, observe that we have the following decompositions
\begin{align*}
\mathbf{e}^{1}(\mathbf{e}^{1})^{{\top }}E_{t_{i_{last}^{t+\delta }-1}}^{n}-\mathbf{e}%
^{1}(\mathbf{e}^{1})^{{\top }}E_{t_{i_{last}^{t}-1}}^{n}
=&\sum_{i=i_{last}^{t}}^{i_{last}^{t+\delta }-1}\mathbf{e}^{1}(\mathbf{e}^{1})^{{\top }}\left(
E_{t_{i}}^{n}-E_{t_{i-1}}^{n}\right),\\
\mathbf{e}^{1}(\mathbf{e}^{1})^{{\top }}\left(
E_{t_{i}}^{n}-E_{t_{i-1}}^{n}\right) =&\mathbf{e}^{1}(\mathbf{e}^{1})^{{\top }%
}\left( e_{i}-I\right) E_{t_{i-1}}^{n}=\mathbf{e}^{1}(\mathbf{e}^{1})^{{\top }}%
\bar{e}_{i}E_{t_{i-1}}^{n}.
\end{align*}
Therefore,
\begin{align*}
\mathbf{e}^{1}(\mathbf{e}^{1})^{{\top }}E_{t_{i_{last}^{t+\delta }-1}}^{n}-\mathbf{e}%
^{1}(\mathbf{e}^{1})^{{\top }}E_{t_{i_{last}^{t}-1}}^{n}=&\mathcal{C}^{n,t}+\mathcal{D}^{n,t}
\end{align*}%
where%
\begin{equation*}
\mathcal{C}^{n,t}=\sum_{j=i_{last}^{t}}^{i_{last}^{t+\delta }-1}\mathbf{e}^{1}(%
\mathbf{e}^{1})^{{\top }}\widetilde{\mathbb{E}}_{i-1}\left[ \bar{e}_{i}\bar{m}%
_{i}\right] E_{t_{i-1}}^{n},~~~\mathcal{D}^{n,t}=%
\sum_{j=i_{last}^{t}}^{i_{last}^{t+\delta }-1}\mathbf{e}^{1}(\mathbf{e}^{1})^{{%
\ast }}\left( \bar{e}_{i}-\widetilde{\mathbb{E}}_{i-1}\left[ \bar{e}_{i}\bar{%
m}_{i}\right] \right) E_{t_{i-1}}^{n}.
\end{equation*}
From here, the analysis follows in the similar manner as that for controlling the
first term on the right hand side of (\ref{bib}). 
                \end{proof}
  
The next convergence theorem is similar in many aspects to the one dimensional situation except for the term $ (\mathbf{e}^1(\mathbf{e}^1)^{\top}-I)\mathsf{h}_i $ which appears in the definition of $ e_i $ and requires a particular argument in the proof to be provided later. This will be resolved by considering separately the process $ E^n $ in different components. That is, $  E^{n,\ell}:=(\mathbf{e}^\ell )^{\top} E^n$ for $ \ell=1 $ will be analyzed following the same arguments as in the one dimensional case. The cases $ \ell\neq 1 $ will require a separate argument that takes into consideration the heuristic fact that this process jumps to zero when $ X^n $ touches the boundary.

In the arguments that follow we will frequently use the convergence of approximations to stopping times or last visit times. For this reason, as in the one dimensional case, we need to describe an equivalent to all the indicator functions of $ U_i $, $ i\in\mathbb{N} $, using instead of the process $ X^n  $ and its  continuous version $ {X}^{n,x} $, the process $ \mathcal{X}^{n,x} $ described in Section \ref{sec:mlt}.

        \begin{theorem}
                \label{bigT}
                The family of processes $(X^{n},E^{n},R^{n},\Gamma^n, \bar\Gamma^n,\mathcal{K}^n,B^n,\bold{B}^n)$ converge in distribution under $\mathbb{Q}^{n}$ to the processes $(Y,\mathcal{E},W,\bold{\Gamma},0,\mathcal{K},B,\bold{B}),$  where
 \begin{itemize} 
 \item The process  $ W $ is a standard
                Brownian motion 
 \item  The   process $Y$ is the solution of the normally reflected equation in the domain $H^d_L$
                \begin{align}
                                Y_{t}=&x+\int_{0}^{t}\sigma (Y_{s})dW_{s}+\bold{B}_t\in H^d_L. \label{reflected}\\
                                \bold{B}_t=&\mathbf{e}^1 \int_0^t1_{(Y^1_s=L)}d|\boldsymbol B|_s\nonumber
                \end{align}
  \item The process $\mathbf{\Gamma}$ satisfies
  \[
 \mathbf{\Gamma}_t=\int_{0}^{t}a^{11}(Y_{s})\pi(Y_s)\mathbf{e}^1dB_s.
\]              
  \item The one-dimensional process $B$ is defined as 
  \[
  B_t=\int_0^ta^{11}(Y_s)^{-1}\mathbf{e}^1\cdot d\bold{B}_s.              \]         
  \item       The ``change of measure'' $ \mathcal{K} $ is given by 
        \begin{align}\label{K}
                \mathcal{K}_t=\exp
                \left(\int_0^tb^{\top}  a^{-1}(Y_s)dY_s-\frac 12\int_0^t\left(b^\top a^{-1}b\right)(Y_s)ds\right).
        \end{align}
       
        Furthermore, $\mathcal{E}^\ell=(\mathbf{e}^\ell)^{\top}\mathcal{E}$ satisfies a linear equation which can be described as follows: For $ s<t $
        \begin{align}\label{el1}
                \mathcal{E}^1_t=\mathcal{E}^1_s+\int_s^t    \left((\mathbf{e}^1)^{\top} \mathbf{D}\bar{b}(Y_u)du+(\mathbf{e}^1)^{\top} \mathbf{D}\sigma^k(Y_u) dW_u^{k} \right)\mathcal{E}_{u}.
        \end{align}
    For any  $t>0 $ and $ \ell\neq 1 $
    \begin{align}\label{el2}
        \mathcal{E}^\ell_t=(\mathbf{e}^\ell)^{\top}1_{(\rho_t=-\infty)}+(\mathbf{e}^\ell)^{\top}\mathcal{E}^\ell_{\rho_t-}1_{(\rho_t\in [0,t])}
        +\int_{\rho_t\vee 0}^t    \left((\mathbf{e}^\ell)^{\top} \mathbf{D}\bar{b}(Y_u)du+(\mathbf{e}^\ell )^{\top} \mathbf{D}\sigma^k(Y_u) dW_u^{k} 
        \right)\mathcal{E}_{u}
    \end{align}
Here, $ \bar{b}=\mathbf{D} b- \mathbf{D}\sigma^{k} (\sigma^{-1}_{k\cdot}\cdot b)$, and
\begin{equation}\label{rhot}
                {\rho}_t:=\sup\{s\leq t;Y^1_s=L\}.
\end{equation}

\end{itemize}
The system \eqref{el1}-\eqref{el2} has a unique solution and it satisfies that $ \mathcal{E}_{s,t} \mathcal{E}_{s} =\mathcal{E}_{t}$ where $   \mathcal{E}_{s,t}$ denotes the solution of the system \eqref{el1}-\eqref{el2} started at $ s\in (0,t) $. That is, for $ \rho_t(s):=\sup\{u\in [s,t);Y_u^1=L\} $ it solves the system
\begin{align*}
\mathcal{E}_{s,t}=&I1_{(\rho_t(s)=-\infty)}+1_{(\rho_t(s)\in [s,t])}\mathbf{e}^1(\mathbf{e}^1)^{\top} 
\mathcal{E}_{\rho_t(s)-}+\int_{\rho_t(s)}^t
\left( \mathbf{D}\bar{b}(Y_u)du+ \mathbf{D}\sigma^k(Y_u) dW_u^{k} 
\right)\mathcal{E}_{s,u}.
\end{align*}

        \end{theorem}
        
        \begin{proof}
                From Theorem \ref{relcomptheoremm},  the laws of family of processes $(X^{n},E^{n},R^{n},\Gamma^n, \bar\Gamma^n,\mathcal{K}^n,B^n,\bold{B}^n)$ are relatively compact. We extract a convergent (in law) subsequence from $(X^{n},E^{n},R^{n},\Gamma^n, \bar\Gamma^n,\mathcal{K}^n,B^n,\bold{B}^n)$ which we re-index.  As in the one-dimensional case, the identification of the limits follows from the individual representations and the application of the Corollary 5.6 in \cite{KurtzProtter}.  

Let us start with the convergence of the processes $(X^{n},R^{n},B^n,\bold{B}^n)$
to the quadruplet $(Y,W,B,\bold{B})$. From (\ref{eq:defbXunderPtilde''}) we can write, for  $t\in [t_{i},t_{i+1})$ 
                \begin{align*}
                                X_{t}^{n} =&X_{t_i}^{n}\notag\\
                                =&X_{i-1}^{n,x}+\sigma
                        _{i-1}(R_{t_i}^{n}-R_{t_{i-1}}^{n})+(\mathbf{B}_{t_i}^{n}-\mathbf{B} _{t_{i-1}}^{n})\\
                                =&x_{0}+\sum_{j=0}^{i-1}\left(\sigma (X_{t_{j}}^{n}) (R^n_{t_{j+1}}-R_{t_{j}}^{n})+(\mathbf{B}_{t_i}^{n}-\mathbf{B} _{t_{i-1}}^{n})\right),\notag\\
                      =&x_{0}+\int_{0}^{t}\sigma ^{n}(X^{n},s)dR_{s}^{n}+\mathbf{B}_{t}^{n},\notag\\
                \end{align*}
                where 
                \[
                        \sigma ^{n}(X^{n},s) =\sigma (X_{t_{j}}^{n}),~~~~~s\in \lbrack
                        t_{j}^{n},t_{j+1}^{n}).  
                \]
From (\ref{fn2}), (\ref{fn3}) and (\ref{fn4}), we have
   \begin{align}
        { \mathbf{B} _{t}^{n}  }  :=&\sum_{i=0}^{n-1}1_{[t_i,t_{i+1})}(t)a_{i-1}\mathbf{e}^1\Delta_iB^n
        ,\\
        =&\sum_{i=0}^{n-1}1_{[t_i,t_{i+1})}(t)\sum_{\ell=1}^da^{\ell 1 }_{i-1}\mathbf{e}^\ell\Delta_iB^n
        ,\\
        \Delta_iB^{n}   =&   2\Delta g_{i-1}^1(\mathtt{X}^1_{i-1})
        -      
        2\bar{\Phi}\left(
        \mathtt{X}_{i-1}^{1,\sigma}
        \right)
        \frac{\mathtt{X}_{i-1}^1}{ {{a^{11}_{i-1}} }},\\
       B _{t}^{n} =&\sum_{i=0}^{n-1}1_{[t_i,t_{i+1})}(t)\Delta_iB ^{n}. 
       \end{align}
In the above identities we used the piecewise constant versions of the processes involved. By applying Corollary 5.6 in \cite{KurtzProtter} and the fact that $(X^{n},R^{n},B^n,\bold{B}^n)$
converges to the quadruplet $(Y,W,B,\bold{B}) $, we
                deduce that $Y$ satisfies the equation \eqref{reflected}, where $B$ is an
                increasing process.

To   show that the process $Y$ is the solution of the normally reflected equation (\ref{reflected}), we deduce that $Y_{t}^1\geq L$ , $t\geq 0$, and that $\int_{0}^{t}1_{(Y_s^{1}>L)} dB_{s}=0,t\geq 0$. Then,  the result will follow in a similar fashion as in the proof of  Theorem 1.2.1 in \cite{ap}. We note first that, since the piecewise constant Euler approximation $X^{n}$ converges in distribution, so does its continuous version $X^{c,n}$ (to the same limit) and the two approximations coincide at partition times.  We can therefore replace $X^{n}$ by $X^{c,n}$ in all identities, if needed. 
               Under $ \mathbb{Q}^n $ the law of  $ X^{c,n}$, on the interval $[t_{i-1},t_i] $ is that of a process reflected when it reaches the boundary
$ H_L^d $ . More precisely, the law of the first coordinate of $ X^{c,n}$, on the interval $[t_{i-1},t_i] $, is that of the one dimensional process \begin{align}\label{calxn}
        \bar{\mathcal{X}}^{n,1}_t=(X^n_{i-1})^1+\sigma_{i-1}(W_t- W_{t_{i-1}})+\Lambda_{t_{i-1}\mapsto t}.
\end{align}
With a slight abuse of notation, here $W$ is a Wiener process under $ \mathbb{Q}^n$ and 
$\Lambda_{t_{i-1}\mapsto t}$ is an increasing (finite variation) process such that  
\begin{align*}
\Lambda_{t_{i-1}\mapsto t}=&\int_{t_{i-1}}^t1_{(\mathcal{X}^1_s=L)}d|\Lambda^{n,x}|_{ s}  
\end{align*}
that only increases when $\bar{\mathcal{X}}^{n,1}$ hits $L$. Alternatively, consider 
        $ \mathcal{X}^{n} $ to be the reflected process defined in Section \ref{sec:mlt}.
Then $\bar{\mathcal{X}}^{n,1}$ is the first component process of the process $ \mathcal{X}^{n} $ conditioned to start at $X^n_{i-1}$ at time $t_{i-1}$ and run on the interval $[t_{i-1},t_i] $.

We make use a slight modification of the function $ \varphi_m $ used in one dimension, i.e., 
                \begin{align*}
                        \varphi_m(x)=\min(1,-m(x^1-L)) I_{(x^1\le  L)}.
                \end{align*}
   We deduce that, since the support of $\mathcal{X}_{t}^{n}$ is in the set $[L,\infty)$ 
                \begin{eqnarray*}
                \mathbb{P}\left( Y_{t}\leq L-\frac{1}{m}\right) &\leq& \widetilde{\mathbb{E}}\left[
                \varphi _{m}\left( Y_{t}\right) \right] =
                \lim_{n\mapsto \infty }\widetilde{%
                        \mathbb{E}}\left[ \varphi _{m}\left( X_{t}^{n}\right) \bar{M}_{n}^{n}\right]\\
&=&\lim_{n\mapsto \infty }\widetilde{%
                        \mathbb{E}}\left[ \varphi _{m}\left( X_{t}^{c,n}\right) \bar{M}_{n}^{n}\right]=
\lim_{n\mapsto \infty }\widetilde{%
                        \mathbb{E}}\left[ \varphi _{m}\left( \mathcal{X}_{t}^{n}\right) \right]
                =0.
                \end{eqnarray*}%
     Hence $Y_{t}^1\geq L$ for any $t\geq 0$.      Similarly, we use    $\psi _{m}$ defined as the positive continuous function with support in the
                interval $(L+\frac{1}{2m},\infty )$ given by 
                \begin{equation*}
                \psi _{m}\left( x\right) =\min \left( 1,m\left( x^{1}-L-\frac{1}{2m}\right)
                \right)1_{(x^{1}\geq L+\frac{1}{2m})}
                \end{equation*}%
         to deduce  that, almost surely, 
                \begin{equation*}
                \int_{0}^{T}\psi _{m}\left( Y_{s}\right) dB_{s}=0.
                \end{equation*}%
In fact,
from (\ref{fn2}), (\ref{fn3}) and (\ref{fn4})
we have that 
\begin{align*}
                        \mathbb{E}\left[ \int_{0}^{T}\psi _{m}\left( Y_{s}\right) dB_{s}\right]
                =&\lim_{n\mapsto \infty }\widetilde{\mathbb{E}}\left[ \left(
                \sum_{i=0}^{n-1}\psi _{m}\left( X_{i-1}^{n}\right)
                a_{i-1}(B_{t_{i}}^{n}-B_{t_{i-1}}^{n})\right) \bar{M}_{n}^{n}\right] \\
                \le&\lim_{n\mapsto \infty } 2\|a\|_\infty\Delta\widetilde{\mathbb{E}}\left[ \left(
                \sum_{i=0}^{n-1}\psi _{m}\left( X_{i-1}^{n}\right) g_{i}(X_{i-1}^{1,L})\right) \bar{M}_{n}^{n}\right] \\
                \leq &2\|a\|_\infty\Delta \widetilde{\mathbb{E}}\left[ \left(
                \sum_{i=0}^{n-1}\psi _{m}\left( X_{i-1}^{n}\right) g_{i}(X_{i-1}^{L})\right) 
                \bar{M}_{n}^{n}\right] \\
                \leq &C\exp \left( -\frac{1}{16\| a\|_\infty \Delta m^{2}}\right) \sqrt{\Delta }\widetilde{\mathbb{E}}\left[
                \left( \sum_{i=0}^{n-1}\exp \left( -\frac{\left( X^1_{i-1}-L\right) ^{2}}{%
                        4a_{i}\Delta }\right) \right) \bar{M}_{n}^{n}\right] \\
                \leq &C\lim_{n\mapsto \infty }\exp \left( -\frac{n}{16\| a\|_\infty  Tm^{2}}\right) =0.
                \end{align*}

               Therefore,
                \begin{equation*}
                0\leq \int_{0}^{t}1_{(Y_s>L)}dB_{s}=
                \lim_{m\rightarrow \infty }\int_{0}^{t}\psi _{m}\left( Y_{s}\right) dB_{s}=0.
                \end{equation*}
        Let us turn now to study the limit the pair $(\Gamma^n, \bar\Gamma^n)$. From the recurrence formula (\ref{eq:defG}) we get that   $t\in [t_{i},t_{i+1})$ 
                \begin{align*}
                                \Gamma_{t}^{n} =&\Gamma_{t_i}^{n}\notag\\
                                =&\bar\Gamma_{t_i}^{n} +\int_{0}^{t}\bar\sigma ^{n}(X^{n},s)dB_{s}^{n}+O_{i-1}^E(\sqrt{\Delta}),\notag
                \end{align*}
                where 
                \[
                        \bar\sigma ^{n}(X^{n},s) =a^{11} (X_{t_{j}}^{n})\pi(X_{t_{j}}^{n})\mathbf{e}^1,~~~~~s\in \lbrack
                        t_{j}^{n},t_{j+1}^{n}).  
                \]                
We have shown that $\bar{\Gamma}^n$ is a martingale that vanishes
                in the limit (see Lemma \ref{lem:14}). Therefore, again by applying Corollary 5.6 in \cite{KurtzProtter} we deduce that $\Gamma^{n}$ converges in distribution to a process $\mathbf{\Gamma}$ that given by \[
 \mathbf{\Gamma}_t=\int_{0}^{t}a^{11}(Y_{s})\pi(Y_s)\mathbf{e}^1dB_s.
\]
            Finally note that the process     $\Gamma^{n}$ is one of the 
                processes driving $E^n$. 
 We show below that this term does not contribute to the equation satisfied by the limit of $ E^n $. Therefore $\Gamma$  does not appear in the equation satisfied by $\mathcal{E}$.
                
Also, let us note that the law of the processes $(Y,\mathcal{E},W,\bold{\Gamma},0,\mathcal{K},B,\bold{B})$
                are uniquely identified by the corresponding representation, the pathwise uniqueness of 
                equation (\ref{reflected}) and the fact that $W$ is a Brownian motion. This together with the
                convergence along subsequences, gives us the convergence in law of the whole
                sequence $(X^{n},R^{n},\Gamma^n, \bar\Gamma^n,\mathcal{K}^n,B^n,\bold{B}^n)$ under $%
                {\mathbb{Q}}^n$ to the process $(Y,W,\bold{\Gamma},0,\mathcal{K},B,\bold{B})$. 

The same arguments apply to the process $(\bold{e}^1)^{\top}E^{n}$ which converges to $ \mathcal{E}^1 $ as defined in (\ref{el1}). The only new limit process that requires an extra proof of convergence is $ \mathcal{E}^\ell := (\bold{e}^\ell)^{\top}E^n $ for $ \ell\neq 1 $.  To characterize the limit of $ \mathcal{E}^\ell $, we will use the Skorohod representation which allows us to assume that 
\[
\Xi ^n :=(X^{n},E^{n},R^{n},\Gamma^n, \bar\Gamma^n,\mathcal{K}^n,B^n,\mathbf{B}^n)
\] 
converges almost surely in the Skorohod topology to the process 
\[
\Xi :=(Y,\mathcal{E},W,\mathbf{\Gamma},0,\mathcal{K},B,\mathbf{B}),
\] 
where $ X^{n}$ is not the piecewise constant Euler approximation but the continuous paths Euler approximation $ X^{c,n}$. We can toggle freely between the two as they have the same limit in distribution. As stated above,
under $ \mathbb{Q}^n $ the law of  $ X^{c,n}$, on the interval $[t_{i-1},t_i] $ is that of a process reflected when it reaches the boundary
$ H_L^d $ . More precisely, the law of the first coordinate of $ X^{c,n}$, on the interval $[t_{i-1},t_i] $, is that of the one dimensional process $\mathcal{X}^{n,1}$ introduced in (\ref{calxn}). Alternatively, consider 
        $ \mathcal{X}^{n} $ the reflected process defined in Section \ref{sec:mlt}.
        Then using Lemma \ref{th:9}, we obtain
        \begin{align}
                \label{eq:fiA}
                \tilde{\mathbb{E}}_{i-1}[f(X^{c,n}_{i})\bar{m}_{i}]=&\tilde{\mathbb{E}}_{i-1}[
        \tilde{\mathbb{E}}_{i-1}[f(X^{c,n}_{i})/X_{i}^{c,n}\cdot\mathbf{e}^1]\bar{m}_{i}]\\
                =&
                \mathbb{E}_{i-1}[f(\mathcal{X}^{n}_{t_i})/\mathcal{X}^n_{t_{i-1}}=X_{i-1}].\nonumber
        \end{align}
       
    Therefore, we can include the process $\mathcal{X}^{n}  $ in the Skorohod representation if required. 
For example, recall that  $\bar{\rho}_t^{n}:=\sup\{t_i\le t;U_i\le p_i\} $, therefore alternatively, for fixed $ t>0 $, following from (\ref{eq:fiA}) we have that 
                \begin{equation}\label{rhobant}
                        \bar{\rho}_t^n=\sup\{t_{i}\leq t;\text{ there exists }s\in[t_{i-1},t_i],\ \mathcal{X}^{n,1}_s=L \}.
                \end{equation}   
Since $X^n$ converges pathwise to $Y$, it follows that the random times $\bar\rho_t^n$
converge almost surely to the times $\rho_t$ as introduced in (\ref{rhot}), see Section \ref{sec:8} for details. This is the last ingredient required to complete the characterization of the limiting process  $\mathcal{E}$. 

                Under the measure  ${\mathbb{Q}}^n$, $E^n$ has the form 
\[
E^{n}_j:=\prod_{i=1}^je_i, \ \   E^{n}_0:=I 
\]         
where
\begin{eqnarray*}
e_i&=&I+(\mathbf{e}^1(\mathbf{e}^1)^{\top}-I)
        \mathsf{h}_i+\bar{e}_{i}\\
\bar{e}_{i}&=& \bar{b}_{i-1}\Delta(1-\mathsf{h}_i)+\mathbf{D}\sigma_{i-1} ^{k }  (R_{t_i}^{n}-R_{t_{i-1}}^{n})^k+
        (\Gamma_{t_i}^{n}-\Gamma_{t_{i-1}}^{n})
        +\pi_{i-1}\mathtt{X}^1_{i-1}\mathsf{h}_i        
\end{eqnarray*}
Since $(\mathbf{e}^1)^{\top} (\mathbf{e}^1(\mathbf{e}^1)^{\top}-I) = 0$ and $ (\mathbf{e}^1)^{\top}\pi_{i-1}=0 $ we deduce, in the case of $ E^{n,1} :=(\mathbf{e}^1)^{\top} E^n$, the following difference equation:
                \begin{align*}
                        E^{n,1}_{t_i} &:=(\mathbf{e}^1)^{\top} E^n_{t_i}\\
                        &=(\mathbf{e}^1)^{\top} E^{n}_{t_{j-1}}+
                        \sum_{j=1}
                        ^i(\mathbf{e}^1)^{\top}\bar{e}_jE^{n}_{t_{j-1}}    ,
                        \end{align*}
      from which, by induction one proves that 
\begin{align*}
        E^{n,1}_{t}=&E^{n,1}_{0}+\sum_{j=1}^{\lceil \frac{t}{\Delta}\rceil}
        (\mathbf{e}^1)^{\top}\bar{e}_iE^{n}_{j-1}\\
         =&(\mathbf{e}^1)^{\top} +U^n_t+\int_0^t \widetilde{b}^n(E^{n},X^{n},s)ds +\int_{0}^{t}\widetilde{\sigma}^{n,k}(E^{n},X^{n},s)d(R_{s}^{n})^k+%
                        \int_{0}^{t}(\mathbf{e}^1)^{\top} d\Gamma _{s}^{n}E_{s}^{n}. 
                \end{align*}
                Here we have used  the following equalities
                \begin{align*}
                	\widetilde{\sigma}^{n,k}(E^{n},X^{n},s) =&\mathbf{(\mathbf{e}^1)^{\top}D}\sigma_{j}(X_{t_{j}}^{n})E_{t_{j}}^{n},~~~~~s\in \lbrack
                	t_{j}^{n},t_{j+1}^{n}) \\
                	\widetilde{b}^{n}(E^{n},X^{n},s) =&(\mathbf{e}^1)^{\top}\bar{b}_j ^{
                	}(X_{t_{j}}^{n})E_{t_{j}}^{n}, ~~~~~s\in \lbrack t_{j}^{n},t_{j+1}^{n})\\
                	U^n_t =&- \sum_{i=1}^{\lceil \frac{t}{\Delta}\rceil}\bar{b}_{i-1}\Delta\mathsf{h}_iE^n_{t_{i-1}}.
                \end{align*}
               Note that due to the definition of $ \pi $ and \eqref{eq:defG}, we have that $$\lim_{n\to\infty}  \int_{0}^{t}(\mathbf{e}^1)^{\top} d\Gamma _{s}^{n}E_{s}^{n}=0.$$

Therefore the result, \eqref{el1},  for the component $E^{n,1}_{t}$ follows from \cite{KP} using Corollary \ref{cor16}.    
        Again in the above identities we have used the piecewise constant versions of the processes involved.

   For the analysis of $E^{n,\ell}$ we proceed as follows. We define the set 
\begin{equation*}
\Phi ^{n,t}=\left\{ \omega \in \Omega |\mathsf{h}_{i}\left( \omega \right)
=0~\mathrm{for~any}~i~\text{\textrm{such~that}}~t_{i}\in [0,t]
\right\}. 
\end{equation*}      
 Again recall that 
        as in the 1D case, we can describe an equivalent event to all the indicators functions of $ U_i $, $ i\in\mathbb{N} $, using instead of the process $ X^n  $ and its  continuous version $ {X}^{c,n} $, the process $ \mathcal{X}^{n} $ described in Section \ref{sec:mlt} and
        \begin{align}\label{eq:fi}
                \tilde{\mathbb{E}}_{i-1}[f(X^{c,n}_{i})\bar{m}_{i}]=
                \mathbb{E}_{i-1}[f(\mathcal{X}^{n}_{t_i})/\mathcal{X}^n_{t_{i-1}}=X_{i-1}].
        \end{align}
        Furthermore,
        \begin{align}\label{eq:sec}
                \tilde{\mathbb{E}}_{i-1}[f(X^{c,n}_{n})E^n_{i-1:n}\mathsf{h}_i\bar{M}^n_{i-1:n}]=
                \mathbb{E}_{i-1}[f(\mathcal{X}^{n}_{T})1_{(\exists s\in [t_{i-1},t_i];\mathcal{X}^n_s\cdot\mathbf{e}^1=L)}E^n_{i-1:n}/\mathcal{X}^n_{t_{i-1}}=X_{i-1}].
        \end{align}
Here, we remark that the definition of $ E^n $ on the right hand side of the above equality needs to be modified using $ \mathcal{X} $ instead of $ X $. In this sense, we abuse the notation using the same symbol for this process. 
 
 Therefore, we have that the set $\Phi ^{n,t}$ can be identified with the set where $X^{c,n}$ does not hit the boundary, in other words $X^{c,n}_s\cdot\mathbf{e}^1>L$ for any $s\in [0,t]$. On the complement of the set, we denote
 \begin{align*}
                        \bar{\rho}_t^n:=\sup\{t_{i}\leq t;\mathsf{h}_{i}\left( \omega \right)
=1 \}.
                \end{align*}
The complement of the set $\Phi ^{n,t}$ can be identified with the set where $X^{c,n}$ hits the boundary at least once, in other words $X^{c,n}_s\cdot\mathbf{e}^1=L$ for some $s\in [0,T]$. On the set $ \Phi ^{n,t} $, we have similarly to the case $\ell=1$ that
            \begin{align*}
                E^{n,\ell}_{t}=E^{n,\ell}_{0}+\sum_{j=1}^{\lceil \frac{t}{\Delta}\rceil}
                (\mathbf{e}^\ell)^{\top}\bar{e}_iE^{n}_{j-1}.
            \end{align*}  
On the complement of $ \Phi ^{n,t}$, we have
\[
        E^{n,\ell}_{t}=(\mathbf{e}^\ell)^{\top} E^{n,\ell}_{\bar{\rho}_t^n\Delta^{-1}-1}+
        \sum_{j=\bar\rho_t^n\Delta^{-1}}^{\lceil \frac{t}{\Delta}\rceil}(\mathbf{e}^\ell)^{\top}\bar{e}_jE^{n}_{t_{j-1}}
\]
We put the two cases together and use  $\bar{\rho}_t^n\vee 0$ on $\Phi ^{n,t}$ so that it equals $ 0 $ on this set. Then 
\[
        E^{n,\ell}_{t}=E^{n,\ell}_{0}1_{\bar{\rho} ^{n}_t=-\infty}+ E^{n,\ell}_{\bar{\rho}_t^n\Delta^{-1}}1_{\bar{\rho} ^{n}_t\in [0,t]}+
        \sum_{j=(\bar\rho_t^n\Delta^{-1}+1)\vee 0}^{\lceil \frac{t}{\Delta}\rceil}(\mathbf{e}^\ell)^{\top}\bar{e}_jE^{n}_{t_{j-1}}.
\]

            As $ E^n  $ converges, we have that for $ \bar{\rho} ^{n}_t\in [0,t] $
            \begin{align*}
            (\mathbf{e}^\ell)^{\top} E^{n}_{\rho_t\Delta^{-1}}=
             (\mathbf{e}^\ell)^{\top}\bar{e}_{\bar\rho_t\Delta^{-1}}E^{n}_{\bar\rho_t\Delta^{-1}-1}\to 0.
            \end{align*}
       As above, we have that the convergence of the sum follows due to the uniform convergence 
        \begin{align*}
        \sum_{j=0}^{\lceil \frac{t}{\Delta}\rceil}
        \mathbf{e}^\ell\bar{e}_iE^{n}_{j-1}\to \int_0^t    \left((\mathbf{e}^\ell)^{\top} \mathbf{D}\bar{b}(Y_u)du+(\mathbf{e}^\ell )^{\top} \mathbf{D}\sigma^k(Y_u) dW_u^{k} \right)\mathcal{E}_{u}.
        \end{align*}
    
    Since $ \bar{\rho}^n_t\to\rho_t $ (see Section \ref{sec:8}) and the above terms converge uniformly we get that \eqref{el1} and \eqref{el2} hold true for each fixed $ t>0 $.
    
    Note there is a unique solution to  the system \eqref{el1}-\eqref{el2}.
    For this, is enough to consider  two solutions, say $ \mathcal{E}^i $, $ i=1,2 $ and define $ \xi_t:=\mathcal{E}^1_t-\mathcal{E}^2_t $. Then, we have 
    \begin{align*}
        |\xi^{\ell j}_t|^2\leq &2\sup_{s\leq t}\left|\int_0^s    \left((\mathbf{e}^\ell)^{\top} \mathbf{D}\bar{b}(Y_u)du+(\mathbf{e}^\ell )^{\top} \mathbf{D}\sigma^k(Y_u) dW_u^{k} 
        \right)\xi_{u}\mathbf{e}^j\right|^2.
    \end{align*}
From here, using classical Gronwall type arguments one obtains that $ \xi=0 $. Therefore using the continuity by the right of $ \rho_t $, the characterization of $ \mathcal{E} $ follows.

Moreover, $\mathcal{E}$ satisfies $ \mathcal{E}_{s,t} \mathcal{E}_{s} =\mathcal{E}_{t}$ where $   \mathcal{E}_{s,t}$ denotes the solution of the system \eqref{el1}-\eqref{el2} started at $ s\in (0,t) $.
This is obtained by taking the limit of the identities $ E^n_{\lceil s/\Delta\rceil ,\lceil t/\Delta\rceil } E^n_{\lceil s/\Delta\rceil }= E^n_{\lceil t/\Delta\rceil }$.

\end{proof}

With the above result, as in the one dimensional case, one obtains using Girsanov's theorem and It\^o formula the following result.
\begin{theorem} 
        \label{th:5}Let $ (Y,\mathcal{E}) $ satisfy \eqref{eq:Y} then 
        \begin{equation}  \label{limder}
        \begin{aligned}
        \tilde{\mathbb{E}}\left[  {\mathbf{D}}{f}\left( X_{n}^{n,x}\right)E^{n}\mathcal{K}_n\bar{M}_{n}%
                \right]\to &
                \widetilde{\mathbb{E}}\left[ {{\mathbf{D}}f}^{\prime }\left( Y_T\right)
                \mathcal{E}_T\mathcal{K}_T\right]\\
                &=\mathbb{E}\left[  {\mathbf{D}} {f}\left( Y_T\right)\xi_T\right ],\\
                        \tilde{\mathbb{E}}\left[\sum_{i=1}^n f_i\frac{b^1(z_{i-1})
                        }{{{a^{11}}(z_{i-1})}}       
                        (\mathbf{e}^1)^{\top}  E^{n}_{i-1}\mathsf{h} _i\mathcal{K}_i\bar{M}_{i}\right ]
                \to &   
               \widetilde{\mathbb{E}}_{}\left[f(Y_T)\frac{b^1
                }{{{a^{11}}}} (Y_{\rho_T})    
                (\mathbf{e}^1)^{\top} 
                \mathcal{E}_{\rho_T}\mathcal{K}_T 1_{(\tau\leq T)} \right]\\
                &=
                \mathbb{E}\left[ f(Y_T)
                \frac{b^1
                }{{{a^{11}}}} (Y_{\rho_T})    
                (\mathbf{e}^1)^{\top}  \xi_{\rho_T}
                        1_{\tau\leq T}\right ].
        \end{aligned}
        \end{equation} 
    Here $\xi$ is the solution of the equation (\ref{eq:E}). 
\end{theorem}
        \begin{remark}
                \label{rem:16}To keep the notation consistent, in (\ref{limder}) we kept the
                tilde in the expression for the expectation in the above two limits, 
                $\widetilde{\mathbb{E}}\left[ {{\mathbf{D}}f}^{\prime }\left( Y_T\right)
                \mathcal{E}_T\mathcal{K}_T\right]$ and $\widetilde{\mathbb{E}}_{}\left[f(Y_T)\frac{b^1
                }{{{a^{11}}}} (Y_{\rho_T})    
                (\mathbf{e}^1)^{\top} 
                \mathcal{E}_{\rho_T}\mathcal{K}_{\rho_T} 1_{(\tau\leq T)} \right]$. The expectation in these two quantities are taken with respect to a
                probability measure, which we can denote by $\widetilde {\mathbb{P}}$ (via a slight abuse
                of notation as this is not necessarily the measure $\widetilde {\mathbb{P}}$ defined in
                Section \ref{sec:4}), under which $(Y,\mathbf{B})$ is the solution of (\ref{reflected}%
                ), the normally reflected equation in the domain $H^d_L$with no drift term. One
                can use an equivalent representation of the two limits as $\mathbb{E}\left[  {\mathbf{D}} {f}\left( Y_T\right)\xi_T\right ]$ and $\mathbb{E}\left[ f(Y_T)
                \frac{b^1
                }{{{a^{11}}}} (Y_{\rho_T})    
                (\mathbf{e}^1)^{\top}  \xi_{\rho_T}
                        1_{(\tau\leq T)}\right ]$, where the expectation $\mathbb{E}$ in the two previous quantities
                are taken with respect to a probability measure, which we can denote by $\mathbb{P}$
                (again, via a slight abuse of notation as this is not necessarily the
                original measure $\mathbb{P}$), under which $(Y,\mathbf{B})$ is the solution of (\ref{eq:Y})+(\ref{eq:B}), the normally reflected equation in the domain $H^d_L$ that incorporates the
                drift term $b$. The transfer from $\widetilde {\mathbb{P}}$ to $\mathbb{P}$ is done via Girsanov's
                theorem. More precisely, we have that $\displaystyle \left.\frac{d\widetilde {\mathbb{P}}}{d
                        \mathbb{P}}\right|_{{\mathcal{F}}_T}={\widetilde{K}_T}^{-1}$, where 
                \begin{equation*}
                \widetilde{K}_t:=\exp \left( \int_{0}^{t}b^{\top}\sigma^{-1}(Y_{s})^\top dW_{s}-\frac{1}{2}%
                \int_{0}^{t}b^{\top}a^{-1}b(Y_{s})ds\right),\ \ \ t\ge 0 . 
                \end{equation*}
                The equivalent representation in (\ref{limder}) is obtained once we observe
                that the process $\xi$, as defined in equation (\ref{eq:E}), satisfies the
                identity $\xi_{T}=\frac{\mathcal{E}_T\mathcal{K}_T}{\widetilde{K}_T}$.
        \end{remark}

        \begin{proof}[Proof of Theorem \protect\ref{th:5}]
                The first result is immediate from Theorem \ref{bigT} and the moment
                estimates in Lemma \ref{moments}. The second one, requires a rewriting of
                the expectation using a path decomposition. Recall that $\bar{\rho}_T^{n}=\sup\{t_i;U_i\le p_i\} $ and therefore $1_{(\bar{\rho}%
                        _T^{n}=t_i)}=1_{(U_i\leq p_i,U_{i+1}>p_{i+1},...,U_n>p_n)} $. In particular,
                we define $\bar{\rho}_T^n=0 $ and $E^n_{-\Delta}=0 $ if $\{t_i;U_i \le p_i\}
                =\emptyset $. Furthermore if we let  $\bar{\tau}_T^n=\inf\{t_i;U_i \le p_i \}\wedge T $ and using the
                tower property for conditional expectations we obtain that
                \begin{align*}
                \sum_{i=1}^{n}\widetilde{\mathbb{E}}_{0,x}\left[ f_{i}\left( X_{i}^{n}\right)
                E_{i-1}^{n}\bar{h}_{i}\mathcal{K}^n_i\bar{M}_{i}^{n}\right] = &{\frac{b}{a}}%
                \left( L\right)\widetilde{\mathbb{E}}_{0,x}\left[f(X^{n}_n){\mathcal{K}}%
                _n\sum_{i=1}^{n}{E}^n_{i-1}1_{(\bar{\rho}^n_T=t_i)}\bar{M}^n_{i-1}\right] \\
                =&{\frac{b}{a}}\left( L\right)\widetilde{\mathbb{E}}_{0,x}\left[f(X^{n}_n){%
                        \mathcal{K}}_n1_{(\bar{\tau}_T^{n}<T)} {E}^n_{\bar{\rho}_T^{n}-\Delta}\bar{M}^n%
                _n\right].
                \end{align*}
                Also
                $( \bar{\rho}_T^{n},\bar{\tau}_T^{n}) $
                converges to $( \rho_T,\tau)$, where $\rho_T:=\sup\{s<T: Y^1_s=L\}$ (see Section \ref{sec:8})) from the
                convergence of  $(X^{n},E^{n},R^{n},\Gamma^n, \bar\Gamma^n,\mathcal{K}^n,B^n,\mathbf{B}^n)$ to  $(Y,\mathcal{E},W,\mathbf{\Gamma},0,\mathcal{K},B,\mathbf{B}),$  first in distribution and then $\mathbb P$-almost surely via a Skorohod representation. The result follows as $P(\tau=T)=0$.
        \end{proof}
        
        We are now ready to prove that the limit deduced in Theorem \ref{th:5}
        corresponds to the representation (\ref{eq:diff}) of the derivative of $\mathbb{E}[f(X_{T})1_{(\tau >T)}]$. To do so, we re-introduce the explicit dependence on the
        starting value $x$ of the various processes needed to obtain the
        representation of the derivative and define the functions $\varphi
        _{n},\varphi ,\psi :[T,\infty )\mapsto \mathbb{R}$ as follows 
        \begin{align*}
\varphi \left( x\right) =&\mathbb{E}_{0,x}[f(X_{T\wedge \tau })]. \\
\psi \left( x\right) =&\mathbb{E}\left[  {\mathbf{D}} {f}\left( Y_T\right)\xi_T\right ]+               \mathbb{E}\left[ f(Y_T)
                \frac{b^1
                }{{{a^{11}}}} (Y_{\rho_T})    
                (\mathbf{e}^1)^{\top}  \xi_{\rho_T}
                        1_{\tau\leq T}\right ] . \\
\varphi _{n}\left( x\right) =&\mathbb{E}_{0,x}[f(X_{T\wedge \tau^{n}
}^{n})].
        \end{align*}

Putting all the results together up to this point, we obtain the first result announced in \eqref{eq:diff}.
\begin{theorem}\label{th:24}
        Let $ f\in C^1_b(H^d_L) $ such that $ f(x)=0 $, $ x\in \partial H^d_L $. Then, under {\bf Hypothesis 1},  $P_Tf$ is differentiable on $H^d_L$ and its derivative has the representation  \eqref{eq:diff}. That is,
                \begin{equation}\label{eq:0'}
                {\mathbf{D}}_x P_Tf(x)=    {\mathbf{D}}_x\mathbb{E}[f(X_T)1_{(\tau>T)}]=
                \mathbb{E}[{\mathbf{D}} f(Y_T)\xi_T]+
                \mathbb{E}\left[f(Y_T)
                {  \frac{b^1}{{a^{11}}}(Y_{\rho_T})(\mathbf{e}^1)^{\top}  } \xi_{\rho_T}     1_{\tau\leq T}
                \right].
        \end{equation}
        \end{theorem}
     \begin{proof}   
        The existence of the above derivative is proved as in the one dimensional case, which requires the non-trivial technical arguments proving the boundedness of the second derivative provided in Proposition \ref{prop:30}. This is an important difference with the 1D argument. Once we accept this, the proof follows the use of an Arzel\`a-Ascoli classical argument.
        
 Following  \cite{Gobet}, we
                have that $\lim_{n\rightarrow \infty}\varphi^n=\varphi$ uniformly on any
                compact set $K\subseteq H^d_L $. 
                Let us restrict first to the case that $f\in C_{b}^{2}(H^d_L)$. We proved in Theorem \ref{th:5} that $\mathbf{D}_x\varphi _{n}$ converges
                pointwise to $\psi$. Now, there exists a function $ \phi_n $, whose explicit definition is given in Section 
                \ref{sec:7.6}, for which  $ \phi_n-\mathbf{D}_x\varphi _{n} $ converges uniformly to zero on compact sets. Moreover the derivative $\mathbf{D}_x\phi _{n}$  is  uniformly bounded. As a result, the convergence of  $\mathbf{D}_x\varphi _{n}$ to $\psi$ is uniform on any compact set $K\subseteq H^d_L$. 
                
                Since the differentiation operator is a closed operator in the uniform topology, we deduce immediately that the limit $\varphi$ is differentiable and its limit is indeed $\psi$.
                The generalization to $f\in C_{b}^{1}(H^d_L)$ is done by taking the limit of the identity \eqref{eq:0'} along a sequence $f_n\in C_{b}^{2}(H^d_L)$ such that $\lim_{n\rightarrow \infty}\mathbf{D}f_n =\mathbf{D}f$ uniformly on $H^d_L $ and $\lim_{n\rightarrow \infty}f_n =f$ pointwise (and therefore $\lim_{n\rightarrow \infty}f_n =f$ uniformly on any compact $K\in H^d_L $).  
  \end{proof}

\subsection{Characterization of the derivative of the semigroup}
\label{sec:6.1}
In the previous sections, we have been able to prove that $ \mathbf{D} _{x}\mathbb{E}\left[ f\left( X_{T\wedge \tau }^{n,x}\right) \right] $ converges and characterize the limit. 
In this section, we prove the final part of the program which states that the formula \eqref{eq:main} is satisfied.

In order to do this, we will first neglect the terms in $ \mathbf{D} _{x}\mathbb{E}\left[ f\left( X_{T\wedge \tau }^{n,x}\right) \right] $ which converge to zero uniformly (see Lemma \ref{lem:uc}). Then we will prove that the remaining terms generate a linear equation and solving it leads to \eqref{eq:main}. This argument was also used in the 1D case.

%
%

        \begin{theorem}
                \label{th:44}
                Let $ f\in C^1_b (H^d_L)$ and $ \Psi $ defined as in \eqref{eq:Psi} then
                \begin{align*}
                        {\mathbf{D}}_x P_Tf(x)=&\mathbb{E}\left[{\mathbf{D}} f(Y_T)\Psi_T\right].
                \end{align*}
        \end{theorem}
        \begin{proof}

                To characterize the derivative process, one has to consider first the linear equation associated to 
                $ \partial_1P_Tf(x) $. 
                
                In fact, as in the one dimensional case using derivative estimates and taking limits of
                \begin{align*}
                \mathbb{E}
          \left[
           \sum_{i=1}^n
            \left(f_i-f_i(\pi({X}_i))
            \right)
            \frac{b^1_{i-1}
            }{a^{11}_{i-1}}     
        (\mathbf{e}^1)^{\top}  E^{n}_{i-1}\mathsf{h}_i
        \mathcal{K}_i\bar{M}_{i}
        \right] 
                \end{align*}
            instead of the second result in Theorem \ref{th:5}
                                one obtains the following linear equation for $ \hat{x}:=(L,x_2,...,x_d) $
                \begin{align}
                       \mathbf{D}_xP_Tf(\hat{x})=\mathbb{E}\left[{\mathbf{D}} f(Y_T)\mathcal{E}_T\right]+\mathbb{E}\left[\int_0^T{{b}^1({Y}_s)
                        }
                        (\mathbf{e}^1)^{\top} \mathcal{E}_s\partial_1P_{T-s}f({Y}_s)d{B}_s\right].
                        \label{eq:gen}
                \end{align}
            In particular, considering the derivative with respect to $ x^1 $ one obtains the following linear equation in $ \partial_1P_{T-s}f $:
            \begin{align*}
                \partial_{1}P_Tf(\hat{x})=\mathbb{E}\left[{\mathbf{D}} f(Y_T)\mathcal{E}_T\mathbf{e}^1\right]+\mathbb{E}\left[\int_0^T{{b}^1({Y}_s)
                }
                (\mathbf{e}^1)^{\top} \mathcal{E}_s\mathbf{e}^1\partial_1P_{T-s}f({Y}_s)d{B}_s\right].
            \end{align*}
                This equation is solved with similar steps as in the one dimensional case except that one has to take into account the following identity:
                \begin{align*}
                    (\mathbf{e}^1)^{\top} \mathcal{E}_{s,r}\mathbf{e}^1    (\mathbf{e}^1)^{\top} \mathcal{E}_s\mathbf{e}^1d{B}_s=
                        (\mathbf{e}^1)^{\top} \mathcal{E}_r\mathbf{e}^1d{B}_s.
                \end{align*}
                The above identity holds true as $ d{B}_s\neq 0 $ implies
                \begin{align}
                       (\mathbf{e}^1)^{\top} \mathcal{E}_{s,r}\mathbf{e}^1 (\mathbf{e}^1)^{\top} \mathcal{E}_s\mathbf{e}^1=
                        \sum_{\ell=1}^d(\mathbf{e}^1)^{\top} \mathcal{E}_{s,r}\mathbf{e}^\ell(\mathbf{e}^\ell)^{\top} \mathcal{E}_{s}\mathbf{e}^1=
                        (\mathbf{e}^1)^{\top} \mathcal{E}_{s,r}\mathcal{E}_{s}\mathbf{e}^1=(\mathbf{e}^1)^{\top} \mathcal{E}_r\mathbf{e}^1. \label{eq:aha}
                \end{align}
                This gives using a similar argument as in the 1D case
                \begin{align}
                      \notag  \partial_1P_Tf(\hat{x})=&\mathbb{E}\left[{\mathbf{D}} f(Y_T)\left(\mathcal{E}_T\mathbf{e}^1+\int_0^T\mathcal{E}_{r,T}\mathbf{e}^1(\mathbf{e}^1)^{\top}\mathcal{E}_r\mathbf{e}^1e^{\int_0^r{{b}^1({Y}_s)
                                } d{B}_s}b^1(Y_r)d{B}_r
                        \right)\right]\\
                        =&
                        \mathbb{E}\left[{\mathbf{D}} f(Y_T)\mathcal{E}_T\mathbf{e}^1e^{\int_0^T{{b}^1({Y}_s)
                                } d{B}_s}
                        \right].
                       \notag
                \end{align}
                 Replacing in the general formula  \eqref{eq:gen}, we obtain
                \begin{align*}
                        {\mathbf{D}}_x P_Tf(x)=\mathbb{E}\left[{\mathbf{D}} f(Y_T)\mathcal{E}_T\right]+\mathbb{E}\left[\int_0^T{{b}^1({Y}_s)
                        }
                        (\mathbf{e}^1)^{\top} \mathcal{E}_s\partial_1P_{T-s}f({Y}_s)d{B}_s\right]
                \end{align*}
                which gives using the explicit formula for $ \partial_1P_Tf(\hat{x}) $ and the tower property
                \begin{align*}
                        {\mathbf{D}}_x P_Tf(x)=&\mathbb{E}\left[{\mathbf{D}} f(Y_T)\mathcal{E}_T\right]+
                        \mathbb{E}\left[\int_0^T{{b}^1({Y}_s)
                        }
                        {\mathbf{D}} f(Y_T)\mathcal{E}_{s,T}\mathbf{e}^1
                        (\mathbf{e}^1)^{\top} \mathcal{E}_se^{\int_s^T{{b}^1({Y}_u)
                                } d{B}_u}d{B}_s\right].
                \end{align*}
                Using the same argument as in \eqref{eq:aha}, one obtains:
                \begin{align*}
                        {\mathbf{D}}_x P_Tf(x)=&\mathbb{E}\left[{\mathbf{D}} f(Y_T)\mathcal{E}_T\right]+
                        \mathbb{E}\left[        {\mathbf{D}} f(Y_T)\mathcal{E}_{T}\int_0^T{{b}^1({Y}_s)
                        }
                        e^{\int_s^T{{b}^1({Y}_u)
                                } d{B}_u}d{B}_s\right]\\
                        =&\mathbb{E}\left[      {\mathbf{D}} f(Y_T)\mathcal{E}_{T} 
                        e^{\int_0^T{{b}^1({Y}_u)
                                } d{B}_u}\right].
                \end{align*}
            The conditions on $ \Psi_t =\mathcal{E}_{T} 
            e^{\int_0^T{{b}^1({Y}_u)
            	} d{B}_u}$ in \eqref{eq:Psi} and \eqref{eq:sc} follow by Ito's formula and using the equation \eqref{el1}.
        \end{proof}
    \subsection{The BEL formula}
        We now proceed with the integration by parts formula. The proof is parallel to the one dimensional case and for this reason, we only mention the main points in the proof. 
       
        \begin{theorem}
                \label{th:19}
                Let $f:H^d_L \to\mathbb{R} $ be a measurable and bounded function such
                that $f(x)=0 $ for $ x\in\partial H^d_L $ .   Then for $ x\in {H^d_L}$, we have 
               \begin{align*}
                 T\mathbf{D}_xP_Tf(x)=\mathbb{E}\left[f(Y_T)\int_{\rho_T\vee 0}^T(\mathbf{e}^\ell)^{\top}\sigma^{-1}(Y_s)\Psi_sdW^\ell_s\right].
               \end{align*}   
            
        \end{theorem}
        
        We remark that the random variable which appears in the above BEL formula
        corresponds to the one in the classical BEL diffusion  formula in the case when $B_T=0 $ or equivalently $ \rho_T=0 $.
        
        \begin{proof}
                The idea of the proof is  to  start from an approximation to the result in
                Theorem \ref{th:44} but considering the interval $[0,t_i] $ instead of $%
                [0,t_n] $. Using \eqref{eq:td1} and with the same arguments as in the proof of Theorem \ref{bigT}  gives 
                \begin{align*}
                        &\mathbf{D}_xP_Tf(x) =\widetilde{\mathbb{E}}\left[\mathbf{D}_if_i\widehat{E}_{i}^{n}%
                        \mathcal{K}_i \bar{M}_{i}^{n}\right]+o(1) .
                \end{align*}
            As in the one dimensional case, $  \widehat{E}%
            _j^{n,x}:=\prod_{i=1} ^j\widehat{e}_j $ with

            \begin{align*}
                \widehat{e}_i:=&I+(\mathbf{e}^1(\mathbf{e}^1)^{\top}-I)
                \mathsf{h}_i
                +\mathbf{D}\sigma_{i-1} ^{k }
                \bar {r}^k_{i}+\bar{b}_{i-1}\Delta+\frac{b^1_{i-1}}{a^{11}_{i-1}}\mathtt{X}_{i}^1I
                \mathsf{h}_i.
            \end{align*} 
        
                Now applying an integration by parts in the interval $[t_{i-1},t_i] $, one obtains  
                \begin{align}
                        &\widetilde{\mathbb{E}}_{i-1}\left[ \mathbf{D}_if_i \widehat{E}_{i}^{n}\mathcal{K}^n_i
                        \bar{M}_{i}^{n}\right]= \widetilde{\mathbb{E}}_{i-1}\left[ f_i\varkappa_i \widehat{E%
                        }_{i-1}^{n} \mathcal{K}^n_i\bar{M}_{i}^{n}\right]+\frac{O_{i-1}^E(1)}{\Delta}
                        \notag\\
                        &\varkappa_i\Delta:= \left(\sigma^{-1}_{i-1}\left(I+(\mathbf{e}^1(\mathbf{e}^1)^{\top}-I)\mathsf{h}_i\right)\right)^{k\cdot}Z^k_i+\frac{\sigma^{1\cdot}_{i-1}\mathtt{X}^1_{i-1}}{a_{i-1}^{11}}\sigma^{-1}_{i-1}(2\mathbf{e}^1(\mathbf{e}^1)^\top-I
                        )\mathsf{h}_i  \notag\\
                        & -\Delta b^{\top}_{i-1} a^{-1}_{i-1}   \left(I+(\mathbf{e}^1(\mathbf{e}^1)^{\top}-I)\mathsf{h}_i\right)+Z_i^{\top}\sigma^{-1}_{i-1}\mathbf{D}\sigma^{k}_{i-1}Z_i^k-(\sigma^{-1})^{k\cdot}\mathbf{D}\sigma^{k}_{i-1}\Delta.  
                        \label{eq:BEL}
                \end{align}      
                Doing this at any time $j$ one obtains an approximative discrete time BEL formula which can
                be multiplied by $\Delta $ and then summed over all $j =1,...,n$ . Using this argument one obtains
                approximations to integrals as in the one dimensional case.

                Taking limits the result follows. 
                In particular, note that we have used for the last term in \eqref{eq:BEL}, the fact that $ (Z^k_i)^2-\Delta $ and $ Z^k_iZ^{\ell_i} $ are increments of discrete time martingales with quadratic variation that converge to zero. On the other hand, for the second term in \eqref{eq:BEL}, we have that as $ f_i(\pi({X}_{i}))=0 $ then
                \begin{align*}
                      \widetilde{\mathbb{E}}_{i-1}\left[   f_i\mathtt{X}^1_{i-1}\mathsf{h}_i\right] =\widetilde{\mathbb{E}}_{i-1}\left[ (f_i-f_i(\pi({X}_i)))\mathtt{X}^1_{i-1}\mathsf{h}_i\right]=O^E_{i-1}(\sqrt{\Delta}).
                \end{align*}
                
                Then as in the one dimensional case, we obtain the result.
             
        \end{proof}



\section{Appendix}

\subsection{Differentiation rules}\label{diffrules}
In the following, we will use an integration by parts formula with respect to the law of the Gaussian increment $ \Delta_iW^k$. Recall that we denote the derivative with respect to $\Delta_iW^k $ as $ D^k_i $ as given in Definition \ref{def:2a}. We remark that the results in this Section  do not use Hypothesis \ref{hyp:12} and therefore it also applies to the case when this hypothesis is not assumed. A similar remark applies to the formulas for second derivatives in Section \ref{sec:7.6}.

\begin{lemma}
        \label{lem:2H}
        Let $ G_i:=g(X_{i-1},\Delta_iW^k)$  for  $ g:\mathbb{R}^d\times\mathbb{R}^d\to \mathbb{R}$ a Lipschitz function. Then
        \begin{align}
                &\mathbb{E}_{i-1}\left[G_i{\mathbf{D}}_{i-1} p_i\right]=-2\mathbb{E}_{i-1}\left[G_i{\mathbf{D}}_{i-1}\left(\frac{{b}^1_{i-1}
                        \mathtt{X}_{i-1}^1}{a^{11}_{i-1}}\right )\mathsf{h}_i\right]-2\mathbb{E}_{i-1}\left[D_i^kG_i{\mathbf{D}}_{i-1}\left (\frac{      {\sigma}^{1k} _{i-1}
                        \mathtt{X}_{i-1}^1
                }{a^{11}_{i-1}}\right )\mathsf{h}_i
                \right],
                \label{eq:2f}
        \end{align}
where $ p_i= \exp\left(-2\frac{\mathtt{X}_i^1\mathtt{X}_{i-1}^1}{a^{11}_{i-1}\Delta}\right)$ is defined in \eqref{maandmbar}.
\end{lemma}
\begin{proof}
        The proof is similar to the one-dimensional case, except that one has to take into account the fact that the gradient is a row vector. 
        We have that 
        \begin{equation}
                {\mathbf{D}}_{i-1} p_i=-2p_iA_{i},
                \ \ \ A_{i}:={\mathbf{D}}_{i-1}\left(
                \frac{
                        \mathtt{X}^1_i\mathtt{X}_{i-1}^1
                }{      a^{11}_{i-1}\Delta}
                \right).
                \label{def:aim0}
        \end{equation}
        As in the one dimensional case, using \eqref{eq:defbXm}, we obtain the decomposition $ A_i=B_{i-1}+C_i $ for $ B_i\in\mathcal{F}_{i-1} $ and $ C_i\in\mathcal{F}_i $ with
        \begin{align*}
                B_{i-1}:=&{\mathbf{D}}_{i-1}\left( \frac{
                        (\mathtt{X}_{i-1}^1)^2
                }{a^{11}_{i-1}\Delta}\right )+{\mathbf{D}}_{i-1}\left(\frac{{b}^1_{i-1}
                        \mathtt{X}_{i-1}^1}{a^{11}_{i-1}}\right )
                \\
                C_i:=&  
                {\mathbf{D}}_{i-1}\left(
                \frac{{\sigma}^{1k} _{i-1}
                        \mathtt{X}_{i-1}^1}{a^{11}_{i-1}}
                \right )\frac{\Delta_{i}W^{k} }{\Delta}.
        \end{align*}
        Note that 
        \begin{align*}
                \psi_{i-1}^k:={\mathbf{D}}_{i-1}\left(\frac{{\sigma}^{1k} _{i-1} }{a^{11}_{i-1}}
                \mathtt{X}_{i-1}^1
                \right )\in\mathcal{F}_{i-1} \text{ and }D_i^kp_i=
                -2\frac{        {\sigma}^{1k} _{i-1}
                        \mathtt{X}_{i-1}^1
                }{a^{11}_{i-1}\Delta } p_i.
        \end{align*}
        Using the above information, we can obtain \eqref{eq:2f}. In fact,
        \begin{align*}
                \mathbb{E}_{i-1}\left[ G_{i}
                2p_i\psi_{i-1}^k\frac{\Delta_{i}W^{k} }{\Delta}\right]
                =&\mathbb{E}_{i-1}\left[ G_{i}
                2p_i\psi_{i-1}^k\frac{\Delta_iW^k}{\Delta}\right]
                =\mathbb{E}_{i-1}\left[D^k_i\left( G_{i} p_i\right)
                2\psi^k_{i-1}\right]\\
                =&\mathbb{E}_{i-1}\left[D^k_i\left( G_{i} p_i\right)
                2\psi^k_{i-1}\right]\\
                =&2\mathbb{E}_{i-1}\left[ D_i^kG_{i} p_i
                \psi^k_{i-1}\right]-4
                \mathbb{E}_{i-1}\left[G_{i} p_i
                \frac{  {\sigma}^{1k} _{i-1}
                        \mathtt{X}_{i-1}^1
                }{a^{11}_{i-1}\Delta }
                \psi^k_{i-1}\right].
        \end{align*}
        Adding this result to the computation of $ B_{i-1} $,  we obtain, using the product rule of differentiation,  that
        \begin{align*}
                \mathbb{E}_{i-1}\left[ G_{i}
                2 p_iA_{i}\right]
                =&2\mathbb{E}_{i-1}\left[ D_i^kG_{i} p_i
                \psi^k_{i-1}\right]+2\mathbb{E}_{i-1}\left[ G_{i} p_i
                {\mathbf{D}}_{i-1}\left(
                {b}^1_{i-1} 
                \frac{\mathtt{X}^1_{i-1}}{a^{11}_{i-1}}
                \right )\right].
        \end{align*}
        From here \eqref{eq:2f} follows.   
        %
        %
\end{proof}

\begin{lemma}
        \label{cor:23H}
        Assume that $ f:\mathbb{R}^d\to\mathbb{R} $ is a Lipschitz function such that $ f(x)=0 $ for all $ x\in  (H^d_L)^c $ and $ \lambda^0\in C^1_b(\mathbb{R}^d\times\mathbb{R}^d,\mathbb{R}) $. Define $ f_i:=f(X_i) $ and $ \lambda_i^0:= \lambda^0(X_{i-1},\Delta _i W)$.  Then we have 
        \begin{align*}
                {\mathbf{D}}_{{i-1}}\mathbb{E}_{{i-1}}\left[f_i{\lambda}^0_{i}{m}_{i}\right]
                =&\mathbb{E}_{{i-1}}\left[
                {\mathbf{D}}_{i-1}(f_i\lambda_i^0){m}_i
                +2D_i^k(f_i\lambda_i^0){\mathbf{D}}_{i-1}\left (\frac{        {\sigma}^{1k} _{i-1}
                        \mathtt{X}_{i-1}^1
                }{a^{11}_{i-1}}\right )1_{({X}^1_i>L)}\mathsf{h}_i\right.\\
                &\quad\quad\left.+2{f_i\lambda_i^0}{\mathbf{D}}_{i-1}\left(\frac{{b}^1_{i-1}
                        \mathtt{X}_{i-1}^1}{a^{11}_{i-1}}\right )1_{({X}^1_i>L)}\mathsf{h}_i\right].
        \end{align*}
        
        Furthermore for any random variables $ A_i,B_i\in\mathcal{F}_i $, we have
        \begin{align*}
                \mathbb{E}_{i-1}[A_i{m}_i+B_i1_{({X}^1_i>L)}\mathsf{h}_i]=\mathbb{E}_{i-1}\left[\left(A_i(1-\mathsf{h}_i)+\frac 12 B_i\mathsf{h}_i
                \right)\bar{m}_i\right ].
        \end{align*}
\end{lemma}
\begin{proof}
        
        First, note that the following is a Lipschitz function which vanishes at the boundary of $ H^d_L $:
        \begin{align*}
                x\mapsto 1_{(x^1>L)}(1- p(X_{i-1},x))=1_{(x^1>L )}\left(1-
                \exp\left({-2\frac{
                                (x^1-L)\cdot 1_{(x^1>L)}\cdot \mathtt{X}_{i-1}^1
                        }{a^{11}_{i-1}\Delta }}  \right)\right).
        \end{align*} 
        Here, recall that $ p(x,y) $ is defined as the function that defines $ p_i $.
        Second, we remark in reference to $ 1_{({X}^1_i>L)}$ that one can always ignore the differentiation of this term because it leads to a Dirac delta distribution and $ f $ is a Lipschitz continuous function with $ f(x)=0 $ for $ x\in (H^d_L)^c $.
        
        With these two remarks, the result follows from direct applications of differentiation and{ the arguments in the proof of Lemma \ref{lem:2H}  }.
\end{proof}

With these results we can now prove Lemma \ref{lem:2.1H}.

\begin{proof}[Proof of Lemma \ref{lem:2.1H}]
        First, note that $ f_i $ is a continuously differentiable  function with  $ f_i\big|_{\partial H^d_L} =0$. 
        The proof is an application of  Lemma \ref{cor:23H} with the corresponding algebra in order to identify all terms in \eqref{eq:td1bH}. Most of the proof is as in the 1D case. In fact, for $ \tau^n:=\inf\{s\geq 0; X^{c,n,1}_s=L\} $, we have
        \begin{equation*}
        	\mathbb{E}_0\left[f(X_n)1_{(\tau^n>T)}\right]=
        	\mathbb{E}_0\left[f(X_n)\mathbb{E}\left[1_{(\tau^n>T)}\Big/\mathcal{F}_n\right]\right]=
        	\mathbb{E}\left[ f\left( {X}_{n}\right) M^n_{n}\right] =\mathbb{E}\left[
        	f_{1}M^n_{1}\right] .
        \end{equation*}
    Here, we have used the fact that $ \mathbb{E}\left[1_{(\tau^n>T)}\Big/\mathcal{F}_n\right]=\mathbb{E}\left[1_{(\tau^n>T)}\Big/\sigma(X^{c,n,1}_i;i\leq n)\right] =M^n_n$.

        In order to guide the reader through the extra terms that do not appear in the 1D case, note that when one differentiates $ m_i $, one obtains:
        \begin{align*}
                \mathbb{E}_{i-1}\left[f_i1_{({X}^1_i>L)}\mathbf{D}_{i-1}p_i\right]      =&2\mathbb{E}_{i-1}\left[f_i{\mathbf{D}}_{i-1}\left(\frac{{b}^1_{i-1}
                        \mathtt{X}_{i-1}^1}{a^{11}_{i-1}}\right )1_{({X}^1_i>L)}\mathsf{h}_i\right]\\
                &-2\mathbb{E}_{i-1}\left[{\mathbf{D}}_if_i\sigma^k_{i-1}{\mathbf{D}}_{i-1}\left (\frac{      {\sigma}^{1k} _{i-1}
                        \mathtt{X}_{i-1}^1
                }{a^{11}_{i-1}}\right )1_{({X}^1_i>L)}\mathsf{h}_i\right].
        \end{align*}
        The explicit calculation of the derivative in the second term of the above expression give the constant order term and the last term in $ \bar{e}_i $ which appear in \eqref{eq:td1bH}. In fact, note that 
        ${\mathbf{D}}_{i-1}\mathtt{X}_{i-1}=(\mathbf{e}^1 )^{\top} $, as $\mathtt{X}_{i-1}=X^1_{i-1}-L  $. Therefore, 
        \begin{align*}
                {\mathbf{D}}_if_i\sigma^k_{i-1}\left (\frac{      {\sigma}^{1k} _{i-1}{\mathbf{D}}_{i-1}
                        \mathtt{X}_{i-1}^1
                }{a^{11}_{i-1}}\right )=   {\mathbf{D}}_if_i\sum_{\ell=1}^d\frac{{a}^{1\ell}_{i-1}}{a^{11}_{i-1}}
                \mathbf{e}^\ell(\mathbf{e}^{1})^{\top}={\mathbf{D}}_if_i\mathbf{e}^{1}(\mathbf{e}^{1})^{\top}+{\mathbf{D}}_if_i\pi_{i-1}.
        \end{align*}
        This expression appears in \eqref{eq:td1bH}  and this expression is new in comparison with the one dimensional case.
\end{proof}

\subsection{Probabilistic representations of killed and reflected Brownian motion}\label{apkilres}
The goal of this section is to introduce a probabilistic representation for the first component of multidimensional correlated Wiener process with killing and likewise for the reflecting case. \
These results are used in  the definition of $ p_i $ in \eqref{maandmbar} and \eqref{eq:fiA}.

For this, we define $ U^j_t:=u^j+\sigma^{jk} W^k_t $, $ j=1,...,d $, $ k=1,...,r $ and we assume that $ \sigma\sigma^{\top}  =a $ is an invertible matrix. We define the stopping time  $ \tau:=\inf\{s>0;\ \mathbf{e}^1\cdot U_s=L\} $, where we assume without loss of generality that $\mathbf{e}^1\cdot u>L  $. We have the following probabilistic representation:
\begin{lemma}
        \label{th:9}
        Let $ f:H^d_L\to\mathbb{R} $ be a bounded measurable function such that $ f(x)=0 $ for $ x\in\partial H^d_L $.
        The following probabilistic representation is valid for the killed process semigroup with $ u^1>L\geq 0 $, $a_1:=\|\sigma_1\|^2$ and $ \sigma_1 :=(\sigma^{11},...,\sigma^{1d})$.
        \begin{align*}
                \mathbb{E}\left[f(U_T)1_{(\tau>T)}\right]=&\mathbb{E}\left[f(U_T)1_{(U^1_T> L)}\left(1-e^{-2\frac{(U_T^1-L)(u^1-L)}{a_{1}T}}\right)\right]
        \end{align*}
        In particular, we have that the following conditional probability characterization for $ z\geq L $
        \begin{align*}
                \mathbb{P}\left(\tau\leq T\Big|U^1_T=z\right)=e^{-2\frac{(z-L)(u^1-L)}{a_{1}T}}.
        \end{align*}
        In a similar fashion, we have for $ f:H^d_L\to\mathbb{R} $ a bounded measurable function and the normally reflected process $ V_t=U_t+\ell_t$ on $ H^d_L $. Then if $ f $ is a function of only the first component, we have 
        \begin{align*}
                \mathbb{E}\left[f(V^1_T)\right]=&\mathbb{E}\left[f(U^1_T)1_{(U^1_T> L)}\left(1+e^{-2\frac{(U_T^1-L)(u^1-L)}{a_{1}T}}\right)\right].
        \end{align*}
        
\end{lemma}
\begin{proof}
        The proof 
        of all the above result is obtained by considering the conditional expectation with respect to $ U^1 $ which is a process to be killed in the first result and a one dimensional reflected Wiener process with variance $ a_{1}$ in the second statement. Then all results follow from the 1D case and reflection principle (see e.g. \cite{Gobet} or \cite{Baldi} ) 
\end{proof}

\begin{remark}
        \label{rem:12}
        The generalization of the above results  to the representation  of $ \mathbb{E}[f(V_T)] $ is more complex than stated in the above expression. In fact, the full distribution for the correlated case is implicitly given in \cite{Lepingle}. A careful calculation shows that only in particular situations, such as the one assumed in {\bf Hypothesis \ref{hyp:12} },  a simple statement as the one above holds with $ V^1 $ and $ U^1 $ replaced by $ V $
and $ U $.

\end{remark}

\subsection{Moments of local times }

\label{sec:mlt}

Similar  to the 1D case, we define $  (\mathcal{X}^n ,\Lambda^{n})$ as the unique solution of the following hyperplane reflected equation 
\begin{align*}
        \mathcal{X}^{n}_t=&x+\int_{0}^t\sigma^{k}(\mathcal{X}^{n}_{\tau(s)})dW^k_s+\int_0^t\sum_{\ell=1}^d\frac{a^{\ell 1}}{a^{11}}(\mathcal{X}^{n}_{\tau(s)})\mathbf{e}^\ell d\Lambda^{n}_{ s}\in H^d_L,\\
        \Lambda^{n}_{ t}=&\int_{0}^t1_{(\mathcal{X}^1_s=L)}d|\Lambda^{n}|_{ s}\geq 0.
\end{align*}

Here $ |\Lambda^{n}| $ denotes the total variation of the real valued process $ \Lambda^{n} $. Furthermore, $ \tau(s)=t_i $ if $ s\in [t_i,t_{i+1}) $.
We also define $ \bar{\Lambda}^y$ as the local time of the one dimensional process $ ( \mathcal{X}^{n}\cdot\mathbf{e}^1)_{t\in [0,T]}$ at the point $ y>L $.

\begin{lemma}
        \label{lem:37}
        There exists a positive constant $C_q $ independent of $n, x, y,z$ such that for any $ q>0 $
        \begin{align*}
                \mathbb{E}\left[ \exp \left( q\max_{i=1,...N}|\mathcal{X}^{n}_{i}-x|\right) %
                \right]+\mathbb{E}\left[ \exp \left( q\Lambda^{n}_T\right) %
                \right]+\mathbb{E}\left[ \exp \left( q\bar{\Lambda}^y_T\right) %
                \right]\leq C_q.
        \end{align*}
        Also,  for any $ 0<s<t\leq T $, we have for $ \Lambda_{s\mapsto t}:=\Lambda_{ t}- \Lambda_{ s}$ and similarly for $ \bar{\Lambda} ^y$, there exists a constant $ c >0$ independent of $ n,y $ such that
        \begin{eqnarray}
                \mathbb{E}\left[ \left( \Lambda^{n} _{s\mapsto t}\right) ^{2}\right] &\leq
                &c\left( t-s\right) ,~~~\mathbb{E}\left[ \Lambda ^{n}_{s\mapsto t}\right]
                \leq c\sqrt{t-s}  \label{moduluso} \\
                \mathbb{E}\left[ \left( \bar{\Lambda}_{s\mapsto t}^{y}\right) ^{2}\right]
                &\leq &c\left( t-s\right) ,~~~\mathbb{E}\left[ \bar{\Lambda}_{s\mapsto
                        t}^{y}\right] \leq c\sqrt{t-s}. \label{modulusoo}
        \end{eqnarray}
\end{lemma}
\begin{proof}
        Note the process $ \mathcal{X}^{n}\cdot\mathbf{e}^1 $ satisfies the reflected 1D equation
        \begin{align*}
                 \mathcal{X}^{n}_t\cdot\mathbf{e}^1=&x^1+\int_{0}^t\sigma^{1k}(\mathcal{X}^{n}_{\tau(s)})dW^k_s+\Lambda^{n}_{ t}.
        \end{align*}
    The proof is as in the 1D case except that the use of the Dambis-Dubins-Schwartz theorem (see page 174 in \cite{KS}) is applied to the sum of stochastic integrals $ \int_{0}^t\sigma^{1k}(\mathcal{X}^{n}_{\tau(s)})dW^k_s $. Once the estimates for $ \Lambda^n $ are obtained one can also obtain the estimates for the other coordinates of $ \mathcal{X}^{n} $. In order to prove the moment properties for 
     \begin{align*}
        \mathcal{X}^{n}_t\cdot\mathbf{e}^\ell=&x^\ell+\int_{0}^t\sigma^{\ell k}(\mathcal{X}^{n}_{\tau(s)})dW^k_s+\int_0^t\frac{a^{\ell 1}}{a^{11}}(\mathcal{X}^{n}_{\tau(s)})\mathbf{e}^\ell d\Lambda^{n}_{ s},
    \end{align*}
one uses the previous results to bound the last term and additionally using that  $ \frac{a^{\ell 1}}{a^{11}} $ is bounded and that $ \Lambda^{n} $ is increasing together with the exponential moment bounds for the above stochastic integral which follows by application of the Dambis-Dubins-Schwartz theorem.
\end{proof}

As in the 1D case, we have the following corollary which is used when proving the finiteness of moments in the next section and expectations in Theorems \ref{th:24} and \ref{th:44}.

\begin{cor}
        \label{cor:38}
        a. For any constants  $ p,c>0 $, we have that 
        \begin{align*}
                &\sup_n\mathbb{E}\left[J_i^p\right]<\infty,\\
                J_i:=&\exp\left(\Delta^{-1/2}\sum_{i=1}^n\int_L^\infty e^{-\frac{(y-L)^2}{2c\Delta}}\bar{\Lambda}^y_{t_{i-1}\mapsto t_i}dy\right).
        \end{align*}
        b.  Let $(Y,B) $ be solution of the reflected equation (\ref{eq:Y}). Then, for
        any $q>0 $, we have that 
        \begin{align*}
                \mathbb{E}[e^{q\max_{s\in [0,T]}|Y_s|}]+ \mathbb{E}[e^{qB_T}]<\infty.
        \end{align*}
\end{cor}

%
%


\subsection{ Moment type inequalities}
\label{sec:app3}
The proof of the boundedness of the first derivative and the moment estimate for the corresponding matrix exponential follows directly from Lemma \ref{lem:60} below.

Lemma \ref{lem:60} below, is a variant of Lemma 27 of \cite{CK} obtained in the 1D case and is used to prove that products of matrices appearing in the first two derivatives are uniformly bounded. Recall that, for a $ d\times d $ matrix $ A $, its Frobenius norm is given by 
$ \|A\|_F^2:=\langle A,A\rangle_F$ with $\langle A,B\rangle_F:=\sum_{i,j=1}^dA_{ij}B_{ij}$ and that this norm satisfies the following monotonic property: For two positive definite symmetric matrices $ A,B $ such that $ B-A\geq 0 $ we have $ \|A\|_F\leq \|B\|_F $.


%

\begin{lemma}
        \label{lem:60}
        Let $ \mathsf{E} _i\in\mathcal{F}_{t_i}$, $ i=0,...,n$ be a sequence of random square matrices  which satisfy the following linear stochastic difference equation for $ \alpha_i,\ K_i\in\mathcal{F}_{t_i} $
        \begin{align*}
                \mathsf{E}_{i}=(I-K_i+\alpha_i)\mathsf{E}_{i-1}.
        \end{align*}
        Let $\mathsf{A}_{i+1}:=I-K_{i+1}$ and assume  that the matrix $\mathsf{A}_{i+1}^\top\mathsf{A}_{i+1}  -I  $ is negative definite.     
        Given $ p\in 4\mathbb{N} $
        and $C$ independent of 
        $n $, assume that $\alpha_i\in L^p (\Omega)$ and for $q= p,p/2 $, we assume the following inequality is satisfied

        \begin{align}
               &e^{-q{C}(\Delta+\Delta_iJ)}\left(\|\mathsf{E}_{i-1}\|^2_F+R_i\right)^{q/2}-\|\mathsf{E}_{i-1}\|_F^q
                \leq M_i(q)+\Upsilon_i(q),\label{eq:hp}\\
                R_i&:=2\langle \alpha_{i-1}\mathsf{E}_{i-1},(I+\frac {\alpha_{i-1}}2)\mathsf{E}_{i-1}\rangle_F-2
                \langle K_{i}\mathsf{E}_{i-1},\alpha_{i-1}\mathsf{E}_{i-1}\rangle_F.\notag
        \end{align}
        Here, $ M_i (q) $ and $\Upsilon_i(q)$ are $\mathcal{F}_{i}$ measurable random variables and  $\widetilde{\mathbb{E}}_{i-1}[M_i (q)]=0$.
        Furthermore let $ \tilde{C}_q\geq0 $ be a constant independent of $ n $ such that for any $ C\geq \tilde{C}_q $, there exists $ \bar{C}_q\geq 1$  and $ \delta_0(q)>0 $ such that for $ \Delta\leq \delta_0(q)$, $ i=1,..,n $ the following inequality is satisfied:
        \begin{align}
                \label{eq:hypl}
                \widetilde{\mathbb{E}}_{i-1}[   | \Upsilon_i(q)|]+\widetilde{\mathbb{E}}_{i-1}[M_i(q)^2]\leq& \bar{C}_p\|\mathsf{E}_{i-1}\|_F^q\Delta .
        \end{align}
        Then there exists positive constants $ n_0 $ and $ \mathsf{C}_p $ independent of $ n $ such that
        \begin{align*}
                \sup_{n\geq n_0}\tilde{\mathbb{E}}\left[\max_ie^{-\mathsf{C}_p(t_i+J_i)}\|\mathsf{E}_i\|_F^p\right]\leq \mathsf{C}_p.
        \end{align*}
\end{lemma}
\begin{proof}
        Most steps in this proof are the same as in the 1D case (see Lemma 27 in \cite{CK}), we only mention the changes required here for $ p=4$. To simplify expressions, let $ \mathsf{L}_i:=\prod_{j=1}^i\rho_j:=\prod_{j=1}^i\exp\left(-C(\Delta+\Delta_j J)\right) .$
        With this notation, one proceeds as in the 1D case for the matrix process $ \mathsf{L}_i\mathsf{E}_i $ as follows: Consider the differences \begin{align}
                \label{eq:ref}
                \|\mathsf{L}_{i+1}\mathsf{E}_{i+1}\|_F^4 -\|\mathsf{L}_i\mathsf{E}_i \|_F^4=
                \mathsf{L}_{i}^4\left(\rho_{i+1}^{4}\|\mathsf{E}_{i+1}\|_F^4-\|\mathsf{E}_{i}\|_F^4\right).
        \end{align}
        
        Using  the fact that the matrix $\mathsf{A}_{i+1}^\top\mathsf{A}_{i+1}  -I  $ is negative definite, we have 
        \begin{align*}
                0\leq& \|\mathsf{E}_{i+1}\|_F^2=\sum_{j,m,\ell,\ell'=1}^d\mathsf{E}_{i}^{\ell m }\mathsf{E}_{i}^{\ell' m }(\mathsf{A}_{i+1}+\alpha_{i+1})^{j\ell }
                (\mathsf{A}_{i+1}+\alpha_{i+1})^{j\ell' }\\
                &\leq \sum_{j,m,\ell,\ell'=1}^d\mathsf{E}_{i}^{\ell m }\mathsf{E}_{i}^{\ell' m }\left((I+\alpha_{i+1})^{j\ell}
                (I+\alpha_{i+1})^{j\ell' }-K_{i+1}^{j \ell }
                \alpha_{i+1}^{j \ell' }-
                \alpha_{i+1}^{j\ell }
                K_{i+1}^{j\ell' }\right)\\
                &=\|\mathsf{E}_{i}\|^2_F+R_{i+1}.
        \end{align*}

        One uses  the hypothesis \eqref{eq:hypl} with $ q=4 $ to find an upper bound for  \eqref{eq:ref}.
       The first step in order to obtain the estimates is to get a bound for 
        $ \max_i\widetilde{\mathbb{E}}\left[\|e^{-C(t_i+J_i)}\mathsf{E}_i\|_F^4\right] $
        using the above expression, the hypotheses \eqref{eq:hp}-\eqref{eq:hypl} and Gronwall's lemma.  
        The second and final step to bound $ \widetilde{\mathbb{E}}\left[\max_i\|e^{-C(t_i+J_i)}\mathsf{E}_i\|_F^4\right] $  uses a similar argument together with Doob's martingale inequality applied to the square and inequalities in the hypotheses \eqref{eq:hp}-\eqref{eq:hypl} for the case $ q=2 $ and the discrete version of the Gronwall inequality as in the 1D case. 
\end{proof}

The above result is used to prove Lemmas \ref{moments} and \ref{lemma:7}.

The following moment estimates are also required here to bound second derivatives and in the proof of Lemma \ref{lem:uc}. Its proof is the same as the corresponding  proof in the 1D case. 
\begin{lemma}
        \label{lem:35} Assume that $\mathsf{E} _i\in \mathcal{F}_i$, $i=1,...,n$, $ 
        \mathsf{E} _0=1$ is as stated in Lemmas \ref{moments} or \ref{lem:60}. That 
        is, we let $\mathsf{C}\geq 1 $ be a constant so that $\widetilde{\mathbb{E}}\left[
        (\max_ie^{-\mathsf{C}t_i}\|\mathsf{E} _i\|_F)^4\right]^{1/4}\leq \mathsf{C} $. Furthermore, suppose 
        that the sequence of r.v.'s $g_i\in\mathcal{F}_i $ satisfies $|g_i|\leq  
        \mathsf{M}e^{\mathsf{M}(T-t_i)} $ a.s. for some fixed constant $\mathsf{M}%
        >17\mathsf{C}$. Then for $\Delta\leq    ( \mathsf{M}-\mathsf{C})^{-1}\ln(\frac{\mathsf{M}-\mathsf{C}}{16\mathsf{C}}) $  
        \begin{align*}
                (a)\quad \left|\Delta \sum_{i=1}^n\widetilde{\mathbb{E}}\left[g_i\mathsf{E}
                _{i-1}\bar{M}^n_i\right] \right |\leq \frac{\mathsf{M}}8e^{\mathsf{M}T}.
        \end{align*}
        
        Assume now that the above constant $\mathsf{M}>0 $, also satisfies that $%
        \mathsf{M}>\bar C_2$ for a universal constant $\bar{C}_2>\mathsf{C} $
        depending only on the bounds for the coefficients. Then for $\Delta\leq
        \frac 5{\mathsf{M}-\mathsf{C}} $ and $\xi_{i}=\mathtt{X}_{i-1},\mathtt{X}_{i}, Z_i $, we have  
        \begin{align*}
                (b)\quad \left| \sum_{i=1}^n\widetilde{\mathbb{E}}\left[g_i\mathsf{E} _{i-1}
                \xi_{i}1_{(U_i\leq p_i)}\bar{M}^n_i\right]\right|\leq \frac {\mathsf{M}} 8e^{%
                        \mathsf{M}T}.
        \end{align*}
        
        Next, assume $Y_i\in\mathcal{F}_{i} $ satisfies for $p\geq 0 $, $\widetilde{ 
                \mathbb{E}}_{i-1}[Y_i] =O^E_{i-1}(\Delta^{p})$ and $\widetilde{\mathbb{E}}%
        _{i-1}[|Z_iY_i|] =O^E_{i-1}(\Delta^{p}) $. Then for $f\in C^1_b(H^d_L,%
        \mathbb{R} )$ 
        \begin{align*}
                (c)\quad \sup_{x\in H^d_L}\left|\widetilde{\mathbb{E}}_{0,x}\left[\sum_{i=1}^nf(X_i)%
                \mathsf{E} _{i-1}Y_i\bar{M}_n^n\right]\right|\leq C(\|f\|_\infty+\|f^{\prime
                }\|_\infty)\Delta^{p}.
        \end{align*}
        If $Y_i\geq 0 $, then $(c)$ holds with the weaker
        hypotheses: $\widetilde{\mathbb{E}}_{i-1}[Y_i] =O^E_{i-1}(\Delta^p)$ and $f\in
        C_b(H^d_L,\mathbb{R})$.
\end{lemma}
\subsection{The first derivative process: The proof of Lemma \ref{lem:24} }
\label{sec:bdd1}
An important consequence of Lemma \ref{lem:60} is the uniform boundedness of the derivatives ${\mathbf D}_i f_i$. The proof of this fact follows as in the 1D case following the corresponding decomposition formulas.

The boundedness of the first term in \eqref{fprimemH} is straightforward from Lemmas \ref{moments} and \ref{lemma:7}.

For the second term, we proceed by using some reduction implied by the result in Lemma \ref{lem:uc}. In fact, instead of proving the uniform  boundedness of $ \mathbb{E}\left[\sum_{j=i+1}^nf_j\varrho_j^{\top} E^{n}_{j-1} \bar{M}_{j}\right ] $ we can reduce it to the uniform boundedness of $$ \mathbb{E}\left[\sum_{j=i+1}^nf_j\frac{{b}^1_{j-1}
}{{a^{11}_{j-1}}}  \mathsf{h}_j
(\mathbf{e}^1)^\top E^{n}_{j-1} \bar{M}_{j}\right ] .$$

For this, define the following variables which are used to perform path decompositions:
\begin{align*}
        \bar{\mathsf{h}}_i\equiv\bar{\mathsf{h}}_{i:n} =&1_{(U_i\leq p_i,U_{i+1}>p_{i+1},...,U_n>p_n)}, i=1,...,n-1.
\end{align*}
Note that $ \sum_{j=0}^k\bar{\mathsf{h}}_{j:k}=1$ and 
$ \sum_{j=1}^{k}\bar{\mathsf{h}}_{j:k}=1 -1_{(\mathsf{h}_j= 0,j=1,...,k)}$. With these properties, one can rewrite 

\begin{align*}
        &\left |\sum_{j=i+1}^{n}\tilde{\mathbb{E}}_{i,x}\left[ f_{j}\left( X_{j}^{n,x}\right)\frac{{b}^1_{j-1}
        }{{a^{11}_{j-1}}} 
       (\mathbf{e}^1)^\top E_{i:j-1}^{n}\mathsf{h}_j\bar{M}_{i:j}{{\mathcal{K}}}_{i:j}\right]\right |\leq 
        \tilde{\mathbb{E}}_{i,x}\left[| f|(X^{n,x}_n)\left| \sum_{j=i+1}^{n}\frac{{b}^1_{j-1}
        }{{a^{11}_{j-1}}} (\mathbf{e}^1)^\top E_{i:j-1}^{n}\bar{\mathsf{h}}_{j}\right|\bar{M%
        }_{i:n}{{\mathcal{K}}}_{i:n}\right].
\end{align*}

Now, using $\left|\sum_{j=i+1}^{n}\frac{{b}^1_{j-1}
}{{a^{11}_{j-1}}} (\mathbf{e}^1)^\top E_{i:j-1}^{n}\bar{\mathsf{h}}_{j}\right|\leq C\max_{j\geq i+1}
\|E_{i:j-1}^{n}\|_F$ and the moment properties in Lemma \ref{lemma:7}, we obtain:

\begin{align*}
        \left|\sum_{j=i+1}^{n}\tilde{\mathbb{E}}_{i,x}\left[ f_{j}\left( X_{j}^{n,x}\right)
        E_{i:j-1}^{n}\mathsf{h}_j\bar{M}_{i:j}{{\mathcal{K}}}_{i:j}\right]%
        \right|\leq C\|f\|_\infty.
\end{align*}

\subsection{The properties of the second derivative}
\label{sec:7.6}
This section gives the essential part of the argument in proof of Theorem \ref{th:24} in order to prove the convergence of the approximation to the  first derivative  of $ \mathbb{E}[f(X_T)1_{(\tau>T)}] $. Recall that in the argument of the proof, we stated the existence of a function $ \phi_n $ which is defined as 
\begin{align}
        \phi_n(x):=
        {\tilde{\mathbb{E}}}_{0,x}\left[  {\mathbf{D}} {f}\left( X_{n}\right)E^{n}_n\mathcal{K}_n\bar{M}_{n}                \right] + {\tilde{\mathbb{E}}}_{0,x}\left[\sum_{i=1}^n f_i(\mathbf{e}^1)^{\top}  E^{n}_{i-1}\frac{b^1_{i-1}
        }{{{a^{11}_{i-1}}}}       
        \mathsf{h}_i\mathcal{K}_i\bar{M}_{i}\right ].
        \label{eq:DD}
\end{align}

In fact, $ \mathbf{D}_x\varphi_n-\phi_n $ converges uniformly to zero as explained at the beginning of 
Section \ref{sec:bdd1}

%

Therefore, it is enough to obtain a bound which is independent of $ n $ for the derivative of $ \mathbf{D}_x\phi_n(x) $ where

Recall that the random matrix $E^n_{i-1}=\prod_{j=1}^{i-1} e_i $ is the one given in \eqref{e} and \eqref{eq:td1bH}.
Also in order to write in compact form various expressions we define the operator $ \mathbf{D}^\top f=(\partial_1f,...,\partial_df)^\top $.


\begin{proposition}
        \label{prop:30}  There exists a universal constant $\mathsf{M} $ which
        depends only on  the constants of the problem such that for all $i=0,...,n-1 $  
        \begin{align}  \label{eq:87a}
                \sup_{x\in H^d_L}\left |\mathbf{D}^\top_{x}\mathbb{E}_{i,x}\left[\mathbf{D} {f}\left( X_{n}^{n}\right)
                E_{i:n}^{n}\bar{M}_{i:n}^{n}\right]\right| +\sup_{x\in H^d_L}\left |\sum_{j=i}^{n}\mathbf{D}^\top_{x}\mathbb{E}
                _{i,x}\left[ f_{j}\left( X_{j}^{n}\right) (\mathbf{e}^1)^{\top}E_{i:j-1}^{n}\frac{b^1_{i-1}
                }{{{a^{11}_{i-1}}}}       \mathsf{h}_{j}\bar{M}
                _{i:n}^{n}\right]\right|\leq \mathsf{M}e^{\mathsf{M}(T-t_{i})}.
        \end{align}
        In particular, this implies that $ \|\mathbf{D}_i^2f_i\|_\infty  \leq \mathsf{M}e^{\mathsf{M}(T-t_{i})}$.
\end{proposition}

To prove Proposition \ref{prop:30}, we will need some careful calculations with matrices and with this in mind, we introduce the following notation: For $f\in C^1(\mathbb{R}^d,\mathbb{R}^d)$, $ \lambda_i=\lambda^0(X_{i-1},Z_i)+\mathsf{h}_i\lambda^1(X_{i-1},Z_i)$, with $\lambda^0,\lambda^1\in C^1(\mathbb{R}^d\times \mathbb{R}^d ,\mathbb{R}^{d\times d}) $  
\begin{align*}
        {\mathbf{D}}^{\dagger}_{i-1}(f,\lambda_i)_{jk}:=&f^r\left(\partial_{x_j} (\lambda^0_i)^{rk}+\mathsf{h}_i\partial_{x_j}(\lambda^1_i)^{rk}\right)\\
        D^k[f,\Delta_u\lambda]_{j}:=&f^rD^k (\Delta_u\lambda^{rj})=f^r\left(\partial_{z_k} (\lambda^0_i)^{rj}+2\partial_{z_k}(\lambda^1_i)^{rj}\right)\\
        \Delta_u\lambda:=&\lambda^0+2\lambda^1
\end{align*}
In the second formula, we have extended the definition of $ D^k $ to be applied to pairs $ [f,\Delta_u\lambda] $. Recall that $ D^k$ denotes the derivative with respect to the $ k $-th component of the corresponding Gaussian noise in $ \lambda^0 $ and $ \lambda^1 $ and we stress that $ \partial_{x_j}(\lambda^1_i)^{rk} $ is the partial derivative of the function $ (\lambda^1_i)^{rk}(x,z) $ with respect to $ x_j $.

For the next result, we need the following definitions:
  For a square matrix $ H $,
\begin{align*}
        \bar{A}^1_i(\lambda_i)(H):=&(\mathbf{D}_{i-1}X_i)^{\top}  H ^\top\lambda_i-\left( {\mathbf{D}}_{i-1}^{\top}  \left (\frac{        {\sigma}^{1k} _{i-1}
                \mathtt{X}_{i-1}^1
        }{{a^{11}_{i-1}}}\right )({\sigma}^k_{i-1})^{\top}  \right )H^\top \Delta_u\lambda_i\mathsf{h}_i\\
        \tilde{A}^0_i(\lambda_i,\mathbf{D}_if_i):=&{\mathbf{D}}^{\dagger}_{i-1}({\mathbf{D}}_if_i,{\lambda}_{i})
- D^k[{\mathbf{D}}_if_i,\Delta_u{\lambda}_{i}]{\mathbf{D}}_{i-1} \left (\frac{      {\sigma} ^{1k}_{i-1}
        \mathtt{X}_{i-1}^1
}{{a^{11}_{i-1}}}\right )\mathsf{h}_i\notag\\
        \bar{A}^0_i(\lambda_i,{\mathbf{D}}_if_i):=&\tilde{A}^0_i(\lambda_i,\mathbf{D}_if_i)
-
        {\mathbf{D}}_{i-1}^\top  \left(\frac{{b}^1_{i-1}
                \mathtt{X}_{i-1}^1}{{a^{11}_{i-1}}}\right ){\mathbf{D}}_if_i\Delta_u{\lambda}_{i}
        \mathsf{h}_i.
\end{align*}
In the case that $ \lambda_i $ is real valued and $ H $ is a vector, we extend the above definition of $       \bar{A}^1_i(\lambda_i)(H) $ using the same formula. A similar remark applies to $ \bar{A}^0_i(\lambda_i,{\mathbf{D}}_if_i) $.

We also define the space for $ x\in \mathbb{R}^d $, $ y,z\in\mathbb{R} $ 
\begin{align*}
        P^k_{p,i}(x,y,z):=\left\{
        \sum_{r_j\geq 0,r_1+...+r_{d+2}\leq p}\!\!\!\!\!\!\!\!\!\!\!\!\!\!\!\!\zeta_r(X_{i-1})
        z^{r_{d+2}}y^{r_{d+1}}\prod_{j=1}^d(x^j)^{r_j},\zeta_r\in C^k_b(\mathbb{R}^d,\mathbb{R}^{d\times d}), r=(r_1,...,r_{d+2})\right\}.
\end{align*}
When we further impose the condition that $ \zeta_0\equiv 0 $, we denote the corresponding space by $ \bar{P}^k_{p,i}(x,y,z) $. We will mainly use these spaces with $ p=1,2 $.

These spaces will be used to characterize the qualitative properties of coefficients that appear in different derivative formulas without writing the coefficients explicitly. We emphasize that the elements of this space are polynomials in the variables $ (x,y,z) $ with coefficients that depend on $ X_{i-1} $. We will also need this definition with less number of variables and in that case we will use for example,   $ \bar{P}^k_{p,i}(x,y)$ when the polynomials are only products of the coordinates in $ (x,y) $. As with constants we may use the notation $ \mu_i $ for elements of the spaces $ {P}^k_{p,i}(x,y,z) $ or $ \bar{P}^k_{p,i}(x,y) $ as in the result below. Their exact formulas may change from one line to the next.

An important point in order to reduce the amount of calculations is to change the probabilistic representations in order to prove \eqref{eq:87a}. With this in mind, we will temporarily assume that 

{\bf (H)} $ \sigma_{i-1}  $, $ i=1,...,n $ is a lower triangular matrix. 

Lemma \ref{lem:38} shows that this assumption can be removed. 
Using this hypothesis, we will use expressions for the conditional laws as follows: $ \widehat{X}_i$ denotes the expression of the random variables $ X_i^\ell $, $ \ell\neq 1 $ when they are conditioned to $ X_i^1=L $. More explicitly, define for any $ x\in\mathbb{R}^d $, $ {x}^\wedge_i:=(x^2_i,...,x^d_i)$ and $ \hat{X}^1_i=L $
\begin{align}
        \widehat{X}^j_i=&\widehat{X}^j_{i-1}+\sum_{k=2}^d({\sigma}_{i-1}^{k})^{j}{Z}^k_i +(\sigma^1_{i-1})^{j}\frac{L-X_{i-1}^1}{\sqrt{a^{11}_{i-1}}},\ j=2,...,d-1\label{eq:ort1}
\end{align}
In some cases, we may need to use the above decomposition for $ X_i $
With these definitions, we obtain the following push forward formula for second derivatives:
\begin{lemma}[{\bf Modified push forward formula for second derivatives}]
        \label{lem:43}Assume that $ \sigma(x)  $ is lower triangular matrix for all $ x\in H^d_L $.
        Suppose that $ f_i\in C^2_b(\mathbb{R}^d, \mathbb{R}) $ and $ \lambda_i=e_i(1+\kappa_i+\frac{1}{2}\kappa_i^2) $. Then we have 
        \begin{align}
                {\mathbf{D}}^\top_{i-1}{\tilde{\mathbb{E}}}_{{i-1}}\left[{\mathbf{D}}_i{f}_{i}\lambda_i\bar{m}_{i}\right]={\tilde{\mathbb{E}}}_{{i-1}}\left[\left(\bar{A}^1_i(\lambda_i)(\mathbf{D}_i^2f_i)+\tilde{A}^0_i(\lambda_i,\mathbf{D}_if_i)\right)\bar{m}_i\right]+{\tilde{\mathbb{E}}}_{{i-1}}\left[{\mathbf{D}}_i{f}_{i}(\widehat{X}_i)\cdot\mathbf{e}^1\mu^1_i\right].
                \label{eq:exp1}
        \end{align}
        
        Furthermore, $ \mu_i^{1}=g^1_i(\mathtt{X}_{i-1}^1)\nu_i^1$ with $\nu_i^1\in \bar{P}^1_{2,i}(Z_{i}^{\wedge} ,\Delta,\mathtt{X}_{i-1}^1 )
        $ a matrix valued random variable which gives that  $ \tilde{\mathbb{E}}_{i-1}\left[\left|\mu_i^{1}\right|\bar{m}_i\right]\leq C$. 
        
        Also, we have for a matrix valued random variable $ \mu^2_i\in  \bar{P}^0_{2,i}(Z_i,\Delta,\mathtt{X}^1_{i-1}\mathsf{h}_i) $ that
        \begin{align}
                \tilde{A}^0_{i} (\lambda_{i}, \mathbf{D}_if_{i}\mathbf{e}^1(\mathbf{e}^1)^\top)=& -({\mathbf{D}}_{i} f_{i} \cdot \mathbf{e}^1) 
                \left(\frac{{b}^1_{{i} -1}
                }{{a^{11}_{{i} -1}}}\right )\mathbf{e}^1(\mathbf{e}^1)^\top\mathsf{h} _{i}+{\mathbf{D}}_{i} f_{i}\cdot \mathbf{e}^1\mu^2_i.
                \label{eq:exp2}
        \end{align}
\end{lemma}

This formula is called `modified' because the boundary condition $ \mathbf{D}_if_i(x)=0 $ for $ x\in H^d_L $ is not satisfied as in Lemma \ref{lem:2.1H}. We will see later in Lemma \ref{lem:38} that the hypothesis that $ \sigma (x) $ is lower triangular in the above result is not binding.  The proof of this technical lemma is done later. 
 
{\it The iteration of the push-forward derivative formulas for second derivatives:}
We discuss now the case of the iteration of the above formula when one applies repeatedly the linear operator $  \bar{A}^1_i(\lambda_i)(H) $, for $ i=1,...,n $. 
As we already stated, the iteration of these operators can not be written explicitly as a matrix product. Instead, we will write them inductively using the following sequence of backward linear matrix system for $k=\ell-1,...,1$ and with final condition $ A\in\mathbb{R}^{d\times d} $
\begin{align}
        F_{k-1,\ell}=&\bar{A}^1_k(\lambda_k)(F_{k,\ell})\label{eq:iter10}\\
        F_{\ell,\ell}:=&A.\notag
\end{align}
When we want to stress the final condition, we may write $      F_{k-1,\ell}\equiv         F_{k-1,\ell}(A)  $.
\begin{cor}
        \label{cor:43} With the above notation, the iteration of the push-forward derivative formula in  Lemma \ref{lem:43} gives        \begin{align}
                \!\!\!\! {\mathbf{D}}^\top_{0}{\tilde{\mathbb{E}}}\left[{\mathbf{D}}_n{f}_{n}
                E^n_n\bar{M}^n_{n}\right]=      \tilde{\mathbb{E}}\left[F_{0,n}
                (\mathbf{D}^2_nf_n)\bar{M}^n_{n}\right]+
                \tilde{\mathbb{E}}\left[
                \sum_{i=1}^nF_{0,i-1}(\tilde{A}^0_i(\lambda_i,
                \mathbf{D}_if_i)+
                {\mathbf{D}}_i{f}_{i}(\widehat{X}_i)\mathbf{e}^1(\mathbf{e}^1)^\top\mu^1_i)
                \bar{M}^n_{n}\right].
                \label{eq:2der}
        \end{align}
        Here, $ \mu_i^{1}=g^1_i(\mathtt{X}_{i-1}^1)\nu_i^1$ with $\nu_i^1\in \bar{P}^1_{2,i}(Z_{i}^{\wedge} ,\Delta,\mathtt{X}_{i-1}^1 )
        $ and $ \tilde{\mathbb{E}}_{i-1}\left[\left|\mu_i^{1}\right|\bar{m}_i\right]\leq C$.
\end{cor}
We remark here that as $ {\mathbf{D}}_i{f}_{i}(\widehat{X}_i)\mathbf{e}^1(\mathbf{e}^1)^\top\mu^1_i $ is independent of $ X_i^1 $, one can consider it multiplied by $ \bar{m}_i $ or not as $ \tilde{\mathbb{E}}_{i-1}[\bar{m}_i]=1$ and therefore both conditional expectations are the same. In the above expression we have preferred the latter as it gives a compact formula.
Now that we have obtained the one step push forward second derivative formula in Lemma \ref{lem:43}, we have to iterate it for the values of $ i=1,...,n $ and show that their iteration is uniformly bounded. From now on, we will always assume that $ \lambda_i=e_i(1+\kappa_i+\frac{1}{2}\kappa_i^2) $.

As we have previously mentioned a large part of the order analysis shares common arguments with the 1D case. With this in mind, we will not give the explicit calculations, instead we explain the arguments that are required.

{\it Proof of Proposition \ref{prop:30}, Part I: Boundedness of the first term in \eqref{eq:87a}.}

Given the result in Corollary \ref{cor:43}, we have that in order to obtain the boundedness of the first term in \eqref{eq:87a} is enough to  prove that the right hand side of \eqref{eq:2der} is uniformly bounded by the explicit bound stated in \eqref{eq:87a}. In order to do this, we use (often) the results in Lemma \ref{lem:35} as we did in the 1D case. We do this without further mentioning of it. 

In order to understand some of the definitions to follow and to gain some insight on the proof, we start by writing the constant term $ \xi_i(H) $ (in other words, terms that do not depend on $ Z $, $ \Delta $ or $ \mathtt{X}^1\mathsf{h} $) in the explicit expression for $ \bar{A}^1_i(\lambda_i)(H) $ where $ H $ is an arbitrary square matrix. This is given by
\begin{align}
        \label{eq:H}
        \xi_i(H):=&
        H(I+(\mathbf{e}^1(\mathbf{e}^1)^{\top}-I) \mathsf{h}_i)-\mathbf{e}^1(\mathbf{e}^1)^{\top} H(2\mathbf{e}^1(\mathbf{e}^1)^{\top}-I)\mathsf{h}_i\\=&    IHI-(I-\mathbf{e}^1(\mathbf{e}^1)^{\top})H(I-\mathbf{e}^1(\mathbf{e}^1)^{\top})\mathsf{h}_i-\mathbf{e}^1(\mathbf{e}^1)^{\top} H\mathbf{e}^1(\mathbf{e}^1)^{\top}\mathsf{h}_i.\nonumber
\end{align}
The following property will be important in what follows: $ \xi_i(\mathbf{e}^\ell(\mathbf{e}^m)^{\top})=\mathbf{e}^\ell(\mathbf{e}^m)^{\top}(1-1_{(\ell=m=1\text{ or }\ell\neq 1,m\neq 1)}\mathsf{h}_i) $.   This property expresses the fact that once the matrix $ H $ is parallel to one of the basis matrices $ \mathbf{e}^\ell(\mathbf{e}^m)^{\top} $, the constant order terms  remain in every step of the backward iteration when applying the matrix operator $ \xi(H) $. This will lead to the product of the  random variables $ m_i$ when $ \ell=m=1 $ or $ \ell\neq 1 $, $ m\neq 1 $ and $ \bar{m}_i $ otherwise. Heuristically speaking in the first case ``reflection'' is transformed into ``killing'' while in the second ``reflection'' remains unchanged.

Therefore the next step in the argument is to use the iteration of matrix operators $ F $ over matrix decompositions using the following basis 
\begin{align*}
        \{\mathbf{e}^\ell;\ell=1,...,d\}\times\{\mathbf{e}^\ell;\ell=1,...,d\}.
\end{align*}

We divide the proof in two parts.

{\it Step I: Uniform boundedness of the first term in \eqref{eq:2der}:}
The first goal is to prove 
\begin{align}
        \label{eq:goal}
        \sup_n\|{\tilde{\mathbb{E}}}[F_{0,n}(\mathbf{D}^2_nf_n)\bar{M}^n_n]\|_M<\infty .
\end{align}
Here, recall that $ F $ is defined in \eqref{eq:iter10} and we let $ \|A\|_M=\max_{ij}|A_{ij}| $.  In order to prove the above statement, we will use a Gronwall type argument using the ``coordinate'' processes 
\begin{align*}
        \mathbf{e}^l(\mathbf{e}^l)^{\top} \mathbf{D}^2_\ell f_\ell \mathbf{e}^m(\mathbf{e}^m)^{\top}=
        \left(\mathbf{D}^2_\ell f_\ell\right)^{l m}\mathbf{e}^l(\mathbf{e}^m)^{\top},\quad  l,m=1,...,d,
\end{align*} and the linearity of the values of the sequence $ F_{k-1,\ell} $ with respect to its final condition.
With this in mind, we define the following auxiliary backward inductive system for $ k=\ell-1,...,1 $ related to \eqref{eq:H} for a $d\times d $ matrix $ A $:
\begin{align*}
        G_{k-1,\ell}=&G_{k,\ell}(I+(\mathbf{e}^1(\mathbf{e}^1)^{\top}-I) \mathsf{h}_k)-\mathbf{e}^1(\mathbf{e}^1)^{\top} G_{k,\ell}(2\mathbf{e}^1(\mathbf{e}^1)^{\top}-I)\mathsf{h}_k,\\
        G_{\ell,\ell}:=&A.
\end{align*}
When we want to stress the final condition, we denote the solution of this system by $ G_{k,\ell}\equiv G_{k,\ell}(A) $.

The norm estimates will be based on the following telescopic formula:
\begin{align}
        F_{0,n}(\mathbf{e}^l(\mathbf{e}^m)^{\top})^{*}   =&G_{0,n}(\mathbf{e}^l(\mathbf{e}^m)^{\top})^{*}+\sum_{k=0}^{n-1}\left(F_{0,k+1}(G_{k+1,n}(\mathbf{e}^l(\mathbf{e}^m)^{\top}))^{*}-
        F_{0,k}(G_{k,n}(\mathbf{e}^l(\mathbf{e}^m)^{\top}))^{*}
        \right).
        \label{eq:tel}
\end{align}
Here, in order to simplify the notation we have used $ A^{*}:=AA^{\top} $.
Next, straightforward algebraic calculations lead to 
the following properties

\begin{enumerate}[label=\large\protect\textcircled{\small\arabic*}]
        \item\begin{align}
                G_{k-1,n}(\mathbf{e}^l(\mathbf{e}^m)^{\top})=\mathbf{e}^l(\mathbf{e}^m)^{\top}\begin{cases}
                        \prod_{i=k}^n(1-\mathsf{h}_i)&\text{ if either }(l=m=1)\text{ or }(l\neq 1\text{ and }m\neq 1)\\
                        1&\text{ otherwise.}
                \end{cases}
                \label{eq:cases}
        \end{align}
        \item\begin{align*}
                \tilde{\mathbb{E}}_{k}\left[\left(F_{k,k+1}(G_{k+1,n}(\mathbf{e}^l(\mathbf{e}^m)^{\top}))^{*}-
                G_{k,k+1}(G_{k+1,n}(\mathbf{e}^l(\mathbf{e}^m)^{\top}))^{*}\right)\bar{M}_{k:n}\right] =O_{k}^E(1).
        \end{align*}
\end{enumerate}
In fact,
\begin{align*}
        F_{k,k+1}(G_{k+1,n}(\mathbf{e}^l(\mathbf{e}^m)^{\top}))^{*}-
        G_{k,k+1}(G_{k+1,n}(\mathbf{e}^l(\mathbf{e}^m)^{\top}))^{*}\in \bar{P}^0_{2,i}(Z_i,\Delta,\mathtt{X}_{i-1}\mathsf{h}_i).
\end{align*}
That is, the above difference can be expressed as quadratic polynomials with no constants terms and whose coefficients are $ \mathcal{F}_{k} $ adapted and uniformly bounded. These polynomials when multiplied by $ \bar{M}_{k:n} $ depend on increments of the type $ Z^{-}_{k+1}{M}_{k:n}$, $ \Delta {M}_{k:n}$, $ \mathtt{X}^1_k\mathsf{h}_{k+1}{M}_{k+1:n} $ in the cases that $ (l=m=1)\text{ or }(l\neq 1\text{ and }m\neq 1) $ and $ Z^{-}_{k+1}\bar{M}_{k+1}$, $ \Delta {m}_{k+1}\bar{M}_{k+1}$, $ \mathtt{X}^1_k\mathsf{h}_{k+1}\bar{M}_{k}$ in the remaining cases. The estimates for the latter case have been studied in Sections \ref{sec:4} and \ref{sec:5}. The estimates for the former case can be obtained as a small variation of the argument in one dimension which appears in the proof of Lemma 35 in \cite{CK}. In fact, while in Lemma 35 the estimated probabilities are for hitting the boundary in some partition time interval $ [t_k,t_{k+1}] $ starting from $ t_i<t_k $, in the present case we need to estimate the conditonal probability that one hits the boundary after time $ T $ starting from $ t_k $. Changing the time intervals in the arguments of Lemma 35 in \cite{CK} gives {\large \textcircled{\small 2}}.

The proof of \eqref{eq:goal} is obtained through a Gronwall type argument. For this, we use the norm
\begin{align*}
        \|F_{0,j}\|_{*}^2:=\sup_{\|A\|_M\leq 1}\Tr\left(
        {\tilde{\mathbb{E}}}\left[\frac{F_{0,j}(A)^{*}}{\|A\|_M^2}\bar{M}_n^n\right]\right).
\end{align*}
In order to obtain the Gronwall type inequality, one takes expecations in the telescopic sum \eqref{eq:tel}. Next, one uses the tower property of expectations and 
the linearity of $ F_{0,k} $, 
with the estimates obtained in the property {\large \textcircled{\small 2}},
to obtain the following inequality for a sufficiently large constant $ C $
\begin{align*}
        \|e^{-CJ_j}F_{0,j}\|_{*}^2\leq C\left(1+\sum_{k=0}^{j-1}\|e^{-CJ_k}F_{0,k}\|_{*}^2\Delta\right).
\end{align*}
Therefore one obtains from Gronwall's lemma that there exists $ C>0 $ such that
\begin{align}
        \sup_n\max_j    \|e^{-C(t_j+J_j)}F_{0,j}\|_{*}^2\leq C.
        \label{eq:G}
\end{align}
From \eqref{eq:G}, one readily obtains that $  \sup_n\|{\tilde{\mathbb{E}}}[F_{0,n}(\mathbf{D}^2_nf_n)\bar{M}^n_n]\|_M\leq |\mathbf{D}^2f|_\infty\|F_{0,n}\|_{*}$. From here, the uniform boundedness of the first term follows using Corollary \ref{cor:38}. 

{\it Step II: Uniform boundedness of the second and third term in \eqref{eq:2der}:}
In the following proof besides using the arguments in Step I, we will also use the path decomposition method in Lemma 35 of \cite{CK}.
We will use the estimates obtained for $ F_{0,j} $ and the goal is to give an upper bound for the norm of the additional sum with 
index $ i $. 

First, we decompose $  {\mathbf{D}}_{i} f_{i}$ in components and study two cases as before. We start considering terms within
$ {\mathbf{D}}_{i} f_{i}\cdot \mathbf{e}^\ell $ for $ \ell\neq 1$ which have  $\mathcal{F}_{{i}-1} $-conditional expectation of order $ O^E_{{i}-1}(1) $.
In fact\footnote{See also the proof of Step A.1.},  one readily has that $ {\mathbf{D}}_{i} f_{i}\cdot \mathbf{e}^\ell\big |_{\partial H^d_L}=0$ as it is similar to an expression which is multiplied by $ m_i $, because 
$\left(I+(\mathbf{e}^1(\mathbf{e}^1)^{\top}-I)\mathsf{h}_{{i}+1}\right ) \mathbf{e}^\ell=\mathbf{e}^\ell(1-\mathsf{h}_{{i}+1}) $. Therefore, we have using the mean value theorem
\begin{align*}
        {\mathbf{D}}_{i} f_{i}\cdot \mathbf{e}^\ell\mathsf{h}_i =\left({\mathbf{D}}_{i} f_{i}\cdot \mathbf{e}^\ell -{\mathbf{D}}_{i} f_{i}(\hat{X}_i)\cdot \mathbf{e}^\ell\right) \mathsf{h}_i=O_{i-1}^E(1).
\end{align*}
With the above results, we consider the uniform boundedness of $  \tilde{\mathbb{E}}\left[
\sum_{i=1}^nF_{0,i-1}(\tilde{A}^0_i(\lambda_i,
\mathbf{D}_if_i\cdot \mathbf{e}^\ell(\mathbf{e}^\ell)^\top))
\bar{M}^n_{n}\right] $ in \eqref{eq:2der} for $ \ell\neq 1 $.
In fact, one obtains the uniform boundedness easily by adding and subtracting the value of $ {\mathbf{D}}_{i} f_{i}(\hat{X}_{i})\cdot \mathbf{e}^\ell\mathsf{h}_i=0$ when needed. That is, for some $ \mu^0_i\in P^0_{2,i} (Z_i,\Delta,\mathtt{X}_{i-1}\mathsf{h}_i)$ and $ \mu^1_i\in \bar{P}^0_{2,i}(Z_i,\Delta,\mathtt{X}_{i-1}\mathsf{h}_i) $, we have applying the definition of $   \tilde{A}^0_i $ that
\begin{align*}
      \tilde{A}^0_i(\lambda_i,
      \mathbf{D}_if_i\cdot \mathbf{e}^\ell(\mathbf{e}^\ell)^\top)=\left({\mathbf{D}}_{i} f_{i}\cdot \mathbf{e}^\ell -{\mathbf{D}}_{i} f_{i}(\hat{X}_i)\cdot \mathbf{e}^\ell\right) \mathsf{h}_i\mu^0_i+{\mathbf{D}}_{i} f_{i}\cdot \mathbf{e}^\ell\mu^1_i.
\end{align*}

Now, we analyze the term $ \tilde{A}_i^0(\lambda_i,\mathbf{D}_if_i\mathbf{e}^1(\mathbf{e}^1)^\top)  $. To do this, we discard from further consideration all terms whose $ \mathcal{F}_{{i}-1} $-conditional expectation are of order $ O^E_{{i}-1}(1) $ from $ \tilde{A}^0_{i} (\lambda_{i}, \mathbf{D}_if_{i})$ as stated in \eqref{eq:exp2} and then we need to find the order of the remaining term. Explicitly, we have for $ \mu^2_i\in  \bar{P}^0_{2,i}(Z_i,\Delta,\mathtt{X}_{i-1}\mathsf{h}_i) $:
\begin{align*}
        \tilde{A}^0_{i} (\lambda_{i}, \mathbf{D}_if_{i}\mathbf{e}^1(\mathbf{e}^1)^\top)=&B_i+{\mathbf{D}}_{i} f_{i}\cdot \mathbf{e}^1\mu^2_i\\
         B_{i}:=& \nu({\mathbf{D}}_{i} f_{i}, X_{{i}-1})\mathbf{e}^1(\mathbf{e}^1)^\top\mathsf{h} _{i}:= -({\mathbf{D}}_{i} f_{i} \cdot \mathbf{e}^1) 
        \left(\frac{{b}^1_{{i} -1}
        }{{a^{11}_{{i} -1}}}\right )\mathbf{e}^1(\mathbf{e}^1)^\top\mathsf{h} _{i}.
\end{align*}
In order to analyze the term $ B_i $, we need to analyze the coordinate processes as we did in the argument to obtain \eqref{eq:G} using $ \nu({\mathbf{D}}_{i} f_{i}, X_{{i}-1})F_{k,{i-1}}(\mathbf{e}^1(\mathbf{e}^1)^{\top}) $. 
Here, we remark that $ \nu(x,y) \in\mathbb{R}$, $ x,y\in\mathbb{R}^d $ is linear in $ x $.


Recall that 
the value $ F_{\ell,i-1} $ is the solution of the system \eqref{eq:iter10} through the backward iteration on the formula in Lemma \ref{lem:43} from $ i-1,...,\ell $. 

For the rest of the proof we need to use 
an auxiliary backward inductive system of a similar type as in Step I which is defined as follows for $ k=r+1,..,i-1,$ 
\begin{align*}
        G_{k-1,{i}}=&G_{k,{i}}(I+(\mathbf{e}^1(\mathbf{e}^1)^{\top}-I) \mathsf{h}_k)-\mathbf{e}^1(\mathbf{e}^1)^{\top} G_{k,{i}}(2\mathbf{e}^1(\mathbf{e}^1)^{\top}-I)\mathsf{h}_k\\
        G_{{i},{i}}=&C_{r,{i}}:=\mathsf{h}_{i}\nu({\mathbf{D}}_{i} f_{i}(\hat{X}_r), X_{r})\mathbf{e}^1(\mathbf{e}^1)^{\top}.
\end{align*}
The rest of the argument is similar to the 1D case and it  follows the path decomposition method of Lemma 35 in \cite{CK}. In fact, according to \eqref{eq:cases}, and as the matrix under analysis is a multiple of $ \mathbf{e}^1(\mathbf{e}^1)^\top $ we deduce that each time one iterates using the above formula for $ G $, the multiplying factor $ 1-\mathsf{h}_k $ will appear.

Using the same arguments as in Lemma 35 in \cite{CK}, we may apply the path decomposition method with    
\begin{align*}
        \mathsf{h}_{r,k} :=&1_{(U_r\leq p_r,U_{r+1}>p_{r+1},...,U_{k-1}>
                p_{k-1},U_k\leq p_k)} \\
        \widehat{\mathsf{h}}_{r,k}:=&1_{(U_r> p_r,...,U_{k-1}>p_{k-1},U_k\leq p_k)}.
\end{align*}
Note that $  \mathsf{h}_{r,k}+\widehat{\mathsf{h}}_{r,k}=\widehat{\mathsf{h}}_{r+1,k}$.

With this property one obtains the following decomposition of $ \hat{P}_{\ell,r}:= F_{\ell,r}(G_{r+1,i}(C_{r,i}))$:
\begin{align*}
        F_{\ell,i-1}(B_i)\mathsf{h}_i
        =&\left(F_{\ell,i-1}(B_i)-F_{\ell,i-1}(G_{i,i}(C_{i-1,i}))\right)
        \mathsf{h}_i+
        \hat{P}_{\ell,i-1}\mathsf{h}_{i-1,i}+
        \Delta_{i-1}
        \hat{P}_{\ell,\cdot}\widehat{\mathsf{h}}_{i-1,i}+\hat{P}_{\ell:i-2}^{n}\widehat{\mathsf{h}}_{i-1,i}\notag\\
        =&\left(F_{\ell,i-1}(B_i)-F_{\ell,i-1}(G_{i,i}(C_{i-1,i}))\right)
        \mathsf{h}_i+\sum_{r=\ell+1}^{i-1}\left(\hat{P}_{\ell:r}{%
                \mathsf{h}}_{r,i}+\Delta_r%
        \hat{P}_{\ell,\cdot}\widehat{\mathsf{h}}_{r,i}\right)+\hat{P}_{\ell:\ell}^{n}\widehat{\mathsf{h}}_{\ell+1,i}.
\end{align*}
Note that each of the terms in the above sum can now be analyzed for their order. In fact, as in the one dimensional case, one has that 
either $ \left(F_{\ell,i-1}(B_i)-F_{\ell,i-1}(G_{i,i}(C_{i-1,i}))\right)
\mathsf{h}_i $, $        \Delta_r\hat{P}_{\ell,\cdot}\widehat{\mathsf{h}}_{r,i}  $ or $\hat{P}_{\ell:r}{\mathsf{h}}_{r,i}  $ have  $\mathcal{F}_{r-1} $-conditional expectations of order $O^E_{r-1}(1).  $ 

In these estimates, it is important the fact that $ C_{i-1,i} $ in $ G_{r+1,i}(C_{i-1,i}) $ is a multiple of $ \mathbf{e}^1 (\mathbf{e}^1)^{\top} $ as this means that the constant terms will vanish when calculating either $   \Delta_r\hat{P}_{\ell,\cdot}\widehat{\mathsf{h}}_{r,i}  $ or $\hat{P}_{\ell:r}{\mathsf{h}}_{r,i}  $. In fact, due to {\large \textcircled{\small 1}}
in Step I, we have using that $ \xi_r(\mathbf{e}^1(\mathbf{e}^1)^\top\mathsf{h}_{r,i})=0$:
\begin{align*}
        \hat{P}_{\ell:r}{\mathsf{h}}_{r,i} =&\nu(\mathbf{D}_if_i(\hat{X}_r),X_r)F_{\ell,r}(
        G_{r+1,i}(\mathbf{e}^1(\mathbf{e}^1)^\top))\mathsf{h}_{r,i}\\
        F_{\ell,r}(
        \mathbf{e}^1(\mathbf{e}^1)^\top)\mathsf{h}_{r,i}=&   F_{\ell,r-1}\left(\bar{A}^1_{r}(\lambda_r
        )(\mathbf{e}^1(\mathbf{e}^1)^\top)\right)
        \mathsf{h}_{r,i}=F_{\ell,r-1}(\mu^3_{r-1})\mathsf{h}_{r,i}.
\end{align*}
Here, $ \mu^3_{r-1}\in \bar{C}^0_{1,r-1}(\mathtt{X}^1_{r-1}\mathsf{h}_r) $. Therefore we conclude that
\begin{align*}
        \tilde{\mathbb{E}}_{r-1}\left[\bar{A}^1_{r}(\lambda_r
        )(G_{r,i}(C_{r-1,i})){\mathsf{h}}_{r,i}\bar{M}_{r-1:i}\right]        =O^E_{r-1}(1).
\end{align*}
From here, one obtains that $ \tilde{\mathbb{E}}_{r-1}\left[\hat{P}_{\ell:r}{\mathsf{h}}_{r,i}\right ]= { O^E_{r-1}(1)  }$. The other term, $        \Delta_r\hat{P}_{\ell,\cdot}\widehat{\mathsf{h}}_{r,k}  $ is easier to deal with as the constant terms vanish immediately due to $ \widehat{\mathsf{h}}_{r,k}  $ and the fact that we are considering the difference $ \Delta_r\hat{P}_{\ell,\cdot} $.

As in the one dimensional case, the sum in $ i $ is dealt with using the distribution of hitting the boundary for the first time after $ t_r $ at a time between $ t_{i-1}$ and $t_i$.

Finally the last term in \eqref{eq:2der}, $ {\tilde{\mathbb{E}}}_{{i-1}}\left[{\mathbf{D}}_i{f}_{i}(\widehat{X}_i)\mathbf{e}^1(\mathbf{e}^1)^\top\mu^1_i\right] $ is analyzed with similar but simpler arguments.

First, note that $ g^1_i(\mathtt{X}_{i-1}^1)^{-1}\mu_i^{1}=\nu_i^1\in \bar{P}^1_{2,i}(Z_{i}^{\wedge} ,\Delta,\mathtt{X}_{i-1}^1 )
$ and $ \tilde{\mathbb{E}}_{i-1}\left[\left|\mu_i^{1}\right|\bar{m}_i\right]\leq C$. Therefore the only term that needs to be discussed is the term which is proportional to $  \mathtt{X}_{i-1}^1 g^1_i(\mathtt{X}_{i-1}^1)$.
This term, is analyzed as in the above path decomposition method.

{\it Proof of Proposition \ref{prop:30}, Part II: The boundedness of the second term in \eqref{eq:87a}.}

As in the one dimensional case, we also have to discuss the derivative of the second term in \eqref{fprimemH} and which is given also in \eqref{eq:DD}.

Most estimates follow due to the fact that $ f_i $ still represents the conditional expectation of a killed process and the previous arguments presented so far. The only point we want to stress  is the treatment of the constant term in the expression 
\begin{align*}
        \mathbf{e}^1    \left(\left({\mathbf{D}}_if_i {\mathbf{D}}_{i-1}X_i-(\mathbf{e}^1)^\top
        {\mathbf{D}}_if_i\cdot \mathbf{e}^1\right)    \frac{b^1_{i-1}
        }{a^{11}_{i-1}} +\mu_i^1   \right)
        \mathsf{h}_i.
\end{align*}
Here $ \mu_i^1\in \bar{P}^1_{1,i}(\Delta,Z_i,\mathtt{X}^1_{i-1}) $.
As in previous arguments we decompose $ {\mathbf{D}}_if_i=\sum_{\ell=1}^d {\mathbf{D}}_if_i\cdot\mathbf{e}^\ell(\mathbf{e}^\ell)^\top$. In the case $ \ell\neq 1 $, one can repeat arguments as in Step II of the previous proof. Therefore it is enough to consider:
\begin{align*}
        \mathbf{e}^1({\mathbf{D}}_if_i\cdot \mathbf{e}^1)       (\mathbf{e}^1)^\top        \left({\mathbf{D}}_{i-1}X_i -I\right) \frac{b^1_{i-1}
        }{a^{11}_{i-1}}       
        \mathsf{h}_i.
\end{align*}
As $ \mathbf{D}_{i-1}X_i=I+\mathbf{D}_{i-1}\sigma^k_{i-1}Z_i^k$, the above is a linear function of $ \Delta $ and $ Z_i $.
This finishes the proof of Proposition \ref{prop:30}.\hfill$ \Box $

We complete the section with the proofs of Lemma \ref{lem:43} and Corollary \ref{cor:43}. 
In order to give the proof of Lemma \ref{lem:43}, we consider first the case when a boundary condition similar to the one assumed in Lemma \ref{lem:2.1H} is satisfied.
\begin{lemma}[{\bf  The push forward formula for second derivatives}]
        \label{lem:22aH}
        Suppose that $ f_i\in C^2_b(\mathbb{R}^d, \mathbb{R}) $ satisfies $ {\mathbf{D}}_i f_i\Big|_{X_i^1=L}=0 $ for all $ x\in\mathbb{R}^d $. Then we have 
        \begin{align}
                {\mathbf{D}}^\top_{i-1}{\mathbb{E}}_{{i-1}}\left[{\mathbf{D}}_i{f}_{i}\lambda_i\bar{m}_{i}\right]={\mathbb{E}}_{{i-1}}\left[\left(\bar{A}^1_i(\lambda_i)(\mathbf{D}^2_if_i)+\bar{A}^0_i(\lambda_i,{\mathbf{D}}_if_i)\right)\bar{m}_i\right].
                \label{eq:2der0}
        \end{align}
  
        
        
         In the special case that $ \lambda _i=\lambda (X_{i-1}) $ with $ \lambda , f_i\in C^1_b(\mathbb{R}^d, \mathbb{R}) $ satisfying that $ f_i\Big|_{X_i^1=L}=0 $ the following derivative formula (with a natural extension of the notation for $ \bar{A}^1_i $ as the argument ${\mathbf{D}}_if_i  $ is a row vector and not a square matrix): 
        \begin{align}
                2{\mathbf{D}}_{i-1}{\mathbb{E}}_{{i-1}}\left[{f}_{i}\lambda_i\mathsf{h}_{i}1_{({X}^1_i>L)}\right]={\mathbb{E}}_{{i-1}}\left[\left(\bar{A}^1_i(\lambda_i\mathsf{h}_{i})({\mathbf{D}}_if_i ) +\bar{A}^0_i(\lambda_i\mathsf{h}_{i},f_i)^\top\right)\bar{m}_i\right].
                \label{eq:kt}
        \end{align}
\end{lemma}
The proof is a straightforward application of derivation rules and a consequence of Lemma \ref{lem:2H}. 

Next, we will prove that the assumption that $ \sigma  (x)$ is lower triangular can be assumed without loss of generality. 

\begin{lemma}
        \label{lem:38}
        Suppose that $ \sigma $ and $ \check{\sigma} $ are two diffusion coefficients satisfying Hypothesis \ref{hyp:12} such that $ a=\sigma \sigma^\top=\check{\sigma}\check{\sigma}^\top$. Then for $ j=0,1,2 $, we have for any test function $ f\in C^2(\mathbb{R}^d\times\mathbb{R}^d,\mathbb{R}) $ such that the expectations below are finite 
        \begin{align}
        	\label{eq:both}
                \mathbf{D}_{i-1}^j\tilde{\mathbb{E}}_{i-1}\left[f(X_{i-1}X_i)\bar{m}_i\right]=\mathbf{D}_{i-1}^j\tilde{\mathbb{E}}_{i-1}\left[f(X_{i-1},\check{X}_i)\check{m}_i\right].
        \end{align}
        Here, $ \check{X}_i=X_{i-1}+\check{\sigma}_{i-1}Z_i $ and $ \check{m}_i=1_{(\check{X}_i^1>L)} (1+\check{\mathsf{h}}_i) $ with $ \check{\mathsf{h}}_i=1_{(U_i\leq \check{p}_i)} $ and $ \check{p}_i=\exp\left(-2\frac{\check{X}^1_iX^1_{i-1}}{a^{11}_{i-1}\Delta}\right) $.
\end{lemma}
We will use this result by applying 
the Cholesky decomposition to the covariance matrix $ a$. We also remark that due to the expliciteness of the proof of the Cholesky decomposition we have that the matrix $ \check{\sigma} $ has elements which can be written using products of elements of $ \sigma  $ divided by the elements of the diagonal of $ a $. From here, one has that $ \check{\sigma}$ is lower triangular and that it satisfies Hypothesis \ref{hyp:12}. Note that according to the above Lemma, although the elements of the push forward formulas on both sides of \eqref{eq:both} may be different the final result is the same and therefore it is enough to prove that second derivatives will be bounded under this additional hypothesis. 

In what follows in order not to burden the notation, we use the same notation as before (i.e. $ \sigma  $, $ \bar{m}_i $ and $ p_i $ instead of $ \check{\sigma} $, $ \check{m}_i $ and $ \check{p}_i $) and assume that $ \sigma $ is lower triangular.

In this way, assuming without loss of generality that $ \sigma  $ is lower triangular, we have that the noise $ Z^1 $ is the only noise driving $ X^1 $ which is the component where reflection is applied.

With the above reduction, we see that the first component $ X^1 $ is a reflected process which is driven by a single noise $ Z^1 $ and therefore we can apply to it the argument used in the 1D case as stated in \cite{CK}, Section 6.5, Lemma 31. Before doing this, we first expand the Girsanov change of measure.

\begin{proof}[Proof of Lemma \ref{lem:43}]

In the 1D case in order to deal with the differentiation of the change of measure, we appeal to the Taylor expansion of the Girsanov change of measure. This is also the case here. That is,
\begin{align*}
        \tilde{\mathbb{E}}_{i-1}
        \left[e^{\kappa_i}-1-\kappa_i-\frac{\kappa_i^2}{2}\bar{m}_{i-1}\right]=O_{i-1}^E(\sqrt{\Delta}).
\end{align*}
From the above, one sees that in the formula for $ \lambda_i $ , we have already taken into account the above expansion. 

Now, the proof uses the following elements: 1. The differentiation formula in Lemma \ref{lem:22aH}. 2. The reduction using the Cholesky decomposition in Lemma \ref{lem:38}. 

 With these ingredients, the application of the formulas in Lemma \ref{lem:22aH} have to be done using  $ \lambda_i=e_i(1+\kappa_i+\frac{\kappa_i^2}{2}) $.

As we have remarked previously, only the first component $ X^1 $ will be subject to reflection. 
With this in mind, we will consider the analysis of the second derivative  in cases. First, we will consider the derivatives of ${\mathbf{D}}_\ell f_\ell\cdot \mathbf{e}^1 $ which will require some ``compensation'' argument  with the value of the derivative of the test function at the boundary as the one used in Lemma 31 in \cite{CK}.  Now, we define this compensation.

In fact, let for $ \lambda_i=e_i(1+\kappa_i+\frac{1}{2}\kappa_i^2)$:
\begin{align}
        \mathsf{g}_{i-1}\equiv \mathsf{g}(X_{i-1}):=&{\tilde{\mathbb{E}}}_{i-1}\left[{\mathbf{D}}_i{f}_{i}\lambda_i\bar{m}_{i}\right]-{\tilde{\mathbb{E}}}_{i-1}\left[{\mathbf{D}}_i{f}_{i}(\widehat{X}_i)\cdot\mathbf{e}^1(\mathbf{e}^1)^\top
        \bar{m}_{i}\right]\notag\\
        =&{\tilde{\mathbb{E}}}_{i-1}\left[{\mathbf{D}}_i{f}_{i}\lambda_i\bar{m}_{i}\right]-{\tilde{\mathbb{E}}}_{i-1}\left[{\mathbf{D}}_i{f}_{i}(\widehat{X}_i)\cdot\mathbf{e}^1
        (\mathbf{e}^1)^\top
        \right].\label{eq:56}
\end{align}

Remark that  as announced after Lemma \ref{lem:38}, $ \bar{m}_i $ depends only on the noise embedded in the increment $ X_i^1-X_{i-1}^1 $ which is independent from the noise in $ {\mathbf{D}}_i{f}_{i}(\widehat{X}_i)\cdot\mathbf{e}^1(\mathbf{e}^1)^\top $ from which \eqref{eq:56} follows using  that $ {\tilde{\mathbb{E}}}_{i-1,x}\left[\bar{m}_{i}\right]=1 $, $ x^1\geq L $.
The above choice of $ \mathtt{g}_{i-1} $ will give us an alternative way to obtain an expression for second derivatives without the main hypothesis of Lemma \eqref{lem:22aH} which is suitable for the proof of boundedness.


        First, note that as with Lemma \ref{lem:22aH} most of the proof is a straightforward application of derivation rules and a consequence of Lemma \ref{lem:2H}.
        In that sense, many of the terms that appear in $ \bar{A}^0_i $ of Lemma \ref{lem:22aH} are repeated in the present result in $ \tilde{A}^0_i $. The first  difference is that we consider the measure $ \tilde{\mathbb{P}} $ and therefore the underlying process $ X $ has no drift.
        The second and most important is that the hypothesis $ \mathbf{D}_if_i\Big|_{X_i^1=L}=0$ is not assumed here and therefore the derivative of $1_{(X_i^1>L)}  $ will contribute to the final result.

        With this in mind, instead of applying derivative rules directly, we will use the decomposition:
        \begin{align*}
                {\mathbf{D}}^\top_{i-1}{\tilde{\mathbb{E}}}_{{i-1}}\left[{\mathbf{D}}_i{f}_{i}\lambda_i\bar{m}_{i}\right]=&{\mathbf{D}}_{i-1}^\top\left({\tilde{\mathbb{E}}}_{i-1}\left[\left({\mathbf{D}}_i{f}_{i}\lambda_i-{\mathbf{D}}_i{f}_{i}(\widehat{X}_i)\cdot\mathbf{e}^1
                (\mathbf{e}^1)^\top\right)\bar{m}_i
                \right]\right)\\
                &
                +{\mathbf{D}}^\top_{i-1}{\tilde{\mathbb{E}}}_{i-1}\left[{\mathbf{D}}_i{f}_{i}(\widehat{X}_i)\cdot\mathbf{e}^1
                (\mathbf{e}^1)^\top
                \right]=:B^1_{i-1}+B^2_{i-1}.
        \end{align*}

        The next goal in the argument is to prove that $ {\mathbf{D}}_i{f}_{i}\lambda_i-{\mathbf{D}}_i{f}_{i}(\widehat{X}_i)\cdot\mathbf{e}^1
        (\mathbf{e}^1)^\top$ appearing in the first term, $ B^1_{i-1} $
        satisfies a condition equivalent to $ \mathbf{D}_if_i\Big|_{X_i^1=L}=0$ as assumed in Lemma \ref{lem:22aH}.
        In fact, this condition is used when the derivative of $ 1_{(X_i^1>L)} $ appears in the computations. Therefore, our goal is to prove that when the calculation of $ B^1_{i-1}$ involves the derivative of $ 1_{(X_i^1>L)} $ then these terms either become zero or they become terms which are gathered in $ \mu_i^1 $ and which satisfy the properties in the statement.

        
        Most of the calculations are direct applications of Lemma \ref{lem:2H}, so we will concentrate only on all the terms that are multiples of  the derivative of $ 1_{(X_i^1>L)}$ which will be collected in $ \mu_i^1 $. 
        
        The derivative of $ 1_{(X_i^1>L)}$ which will be a multiple of the Dirac delta distribution $ \delta_L(X_i^1) $ can only be obtained from the derivatives $ \mathbf{D}_{i-1}1_{(X_i^1> L)} $  and $ {D}^{1}_i1_{(X_i^1> L)} $ (which appears on the second term on the right of \eqref{eq:2f} for $ k=1 $). Here, recall that $ {D}^{1}_i$,  denotes the derivatives with respect to the noise $ Z^1_i $. 
        
        The first derivative,  $ \mathbf{D}_{i-1}1_{(X_i^1> L)} $, appears from applying directly the differentiation operator $ \mathbf{D}_{i-1}$ and the second is due to the differentiation formula in Lemma \ref{lem:2H}. Either way, this leads to a Dirac delta distribution. In Lemma \ref{lem:22aH}, the fact that the boundary condition $ {\mathbf{D}}_i f_i\Big|_{X_i^1=L}=0$ is assumed implies that these terms are zero. With this in mind, we divide the argument in cases. For the moment, we assume that $ \kappa_i=0 $. The cases are:
        \begin{enumerate}
                \item[Case A.1:] When we exchange the derivative $ \mathbf{D}_{i-1} $ and the expectation $ \tilde{\mathbb{E}}_{i-1} $ in $ B^1_{i-1} $, we will have to differentiate all the terms within the expectation. One of them is $ 1_{(X_i^1> L)}  $
which appears in $ \bar{m}_i $.
\item[Case A.2:]    Similarly, when we consider the derivative $ \mathbf{D}_{i-1} p_i $ in $ B^1_{i-1} $, we have to use formula \eqref{eq:2f} which will  require the derivative $ D^11_{(X_i^1> L)}  $  .

\item[Case B:] When considering the derivative of  $ B_{i-1}^2 $ the terms within the expectation do not involve $ 1_{(X_i^1> L)}  $.
 \end{enumerate}

        {\it Case A.1: When $ \delta_L(X_i^1) $ appears due to direct differentiation $ \mathbf{D}_{i-1}1_{(X_i^1> L)} $ in $ B^1_{i-1} $.} We will prove that 
        \begin{align*}
               {\mathbf{D}}_{i-1}^{\top}1_{(X_i^1>L)} \left({\mathbf{D}}_i{f}_{i}e_i-{\mathbf{D}}_i{f}_{i}(\widehat{X}_i)\cdot\mathbf{e}^1
        (\mathbf{e}^1)^\top\right)(1+\mathsf{h}_i)={\mathbf{D}}_i{f}_{i}(\widehat{X}_i)\cdot \mathbf{e}^1\mathtt{X}^1_{i-1}\delta_{L}(X_i^1)\mu^2_{i-1}.
        \end{align*}
Here, $ \mu^2_{i-1}\in P^1_{0,i-1} $ is matrix valued and of constant order.
        The analysis will be carried out in two parts using the following decomposition $I=\mathbf{e}^1(\mathbf{e}^1)^\top+\sum_{\ell=2}^d\mathbf{e}^\ell(\mathbf{e}^\ell)^\top$. First, using the fact that  $ X_i^1= L$, we have        
        \begin{align}
                \left({\mathbf{D}}_i{f}_{i}\cdot \mathbf{e}^1(\mathbf{e}^1)^\top e_i-{\mathbf{D}}_i{f}_{i}(\widehat{X}_i)\cdot \mathbf{e}^1(\mathbf{e}^1)^\top \right)        
                \delta_L(X^1_i)&={\mathbf{D}}_i{f}_{i}(\widehat{X}_i)\cdot \mathbf{e}^1\left((\mathbf{e}^1)^\top e_i-(\mathbf{e}^1)^\top\right)\delta_L(X^1_i).
                \label{eq:59}
        \end{align} 
Furthermore, using that $ \mathsf{h}_i=1 $ and $ (\mathbf{e}^1)^\top\pi_{i-1}=0 $, one obtains 
\begin{align*}
        \left((\mathbf{e}^1)^\top e_i-(\mathbf{e}^1)^\top\right)\delta_L(X^1_i) =&(\mathbf{e}^1)^\top \sigma^{k}_{i-1}{\mathbf{D}}_{i-1}\left(\frac{\sigma^{1k}_{i-1}}{a_{i-1}^{11}}\right )\mathtt{X}^1_{i-1}\delta_L(X^1_i).
\end{align*}
The conditional expectation $ \mathbb{E}\left[(\mathbf{e}^1)^\top \sigma^{k}_{i-1}{\mathbf{D}}_{i-1}\left(\frac{\sigma^{1k}_{i-1}}{a_{i-1}^{11}}\right )\mathtt{X}^1_{i-1}\delta_L(X^1_i)\big /\mathcal{F}_{i-1},\widehat{X}_i\right] $ is one of the terms in $ \mu_i^1 $.

        So far, we have considered only $ {\mathbf{D}}_i{f}_{i}\cdot \mathbf{e}^1(\mathbf{e}^1)^\top e_i $. Therefore, we still need to consider $ {\mathbf{D}}_i{f}_{i}\cdot \mathbf{e}^\ell(\mathbf{e}^\ell)^\top e_i $, $ \ell=2,...,d $.
        Next, we consider the remaining terms 
        \begin{align*}
                \sum_{\ell=2}^d{\mathbf{D}}_i{f}_{i}\cdot \mathbf{e}^\ell(\mathbf{e}^\ell)^\top e_i \delta_L(X^1_i)=
                \sum_{\ell=2}^d{\mathbf{D}}_i{f}_{i}(L,\hat{X}_i)\cdot \mathbf{e}^\ell(\mathbf{e}^\ell)^\top e_i\delta_L(X^1_i)=0.
        \end{align*}
        The above becomes zero because $ \mathtt{X}^1_i={X}^1_{i}-L=0 $ implies that $ \mathsf{h}_{i+1}=1 $ and therefore $ e_{i+1}\cdot\mathbf{e}^\ell=0 $ for $ \ell\neq 1 $. These properties are used to obtain 
        \begin{align*}
                {\mathbf{D}}_i{f}_{i}(L,\hat{X}_i)\cdot \mathbf{e}^\ell \delta_L(X^1_i)=\tilde{\mathbb{E}}_{i}\left[f_{i+1}e_{i+1}\cdot \mathbf{e}^\ell\bar{m}_{i+1}\right]\delta_L(X^1_i)=0.
        \end{align*}
        
        {\it Case A.2: When $ \delta_L(X_i^1) $ appears due to the use of formula in Lemma \ref{lem:2H} for $ B^1_{i-1} $:}
        We have to differentiate the random variable $ \mathsf{h}_i $ and for that we need to use the formula in Lemma \ref{lem:2H} which will include $  {D}^{1}1_{(X_i^1> L)}  $.  
        First, we have that for some explicit random variables $ \eta_i^j\in P_{1,i}^1(Z_i,\Delta) $, $ j=1,...,4 $
        \begin{align*}
                \left({\mathbf{D}}_i{f}_{i}e_i-{\mathbf{D}}_i{f}_{i}(\widehat{X}_i)\cdot\mathbf{e}^1
                (\mathbf{e}^1)^\top\right)1_{(X_i^1>L)}(1+\mathsf{h}_i)=&{\mathbf{D}}_i{f}_{i}\left(\eta_i^11_{(X_i^1>L)}+\eta_i^2\mathsf{h}_i1_{(X_i^1>L)}\right)\\&+
                {\mathbf{D}}_i{f}_{i}(\widehat{X}_i)\cdot\mathbf{e}^1\left(\eta_i^31_{(X_i^1>L)}+\eta_i^4\mathsf{h}_i1_{(X_i^1>L)}\right).
        \end{align*}
       For the first part of the above expression, using that $ \mathbf{e}^1\pi_{i-1}=0 $,  one gets 
        \begin{align}
                \notag
                &{\mathbf{D}}_i{f}_{i}e_i-{\mathbf{D}}_i{f}_{i}(\widehat{X}_i)\cdot\mathbf{e}^1(\mathbf{e}^1)^\top\\=    &\left(\mathbf{D}_if_i-{\mathbf{D}}_i{f}_{i}(\widehat{X}_i)\right)       e_i+
                {\mathbf{D}}_i{f}_{i}(\widehat{X}_i)\cdot\mathbf{e}^1(\mathbf{e}^1)^\top
                \left(\bar{b}_{i-1}\Delta +\mathbf{D}\sigma_{i-1}^kZ_i^k\right )(1-\mathsf{h}_i)
                \label{eq:60}\\
                &+{\mathbf{D}}_i{f}_{i}(\widehat{X}_i)\cdot \mathbf{e}^1
                \sigma^{1k}_{i-1}{\mathbf{D}}_{i-1}\left(\frac{\sigma^{1k}_{i-1}}{a_{i-1}^{11}}\right )\mathtt{X}^1_{i-1}\mathsf{h}_i\label{eq:61}+\sum_{\ell=2}^d{\mathbf{D}}_i{f}_{i}(\widehat{X}_i)\cdot \mathbf{e}^\ell(\mathbf{e}^\ell)^\top e_i.
        \end{align}
        
        Now we multiply the above expression by $ \bar{m}_i=1_{(\mathtt{X}_i^1>0)}(1+\mathsf{h}_i) $ and differentiate the terms containing $ \mathsf{h}_i $, according to Lemma \ref{lem:2H}. This will lead to the derivative $ {D}^{1}1_{(\mathtt{X}_i^1> 0)}   $ besides other terms which already appear in $ \bar{A}^1_i(\lambda_i)(\mathbf{D}_if_i) $, $ \tilde{A}^0_i(\lambda_i,\mathbf{D}_if_i) $ and  in the previous step which became part of $ \mu_i^1 $ appearing in the statement of the result and which are bounded as we will see below.

         We will study each term in  \eqref{eq:60}-\eqref{eq:61} which contains the term $ {D}^{1}_i1_{(X_i^1> L)}   $ when differentiation of $ \mathsf{h}_i $ is carried out.
        
        In fact, $ {D}^{1}_i1_{(X_i^1> L)} =\delta_L(X^1_i)D^{1}_iX_i^1 = \delta_L(X^1_i){\sigma^{11}_{i-1}}$ implies immediately that the first term of \eqref{eq:60} is  zero: That is, $ \left(\mathbf{D}_if_i-{\mathbf{D}}_i{f}_{i}(\widehat{X}_i)\right)\delta_L(X^1_i) =0$. Furthermore, we also have that
        $ \mathtt{X}^1_i=X^1_i-L=0 $ and $ \mathsf{h}_{i+1}=\mathsf{h}_i=1 $ which will be used in what follows.
        
        For the second term in \eqref{eq:60} and \eqref{eq:61}, note that we multiply it by $ \bar{m}_i =1_{(\mathtt{X}_i^1>0)}(1+\mathsf{h}_i)$ and then we only consider the terms which are multiplied by $ \mathsf{h}_i $ because these are the terms that will require the use of Lemma \ref{lem:2H} and therefore the derivative $ {D}^{1}_i1_{(\mathtt{X}_i^1> L)}$ will appear.

        Using $ \mathtt{X}^1_i=X^1_i-L=0 $, and \eqref{eq:ort1}, the coefficient of $ \mathsf{h}_i $ in the second term in \eqref{eq:60} added to the respective ones in \eqref{eq:61} give:
        \begin{align*}
                &-      {\mathbf{D}}_i{f}_{i}(\widehat{X}_i)\cdot\mathbf{e}^1(\mathbf{e}^1)^\top
                \left(\bar{b}_{i-1}\Delta +\mathbf{D}\sigma_{i-1}^kZ_i^k\right )+{\mathbf{D}}_i{f}_{i}(\widehat{X}_i)\cdot \mathbf{e}^1
                \sigma^{1k}_{i-1}{\mathbf{D}}_{i-1}\left(\frac{\sigma^{1k}_{i-1}}{a_{i-1}^{11}}\right )\mathtt{X}^1_{i-1}\\
                =&-     {\mathbf{D}}_i{f}_{i}(\widehat{X}_i)\cdot\mathbf{e}^1(\mathbf{e}^1)^\top\left(\bar{b}_{i-1}\Delta+\sum_{k=1}^d\sum_{\ell=2}^d\mathbf{D}\sigma_{i-1}^k(\sigma_{i-1}^{-1})^{k\ell}\left(\sigma^\wedge_{i-1}Z_i^\wedge\right)^\ell
                \right )
        \end{align*}
        Here, we have used the fact that $ \sigma $ is lower triangular and the natural extension of the wedge notation to matrices so that $ \sigma_{i-1}^\wedge $ is a $ (d-1)\times (d-1) $ matrix after deleting the first row and the first column.
        
        Then using the formula for the derivative in Lemma \ref{lem:2H} gives  as a coefficient of $ \delta_L(X^1_i)  $
        \begin{align}
                {\mathbf{D}}_i{f}_{i}(\widehat{X}_i)\mathbf{e}^1(\mathbf{e}^1)^\top
                \left(\bar{b}_{i-1}\Delta +
                \sum_{k=1}^d\sum_{\ell=2}^d\mathbf{D}\sigma_{i-1}^k(\sigma_{i-1}^{-1})^{k\ell}\left(\sigma^\wedge_{i-1}Z_i^\wedge\right)^\ell
                \right )\delta_L(X^1_i) D^k_iX_i^1
                {\mathbf{D}}_{i-1}\left (\frac{      {\sigma}^{1k} _{i-1}
                        \mathtt{X}_{i-1}^1
                }{a^{11}_{i-1}}\right )\mathsf{h}_i.
                \label{eq:del}
        \end{align} 
        
        Taking the conditional expectation $ \mathbb{E}_{i-1} $ in the first term which contains $ \bar{b}_{i-1}  $ gives
        \begin{align*}
                {\mathbf{D}}_i{f}_{i}(\widehat{X}_i)\mathbf{e}^1(\mathbf{e}^1)^\top
                \bar{b}_{i-1}\Delta g^1_i(\mathtt{X}^1_{i-1})=O_{i-1}^E(1).
        \end{align*}
        The conditional expectation with respect to $ \mathcal{F}_{i-1}\vee \sigma(Z^\wedge_i) $ of the terms containing $ Z_i^{\ell} \delta_L(X^1_i)$, $ \ell\neq 1 $ in \eqref{eq:del} are increments of discrete time martingales with the form $Z_i^{\ell} g_i^1(X^1_{i-1}) $ because $ Z_i^{\ell} $ and $ X_i^1 $ are independent. These terms are collected in the formula for $ \mu_i^1 $. Therefore the boundedness properties of $ \mu^1_i $ stated in the Lemma are satisfied.

        As for the last term in \eqref{eq:61}, we obtain as in the first step using  $ \mathtt{X}^1_i=0 $ and $ \mathsf{h}_{i+1}=1$  that 
        $  {\mathbf{D}}_i{f}_{i}(\widehat{X}_i)\cdot\mathbf{e}^\ell\delta_L(X^1_i)=0$.  
        
        Therefore, from the above arguments, we obtain that for $ \nu^5_{i-1}\in \bar{P}^1_{1,i-1}({Z}_i^{\wedge},\Delta,\mathtt{X}^1_{i-1}) $
        \begin{align*}
                \tilde{\mathbb{E}}_{i-1}\left[\left({\mathbf{D}}_i{f}_{i}e_i-{\mathbf{D}}_i{f}_{i}(\widehat{X}_i)\cdot\mathbf{e}^1
                (\mathbf{e}^1)^\top\right)D_i^{1}1_{(X_i^1>L)}\mathsf{h}_i\right] =
                {\mathbf{D}}_i{f}_{i}(\widehat{X}_i)\cdot\mathbf{e}^1\nu^5_{i-1}g_i^1(\mathtt{X}^1_{i-1}).
        \end{align*}
        
        
        The derivatives of all other terms are straightforward and carried out as in the proof of Lemma \ref{lem:22aH}.
        In conclusion, the derivative of $ \mathsf{g}_{i-1} $ will be uniformly bounded. In any case, the above proof shows how to estimate and handle the terms which are multiples of $ \delta_L(X^1_i) $ and which are finally collected into $ \mu_i^1 $
        
        The case when $ \kappa_i\neq 0 $ is carried out using the same arguments as above, Note that the new terms when $ \kappa_i\neq 0 $ do not contain constant terms and only linear and quadratic terms in $ (Z_i^\wedge,\Delta) $. The lowest order term  which appears due to the differentiation of $ 1_{(X_i^1>L)} $ is  $- (\mathbf{e}^1)^\top(\sigma^{-1})_{i-1}^{1}\cdot b_{i-1}\mathtt{X}_{i-1}^1g^1_i(\mathtt{X}_{i-1}^1) $.

        %
        %
        
        {\it Case B: The estimates for $ B^2_{i-1} $.}
        The second term in the above decomposition, $ B^2_{i-1} $, can be differentiated directly and the properties stated are easily obtained as $ \bar{m}_i $ does not appear inside the conditional expectation.
        In fact, note that the derivative of the last term in \eqref{eq:56}satisfies:
        \begin{align*}
                {\mathbf{D}}^\top_{i-1} {\tilde{\mathbb{E}}}_{i-1}\left[{\mathbf{D}}_i{f}_{i}(L,\hat{X}_i)
                (\mathbf{e}^1)^\top\right]=
                {\tilde{\mathbb{E}}}_{i-1}\left[{\mathbf{D}}^\top _{i-1}(L,\hat{X}_i)\mathbf{D}^2_i{f}_{i}(L,\hat{X}_i)
                (\mathbf{e}^1)^\top\right].
        \end{align*}
        Terms like the one above also appear when we repeat the arguments in the proof of Lemma \ref{lem:22aH} for the present case and therefore they cancel. This is the reason why they do not appear in the first term of \eqref{eq:exp1}.

        As for the last statement \eqref{eq:exp2}, it follows from the explicit calculations explained above.
\end{proof}

\subsection{Convergence of first and last hitting time}
\label{sec:8}
This section explains how to extend the results regarding the first and last hitting times from the one dimensional case (section 6.6 in \cite{CK}) to the first component of the present multidimensional situation. The notation is also borrowed from \cite{CK}. The convergence of the first and the last hitting times are required in the proof of the Theorem \ref{bigT}. 

        Recall that $X^{n}$ is the sequence of Euler approximation processes of $X$ which start 
        from $x\in  H^d_L $. The sequence $ (X^{n},\bold{B}^n) $ satisfies the evolution equations (\ref{eq:defbXunderPtilde''})+(\ref{fn2})  under $ {\mathbb{Q}}^n$ and converge in distribution to $(Y,\bold{B})$ defined in (\ref{reflected}). The same convergence holds for the continuous paths Euler approximation $ (X^{c,n},\bold{B}^n)$. As already stated,
under $ \mathbb{Q}^n $ the law of  $ X^{c,n}$, on the interval $[t_{i-1},t_i] $ is that of a process $\mathcal{X}^{n}$ reflected when it reaches the boundary
$ H_L^d $ . More precisely, the law of the first coordinate of $ X^{c,n}$, on the interval $[t_{i-1},t_i] $, is that of the one dimensional process $\mathcal{X}^{n,1}$ introduced in (\ref{calxn}). Therefore also 
the sequence $(\mathcal{X}^{n},\Lambda^{L})$ converges in
        distribution to $(Y,\bold{B})$ on any bounded interval $[0,T]$. By using the Skorohod representation
                theorem, we can assume, without loss of generality, that  $(\mathcal{X}^{n},\Lambda
                ^{L}) $ converges almost surely to $(Y,B)$ uniformly on the path space on
                any bounded interval $[0,T]$. In particular, we can look at  the Skorohod representation used in the proof of Theorem \ref{bigT}.
Consequently, we can assume that $(\mathcal{X}^{n},\Lambda^{L})$ converges almost surely to $(Y,B)$. The following definition and subsequent theorem are given in terms of the representation of the processes defined using the Skorohod representation introduced in Theorem \ref{bigT} for which the almost sure convergence holds.

Define $\tau_t$, $\tau _{t,n}$ to be the first times in $ [0,t] $ that the processes $Y$,
        respectively $\mathcal{X}^{n}$ hit the boundary $\partial H^d_L$. That is, 
        \begin{eqnarray*}
        \tau_t&:=&t\wedge\inf \left\{ s\geq 0,\ \ Y_{s}^{1}=L\right\},\ \ \ \ \tau
        _{t,n}:=t\wedge\inf \left\{ s\geq 0,~(\mathcal{X}^{n}_s)^1=L\right\},\\
        \bar{\tau}_t^n&:=& t\wedge\inf \{t_{i}\leq t;\text{ there exists }s\in[t_{i-1},t_i],\ \mathcal{X}^{n,1}_s=L \}.
        \end{eqnarray*}      
   Similarly define $\rho_t$, $\rho_{t,n}$ to be the last times before $t$ the
        processes $Y$, respectively $\mathcal{X}^{n}$ hit the boundary $\partial H^d_L$. i.e., 
        \begin{eqnarray*}
        \rho_t&:=&\sup \left\{ s\le t,\ \ Y_{s}^{1}=L \right\}, 
\\
        \rho_{t,n}&:=&\sup \left\{ s\le t,\ \ \mathcal{X}^{n,x}_{s}=L \right\},\\ 
         \bar{\rho}_t^{n}&:=&\sup\{t_{i}\leq t;\text{ there exists }s\in[t_{i-1},t_i],\ \mathcal{X}^{n,1}_s=L \} .
         \end{eqnarray*} 
        
We have the following: 

   \begin{proposition}
                \label{rdthe} For any fixed $ t\in [0,T] $,  the random times $(\bar{\tau}_t^n,\bar{\rho}_t^{n})$ converge, almost surely, to  $(\tau_{t},\rho_t)$.
        \end{proposition}
     
 \begin{proof}
As in the one dimensional case, we may assume that $ \rho_t\geq 0 $ and that     $ \inf \left\{ s\geq 0,\ \ Y_{s}^{1}=L\right\}<\infty  $.   
           
Observe that $|\bar{\tau}_t^n-\tau_{t,n}|\le 1/n$ and $|\bar{\rho}_t^{n}-\rho_{t,n}|\le 1/n$
so it is enough to prove that $(\tau_{t,n},\rho_{t,n})$ converges almost surely to  $(\tau_{t},\rho_t)$. The proof of the convergence of $ \tau_{t,n} $ is the same as in the one dimensional case as the dimension is not used in the proof. 
In particular, define $u $ and $v^{n} $ as the martingale parts of the processes $Y $ and 
        $\mathcal{X}^{n} $ respectively. Note that we have 
        \begin{equation*}
        \tau_t=t\wedge\inf \left\{ s\geq 0,\ \ (u_{s})^{1}=L\right\},\ \ \ \ \tau
        _{t,n}=t\wedge\inf \left\{ s\geq 0, (v^{n}_s)^1=L\right\},
        \end{equation*}
        as the processes $B$ and $\Lambda^L$ are null before the processes hit the
        boundary. As in the 1D case, $(u)^1$ and $(\upsilon^{n})^1$ are time-changed Brownian motions
        (starting from $x^{1}$) and they ``leave'' the domain $[L,\infty)$ as soon
        as they hit the boundary. That is, $ L $
is a regular boundary point. The proof of the convergence of $ \tau_{t,n} $ is now identical with the corresponding result in 1D.

For the proof of $ \rho_{t,n} $ a similar argument with the 1D is required.  First, using the Dambis-Dubins-Schwarz theorem (see, e.g.  Theorem 4.6, Chapter 3 in \cite{KS})we have that the sum of stochastic integrals associated to $ (\mathcal{X}^{n})^1$ (as opposed to just having a single integral as in the 1D case) can be understood as a time changed Wiener process (starting from $x^{1}$). One uses this Wiener process to redefine the approximation process with the same formulas but using this newly defined Wiener process. As in the one dimensional case, the sequence of approximative time changes converge uniformly due to the uniform ellipticity hypothesis. 
That is, define
\begin{align*}
        \alpha(t):=&\int_0^ta^{11}(\mathcal{X}_s^{n})ds\\
        \widehat{W}_s:=&\int_0^{\alpha^{-1}(t)}\sigma^{1k}(\mathcal{X}_s^{n})dW^k_s.
\end{align*}
Then $ \widehat{W}$ is a Wiener process under its own filtration. Then
$ \mathcal{X}_t^{n} =x+\widehat{W}^{n}_{\alpha(t)} $.  The rest of the proof follows as in the 1D case.

 \end{proof}

\end{document}